\documentclass[11pt]{article}

\title{The de Rham Homotopy Theory and Differential Graded Category}

\date{Department of Mathematics, Faculty of Science, 
Kyoto University.}

\author{Syunji Moriya\footnote{Corresponding address: Department of Mathematics, Faculty of Science, 
Kyoto University, 
Kyoto, 606-8502, Japan.          \   
E-mail adress: \texttt{moriyasy@math.kyoto-u.ac.jp} \ 
Telephone number: 81-075-753-3700 \ 
FAX number: 81-075-753-3711}}
\usepackage{amsthm}
\usepackage{amscd}
\usepackage{amsmath,amsthm,amssymb,amscd,mathrsfs}
\usepackage[arrow,matrix]{xy}
\input xy
\xyoption{all}

\usepackage{enumerate}
\theoremstyle{plain}
\newtheorem{defi}{Definition}[subsection]

\newtheorem{prop}[defi]{Proposition}

\newtheorem{rem}[defi]{Remark}

\newtheorem{lem}[defi]{Lemma}

\newtheorem{sublem}[defi]{Sub-lemma}

\newtheorem{thm}[defi]{Theorem}

\newtheorem{cor}[defi]{Corollary}

\newtheorem{exa}[defi]{Example}

\newcommand{\mf}[1]{{\mathfrak{#1}}}
\newcommand{\mb}[1]{{\mathbf{#1}}}
\newcommand{\bb}[1]{{\mathbb{#1}}}
\newcommand{\mca}[1]{{\mathcal{#1}}}
\newcommand{\ms}[1]{{\mathsf{#1}}}
\newcommand{\msc}[1]{\mathscr{#1}}

\newcommand{\Z}{\bb{Z}}
\newcommand{\Q}{\bb{Q}}
\newcommand{\R}{\bb{R}}

\renewcommand{\L}{\bb{L}}
\newcommand{\gam}{\Gamma}
\newcommand{\lam}{\lambda}

\newcommand{\set}{\ms{Set}}
\newcommand{\sset}{\mathrm{s}\ms{Set}}
\newcommand{\ssetp}{\mathrm{s}\ms{Set}_*}
\newcommand{\ssetpc}{\mathrm{s}\ms{Set}_*^c}
\newcommand{\ssetpgd}{\mathrm{s}\ms{Set}_*^{gd}}

\newcommand{\cset}{\ms{Set}^{\square}}

\newcommand{\nngc}{\ms{C}^{\geq 0}(k)}

\newcommand{\dgc}{\mathrm{dg}\ms{Cat}^{\geq 0}}

\newcommand{\cldgc}{\mathrm{dg}\ms{Cat}^{\mathrm{cl}}}
\newcommand{\dgclp}{\mathrm{dg}\ms{Cat}^{\mathrm{cl}}_*}
\newcommand{\dgcl}{\mathrm{dg}\ms{Cat}^{\mathrm{cl}}}
\newcommand{\acldgc}{\mathrm{dg}\ms{Cat}^{\mathrm{cl}}_*}

\newcommand{\cldgck}{\mathrm{dg}\ms{Cat}^{\mathrm{cl}}_*}

\newcommand{\clcat}{\ms{Cat}^{\mathrm{cl}}}

\newcommand{\clcatk}{\ms{Cat}^{\mathrm{cl}}_*}

\newcommand{\tann}{\ms{Tan}}
\newcommand{\tannp}{\ms{Tan}_*}
\newcommand{\tannplus}{\ms{Tan}^{+}_*}
\newcommand{\tannplusgd}{\ms{Tan}^{+gd}_*}

\newcommand{\calg}{\mathrm{c}\ms{Alg}}

\newcommand{\dgalg}{\mathrm{dg}\ms{Alg}}

\newcommand{\redeqalg}{\mathrm{dg}\ms{Alg}^{\mathrm{red}}}

\newcommand{\credeqalg}{\mathrm{dg}\ms{Alg}^{\mathrm{red}}_0}
\newcommand{\aredeqalg}{\mathrm{dg}\ms{Alg}^{\mathrm{red}}_{0,*}}

\newcommand{\sprk}{\mathrm{s}\ms{Pr}(k)}
\newcommand{\sprkp}{\mathrm{s}\ms{Pr}(k)_*}
\newcommand{\sprq}{\mathrm{s}\ms{Pr}(\Q)}
\newcommand{\sprqp}{\mathrm{s}\ms{Pr}(\Q)_*}

\newcommand{\sprl}[1]{\mathrm{s}\ms{Pr}(#1)^{\mathrm{loc}}}

\newcommand{\sprkl}{\mathrm{s}\ms{Pr}(k)^{\mathrm{loc}}}
\newcommand{\sprklp}{\mathrm{s}\ms{Pr}(k)^{\mathrm{loc}}_*}
\newcommand{\sprkobj}{\mathrm{s}\ms{Pr}(k)^{\mathrm{obj}}}
\newcommand{\sprkobjp}{\mathrm{s}\ms{Pr}(k)^{\mathrm{obj}}_*}
\newcommand{\sht}{\ms{SHT}}
\newcommand{\shtp}{\ms{SHT}_*}
\newcommand{\Mod}{\ms{Mod}}

\newcommand{\proj}{\ms{Proj}}
\newcommand{\Alg}{\ms{Alg}}
\newcommand{\kalg}{k-\ms{Alg}}
\newcommand{\kaff}{\ms{Aff}_k}
\newcommand{\aff}{\ms{Aff}}
\newcommand{\grp}{\ms{Grp}}

\newcommand{\ho}{\ms{Ho}}

\newcommand{\ob}{\mathrm{Ob}}

\newcommand{\homo}{\mathrm{Hom}}
\newcommand{\aut}{\mathrm{Aut}}
\newcommand{\gl}{\mathrm{GL}}
\newcommand{\autbar}{\underline{\aut}}
\newcommand{\inhom}{\mf{Hom}}

\newcommand{\tdr}{\mathrm{T}_{\mathrm{PL}}}
\newcommand{\tdrk}{T_{PL}(K)}
\newcommand{\tdrl}{T_{PL}(L)}

\newcommand{\cspl}{\mathrm{C}_{\mathrm{spl}}}
\newcommand{\ccube}{\mathrm{C}_{\square}}
\newcommand{\tcube}{\mathrm{T}_{\square}}

\newcommand{\T}{\ms{T}}
\newcommand{\Tss}{\ms{T}^{\mathrm{ss}}}

\newcommand{\rep}{\operatorname{Rep}}
\newcommand{\repg}{\mathrm{Rep}(G)}
\newcommand{\reph}{\mathrm{Rep}(H)}
\newcommand{\loc}{\mathrm{Loc}}

\newcommand{\vect}{\mathrm{Vect}}
\newcommand{\coh}{\mathrm{H}}
\newcommand{\coc}{\mathrm{Z}}
\newcommand{\Hom}{\operatorname{Hom}}
\newcommand{\pibar}{\underline{\pi}}
\newcommand{\cdr}{\mathrm{C}_{\mathrm{PL}}}
\newcommand{\cone}{\operatorname{Cone}}
\newcommand{\ared}{\operatorname{A_{red}}}
\newcommand{\mred}{\mathcal{M}_{\red}}
\newcommand{\mdr}{\mathcal{M}_{\mathrm{PL}}}

\newcommand{\alg}{\mathrm{alg}}
\newcommand{\red}{\mathrm{red}}
\newcommand{\ssim}{\mathrm{ss}}

\renewcommand{\ker}{\operatorname{Ker}}
\newcommand{\coker}{\operatorname{Coker}}
\renewcommand{\hom}{\operatorname{Hom}}

\newcommand{\colim}{\operatorname{colim}}

\newcommand{\ot}{\! \otimes \!}
\newcommand{\bt}{\! \boxtimes \!}

\newcommand{\iso}{\mathrm{iso}}
\newcommand{\wcl}{\mathrm{W_{cl}}}

\newcommand{\M}{\mca{M}}
\newcommand{\uni}{\mb{1}}

\newcommand{\op}{\mathrm{op}}
\newcommand{\id}{\mathrm{id}}
\newcommand{\sch}{\mathrm{sch}}
\newcommand{\N}{\mathcal{N}}
\newcommand{\adr}{\mathrm{A}_{PL}}
\newcommand{\ring}{\mca{O}}

\newcommand{\aloc}{\msc{L}}

\newcommand{\ru}{\mathrm{R}_{\mf{u}}}
\newcommand{\gr}{\mathrm{gr}}
\begin{document}

\maketitle
\begin{abstract}
This paper is a generalization of \cite{moriya}. We develop the de Rham homotopy theory of not necessarily nilpotent spaces. We   use two algebraic objects: \textit{closed dg-categories} and \textit{equivariant dg-algebras}. We see these two  objects correspond in a certain way (Prop.\ref{thmbasic}, Thm.\ref{propcomplete1}). We prove an equivalence between the homotopy category of schematic homotopy types \cite{champs} and a homotopy category of closed dg-categories (Thm.\ref{result4}). We give a description of homotopy invariants of  spaces in terms of   minimal models (Thm.\ref{result3}). The minimal model in this context behaves much like the Sullivan's minimal model.   We also provide some examples. We prove an equivalence between   fiberwise rationalizations \cite{kan} and  closed dg-categories with subsidiary data (Thm.\ref{result5}).
 \end{abstract}
\begin{center}
\textit{Keywords: rational homotopy theory, non-simply connected space, \\
dg-category, schematic homotopy type.} 
\end{center}
\section{Introduction}
In \cite{sul}, Sullivan constructed the correspondence between the rational homotopy types of nilpotent spaces and commutative dg-algebras over rationals, which associates polynomial de Rham algebras to spaces.  In particular, he showed that  homotopy invariants of nilpotent spaces, such as rational homotopy groups, can be derived from the algebras. This Sullivan's theory is called the de Rham homotopy theory.\\
\indent In the non-nilpotent case, as a generalization of the rationalization, the fiberwise rationalization was proposed by Bousfield and Kan (\cite{kan}, see also introductions of \cite{nsrat,moriya}). For this notion, A.G$\acute{\textrm{o}}$mez-Tato, S.Halperin and D.Tanr$\acute{\textrm{e}}$ \cite{nsrat} generalized the Sullivan's result to non-nilpotent spaces. \\
\indent Recently, as another non-nilpotent generalization of the rationalization, the schematization \cite{champs} was introduced by To\"en. While the rationalization is the localization with respect to the rational homology groups, the schematization is a candidate for a localization with respect to all cohomology groups with coefficients in finite rank local systems.   In this paper, we generalize the Sullivan's result for the schematization over a field of characteristic 0. We use two algebraic objects which are generalizations of commutative dg-algebras: \textit{closed dg-categories} \cite{moriya} and \textit{equivariant commutative dg-algebras}.  These two algebraic objects  have different advantages. We establish an equivalence between the homotopy category of schematic homotopy types and a homotopy category of closed dg-categories (Thm.\ref{result4}). We give a description of homotopy invariants of  not necessarily nilpotent spaces in terms of the minimal models of  equivariant dg-algebras (Thm.\ref{result3}).\\

\indent Let $k$ be a field of characteristic 0. A closed dg-category is a $k$-linear dg-category which is equipped with a  closed tensor structure consistent with the differential graded structure (see Def.\ref{defofcldgc}). A typical example of a closed tensor structure is tensors and internal homs on the category of representations of a group. If one views a dg-algebra as a dg-category with only one object, a tensor structure on a dg-category is a natural generalization of commutativity of a dg-algebra.  We also need to consider internal homs. The existence of a model category structure on the category of small closed dg-categories was proved in \cite{moriya}. \\
\indent In \cite{moriya}, a closed dg-category $\tdr(K)$ is defined for a simplicial set (or triangulated space) $K$.  Its objects are finite rank $k$-linear local systems on $K$ and  complexes of morphisms are polynomial de Rham complexes with local coefficients.  This is an analogue of  the dg-category of  flat bundles on a manifold, which was defined by Simpson \cite[section 3]{sim}. We consider $\tdr(K)$ as a generalization  of the polynomial de Rham algebra which contains information of finite dimensional representation of the fundamental group. As we see in \cite{moriya}, when the fundamental group of $K$ is a finite group and $k=\Q$, the construction $K\mapsto \tdr(K)$ is equivalent to the fiberwise rationalization  of $K$. But in general, it is different as we consider only {\em finite rank} local systems. \\
\indent We introduce a special class of closed dg-categories which we call Tannakian dg-categories (Def.\ref{deftannakian}). It is characterized by conditions abstracted from the closed dg-category $\tdr(K)$ of connected $K$. These conditions are stated in terms of  Tannakian theory \cite{dmos} and the completeness of a dg-category \cite[section 3]{sim}. Tannakian theory concerns a duality between  affine group schemes and certain closed tensor $k$-linear categories (see Appendix \ref{proobj}).  The completeness of the dg-category means that information of exact sequences in the category of $0$-th cocycles   determines and is determined by the first cohomology groups of morphisms (see subsection \ref{complete}). Tannakian dg-categories are exactly those which correspond to schematic homotopy types.\\
\indent While closed dg-categories have good functorial and homotopical properties, they are not  suitable for computations. To cover this inconvenience,  we  use a $\pi_1$-equivariant commutative dg-algebra $\ared(K)$ defined as follows ($\pi_1=\pi_1(K)$).   
Let $\pi_1^{\red}$ be the  pro-reductive  completion of $\pi_1$ over $k$ (see subsection \ref{notation}) and $\mca{O}(\pi_1^{\red})$ be its coordinate ring. It  has two actions of $\pi_1$: the left and right translations.  With the right translation, we regard it as a local system on $K$. 
\begin{enumerate}
\item As a complex, $\ared(K)$ is the polynomial de Rham complex with coefficients in the local system $\ring(\pi_1^\red)$. 
\item The multiplication is defined from those of polynomial forms and  the coordinate ring. 
\item The action of $\pi_1$ is defined from the left translation on $\ring(\pi_1^\red)$. 
\end{enumerate}
The importance of this dg-algebra first seemed to be recognized by Deligne and it has been studied by Katzarkov, Pantev,  To\"en \cite{champs, kpt1,kpt2} and Pridham \cite{prid0,prid1,prid2}. We will prove a correspondence between Tannakian dg-categories and equivariant dg-algebras, where $\tdr(K)$ corresponds to $\ared(K)$ (see Prop.\ref{thmbasic}, Thm.\ref{propcomplete1}). The proof is not difficult but this  is very useful. For example, we cannot find  the natural representations of the fundamental group on higher homotopy groups if we  see only the dg-algebra, but these representations appear as  objects of the corresponding dg-category.  Results in the following are obtained by using this correspondence.
\subsubsection{Main results}
We shall state main results. Let $\shtp$ be the category of pointed schematic homotopy types (Def.\ref{defsht}). It is a full subcategory of the category of $\infty$-stacks over $k$, which is characterized by a certain $k$-linearity of homotopy sheaves. The schematization $(K\otimes k)^\sch$ of a simplicial set $K$ is a universal schematic homotopy type  for $K$.  Let $\tannp$ be the category of Tannakian dg-categories with a fiber functor. A fiber functor of a Tannakian dg-category $C$ is a dg-functor from $C$ to the category of finite dimensional vector spaces, which preserves closed tensor structures (Def.\ref{defacldgc}). Let $\ssetp^c$ denote the category of pointed connected simplicial sets.
\begin{thm}[Thm.\ref{thmsht}]\label{result4}
There exists an equivalence of categories $\ho(\shtp)\longrightarrow \ho(\tannp)^\op$ such that the following diagram is commutative up to natural isomorphisms:
\[
\xymatrix{\ho(\shtp)\ar[r]^{\sim}&\ho(\tannp)^\op \\
\ho(\ssetp^c)\ar[u]^{(-\otimes k)^\sch}\ar[ur]_{\tdr}.&}
\]
 Here, $\ho(-)$ denotes the corresponding homotopy category (see subsection \ref{notation}).  The equivalence is induced by a Quillen pair between larger model categories.
\end{thm}   
It is known that the schematization comes from $\ared(K)$. In fact, it is realized as the homotopy quotient of  topological realization of $\ared(K)$ by $\pi_1^\red$ (see \cite[Cor.3.3]{kpt2}). But this construction is not natural with respect to maps between simplicial sets as the construction $K\mapsto \pi_1(K)^\red$ is  not natural with respect to those which do not preserve semi-simple representations of fundamental groups (see Rem.\ref{remnotfunct}, the naturality part  of the statement of \cite[Cor.3.3]{kpt2} is wrong. the naturality folds only for the subcategory whose morphisms are those which preserve semi-simple ones). The theorem says the schematization and the construction $K\mapsto \tdr(K)$ are naturally equivalent. We prove this using the model structure on the category of closed dg-categories and results of \cite{champs}. We deduce a similar equivalence in the unpointed case  from Thm.\ref{result4}, see Cor.\ref{thmunpointedsht}. \\ 

\indent A feature of the Sullivan's theory is that the minimal model describes the rational homotopy theory of a space in a very transparent way (see \cite[P38, Analogy to topology]{sul} or  \cite[Thm.11.5]{grimor}). For a simply connected space, the indecomposable modules are  dual to the rational homotopy groups and the differentials correspond to rational k-invariants of the Postnikov tower. \\
\indent The correspondence between dg-categories and dg-algebras enables us to  translate homotopy invariants of a space into invariants of minimal algebras and we obtain an analogous description in the non-nilpotent case, as follows.     
  $\ared(K)$  has a minimal model (in the usual sense) with a semi-simple  $\pi_1$-action. Let $\mca{M}$ be such a minimal model of $\ared (K)$.  Let 
\[
V^1,V^2,\dots, V^i,\dots,\ \ \ \ V^i\subset \mca{M}^i
\] be a sequence of semisimple $\pi_1$-modules generating $\mca{M}$ freely as a commutative graded algebra. Of course, $V^i$ is isomorphic to the $i$-th indecomposable module.  As $\mca{M}$ is minimal, $d(V^i)\subset  \M^1\otimes_kV^i\oplus \M(i-1)^{i+1}$, where $d$ is the differential of $\mca{M}$ and $\M(i-1)$ is the dg-subalgebra of $\M$ generated by $\bigoplus_{j\leq i-1}M^j$. So there is a unique decomposition 
\[
d|_{V^i}=d^{\M^1\otimes_k V^i}\oplus d^{\M(i-1)}
\]
consisting of two homomorphisms of $\pi_1$-modules $d^{\M^1\otimes_k V^i}:V^i\to \M^1\otimes_kV^i$ and $d^{\M(i-1)}:V^i\to \M(i-1)^{i+1}$. 
\begin{thm}[Lem.\ref{lemghe}, Thm.\ref{thmfdhtpygr}]\label{result3}
We use the above notations. Put $\pi_i=\pi_i(K)$ for $i\geq 1$. For $i\geq 2$, we consider $\pi_i$ as a $\pi_1$-module by the canonical action. Let $n\geq 2$. Suppose $\pi_1$ is algebraically good (Def.\ref{defalggood}) and $\pi_i$ is of finite rank as an Abelian group for each $2\leq i\leq n$. Then, for each $ 2\leq i\leq n$,
\begin{enumerate}
\item the dual $(V^i)^{\vee}$ is isomorphic to  $(\pi_i\otimes _{\Z}k)^{\ssim}$,  the semisimplification  of $\pi_1$-module $\pi_i\otimes _{\Z}k$,
\item  the  map $d^{\M^1\otimes_k V^i}$ encodes the information of a presentation of $\pi_i\otimes_{\Z}k$ as  successive extensions of irreducible components of  $(\pi_i\otimes_{\Z}k)^{\ssim}$, and
\item the map $d^{\M(i-1)}$ corresponds to the k-invariant (tensored with $k$)  of the $i$-th level of the Postnikov tower  of $K$.
\end{enumerate}
 In particular, the two data $
V^i$ and $d^{\M^1\otimes V^i}$
determine the $\pi_1$-module $\pi_i\otimes_{\Z}k$ up to isomorphisms.
\end{thm}
  As examples of algebraically good groups, finitely generated free groups, finitely generated  Abelian groups, fundamental groups of Riemann surfaces are known (see \cite[subsection 4.3]{kpt2}). For other examples, see \cite[Thm.4.16]{kpt2} and Thm.\ref{thmext}. We  prove Thm.\ref{result3} by an argument similar to the proof of Hirsch Lemma \cite[Thm.11.1]{grimor}.\\
\indent   We can also recover cohomology groups of any finite rank $k$-local coefficients easily from the minimal model. For a semisimple local system or equivalently, a semi-simple representation $V$ of $\pi_1$, it is \\
$\coh^*((\M\otimes V)^{\pi_1})$, the cohomology group of the complex of invariants of $\M\otimes V$. For a general finite dimensional representation $V$, we need to twist the complex $(\M\otimes V^\ssim)^{\pi_1}$ by some Maurer-Cartan element before we take cohomology ($V^\ssim$ denotes the semi-simplification of $V$).  
  
  \subsubsection{Examples}
We shall provide a simple example deduced from Thm.\ref{result3}.
\begin{exa}[Example \ref{exaknmn}] Suppose $k$ is algebraically closed. Let $n\geq 2$ be an integer and  $M$ be a finite rank abelian group with $\Z$-action. Let $K=K(\Z,M,n):=K(\Z\ltimes  K(M,n-1),1)$. Here, $K(M,n-1)$ is an  Eilenberg Maclane space realized as simplicial abelian group with the induced action of $\Z$. Let $g\in \gl (M\otimes_ {\Z}k)$ be  the action of the generater $1$ of $\Z$ and $g =g^s+g^n$ be a Jordan decomposition, where $g^s$ is semi-simple and $g^n$ is nilpotent. Then, 
\[
V^i=
\left\{ 
\begin{array}{cc}
k\cdot s& (i=1)\\
((M\otimes_{\Z} k)^{\ssim})^{\vee}& (i=n)\\
0   & (otherwise)
\end{array}\right.
\]   
and $d(s)=0$, $d(x)={}^tg^n(x)\cdot s$ for $x\in V^n$. Here, $(M\otimes_{\Z}k)^\ssim$ is the vector space $M\otimes_\Z k$ on which $1\in\Z$ acts by the semisimple part $g^s$, and $\Z$ acts  on $V^1$ trivially. 
\end{exa}
We  give an explicit description of the  model  of   components of a free loop space  (Prop.\ref{proploop}) under the assumption that the component contains loops which represent an element of the center of the fundamental group. We also present a model of cell attachment (Example \ref{exacell}). Models of these topological constructions are not known for the formulation of \cite{nsrat}.  By a method similar to the proof of Thm.\ref{result3}, we will prove that for a nilpotent simplicial set $K$ of finite type, the minimal model of $\ared(K)$ is isomorphic to the Sullivan's minimal model with trivial $\pi_1$-action (Thm.\ref{thmnilpotent}). We also provide a description of minimal model  of a classifying space of a group which is an extension of given group by an abelian group (Thm.\ref{thmext}).\\
\subsubsection{ Equivalence with the fiberwise rationalization} 
In general, it is impossible to recover the fundamental group  from the closed dg-category. The best thing we can obtain is the pro-algebraic completion (\cite{hm}, see also subsection \ref{notation}). So we cannot expect the closed dg-category corresponds to the fiberwise rationalization of a space. Instead, we prove an equivalence between the fiberwise rationalization and the closed dg-category with subsidiary data.  We say a pointed connected simplicial set  is algebraically good if its fundamental group is algebraically good (see Def.\ref{defalggood}) and each higher homotopy groups are finite dimensional $\Q$-vector spaces.  Let $\rep(\gam)$ denotes the category of finite dimensional $k$-linear representation of a discrete group $\gam$.
\begin{thm}[Thm.\ref{thmspace}]\label{result5}
There exists a category $\tannplusgd$ whose objects are triples $(T,\gam,\phi)$ consisting of an object $T\in\tannp$,  an algebraically good group $\gam$, and an equivalence $\phi:\coc^0T\to \rep (\gam)$ of closed $k$-categories with a fiber functor. To a simplicial set $K$ whose fundamental group is algebraically good and whose higher homotopy groups are of finite rank, we can assign an object $(\tdr(K),\pi_1(K),\phi_K)\in\tannplusgd$. This construction induces an equivalence between the homotopy category of algebraically good spaces and the homotopy category of $(\tannplusgd)^\op$.  Under this equivalence $(\tdr(K),\pi_1(K),\phi_K)$ corresponds to the fiberwise rationalization of $K$.
\end{thm}
\subsubsection{Relation with other works}  
We shall mention the relation with other works: \cite{champs,kpt2} and \cite{prid0}. The schematic homotopy types and schematization  are defined over any field of any characteristic. We use results of \cite{champs,kpt2} in this paper. In \cite{prid0}, another kind of algebraic models of  spaces is proposed. They are called pro-algebraic homotopy types and realized as certain simplicial affine group schemes. Schematic homotopy types and pro-algebraic homotopy types are closely related. A pro-algberaic homotopy types are considered as a relative pro-unipotent completion of a schematic homotopy type and these two objects are equivalent on those which come from spaces (see \cite[Cor.3.57]{prid0}).  Some of the results of this paper were proved earlier by To\"en and Pridham (see \cite{champs,kpt2} and \cite{prid0}). For example, in the notation of Thm.\ref{result3}, it was proved in \cite{prid0} that $(V^i)^\vee$ is isomorphic to $\pi_i\otimes k$ as vector spaces under a  bit stronger assumption,  see \cite[Thm.1.58, Rem.4.43]{prid0}. A feature of our approach is that the closed dg-categories are nearer to the equivariant dg-algebras than other  models. In fact, the correspondence between dg-categories and dg-algebras is very clear and so we can obtain  descriptions of the action of  fundamental group on  homotopy groups and k-invariants   as in Thm.\ref{result3} and prove fundamental theorems \ref{thmnilpotent} and \ref{thmext}. These are new results.


\subsubsection{Organization of the paper}
In  section \ref{preliminaries}, we mainly gather  definitions and results from other papers. We recall basic properties of the completion of dg-categories from \cite{sim}. We see that the completion, slightly modified, fits in the context of  closed dg-categories. \\
\indent In  section \ref{dgcdga}, we prepare some technical results  to prove  results in sections \ref{derhamhtpy} and \ref{sht}.    In \ref{tandgc}  we introduce  the notion of Tannakian dg-cateogries.   In \ref{comparison}, we compare Tannakian dg-categories and equivariant dg-algebras and deduce  some lemmas. In the statements and proofs of these lemmas, we use the model category structure on the category of closed dg-categories (see Thm.\ref{thmmodelcldgc}). We also use internal homs. We  introduce  a notion of iterated Hirsch extensions of dg-algebras.  In \ref{derhamfunct}, we define the functor $\tdr$, which we call the generalized de Rham functor, and prove that $\tdr(K)$ comes from $\ared(K)$. All arguments in this section are elementary except for the languages of model categories.\\
\indent In section \ref{derhamhtpy},  we recall the notion of algebraically goodness of a discrete group and prove Thm.\ref{result3} and \ref{result5}. We also provide some examples. For the proof of Thm.\ref{result3}, we mainly follow the method of \cite{grimor} and justify a technical part by using the model category structure on cubical sets in \cite{cisinski}, see Appendix \ref{derhamthm}. In the proofs of some results in this section, we use results of the next section \ref{sht}. \\
\indent In section \ref{sht}, we prove Thm.\ref{result4}. In \ref{definitions}, we recall the notions of schematic homotopy types and schematization from \cite{champs} and define a functor from the category of simplicial presheaves to the category of closed dg-categories. This is an analogue of the generalized de Rham functor (and denoted by $\tdr$, too). As for  logical order, section \ref{sht} is previous to section \ref{derhamhtpy}. \\   
\indent In Appendix, we show some variants of the polynomial de Rham theorem and summarize the Tannakian theory of \cite{dmos}.\\ 
\subsection{Notations and terminologies}\label{notation}
We fix a field $k$ of characteristic 0.  $\Q$ denotes the field of rational numbers. All differential graded objects are assumed to be defined over $k$ and {\em non-negatively cohomologically graded}. 
We denote by $\nngc$ the category of non-negatively cohomologically graded complexes and chain maps. By dg-algebra, we mean commutative dg-algebra (with non-negative cohomological grading). Dg-algebra and dg-category are abbreviated to dga and dgc, respectively. We denote by $\vect$ the category of finite dimensional $k$-vector spaces and $k$-linear maps. More precisely, it denotes a suitably small full subcategory, see the paragraph which precedes Def.\ref{defacldgc} for definition.   We use the notations and terminologies in \cite[1.1]{moriya}. In particular, for a category (resp. a dg-category) $C$, $\ob(C)$ denotes the set of objects of $C$ and $\hom_C(c_0,c_1)$ denotes the set of morphisms (resp. the complex of morphisms) between two objecs $c_0$ and $c_1\in\ob(C)$. $\dgc$ denotes the category of small non-negatively cohomologically graded dg-categories and dg-functors. For a dg-category $C$, $\coc^0C$ denotes the category of $0$-th cocycles of $C$. Its objects are the same as objects of $C$, and its morphisms are morphisms of $C$ which are cocycles of degree 0. For two dgc's $C$ and $D$, $C\boxtimes D$ denotes a dgc whose objects are pairs $(c,d)$ of $c\in \ob(C)$ and $d\in\ob(D)$, and whose complexes of morphisms are tensors of those of $C$ and $D$  (see \cite{moriya}). If a dga $A$ is free as a graded commutative algebra, i.e., it is the tensor of the symmetric algebra generated by even degree generators and the wedge algebra of odd degree ones, we sometimes write $A=\bigwedge(V^i)$ where $V^i$ denotes a module of generators of degree $i$.\\

\indent All schemes are assumed to be defined over $k$. For an affine scheme $X$, we denote by $\ring(X)$ the coordinate ring of $X$. We always identify a finite dimensional $k$-vector space with an affine additive group scheme in the obvious way. \\
\indent Let $\gam$ be a discrete group and $G$ be an affine group scheme.  The term, $\gam$-module or $\gam$-representation represents the same thing.  We always identify  $G$-modules (or $G$-representations) with $\ring(G)$-comodules (see \cite[P.126]{dmos}). $\rep^{\infty}(\gam)$ (resp. $\rep^{\infty}(G)$) denotes the category of possibly infinite dimentional $\gam$-modules (resp. $G$-modules) over $k$ and  $\rep(\gam)$ (resp. $\rep(G)$) denotes the full subcategory of $\rep^{\infty}(\gam)$ (resp. $\rep^{\infty}(G)$) consisting of finite dimensional objects. For a technical reason, we need to make the categories $\rep(\gam)$ and $\rep(G)$ suitably small. See see the paragraph which precedes Def.\ref{defacldgc} for the precise definition. 
$\ring(G)$ has two natural $G$-action: the right and left translations. We denote by $\ring(G)_r$ (resp. $\ring(G)_l$) $\ring(G)$ considered as a $G$-module by the right (resp. left) translation. We say a representation (of a discrete group or a group scheme) is semisimple if it can be decomposed into a direct sum of irreducible representations. We say an affine group scheme is {\em pro-reductive} (or simply, {\em reductive}) if any of its representations is semi-simple. Any affine group scheme $G$ has a decomposition:
\[
G\cong \ru(G)\rtimes G^\red,
\]  
where $\ru(G)$ is the pro-unipotent radical of $G$, and $G^\red=G/\ru(G)$ is pro-reductive (see \cite{prid3} for details). 
Representations of $G^\red$ are in one to one correspondence with semi-simple representations of $G$ via the pullback by the projection $G\to G^\red$. \\
\indent  The {\em pro-algebraic completion of} $\gam$ (\cite{hm}), denoted by $\gam^\alg$, is an affine group scheme over $k$ with a group homomorphism $\psi_\gam:\gam\to\gam^\alg(k)$, where $\gam^\alg(k)$ denotes the group of $k$-valued points of $\gam^\alg$, such that finite dimensional $k$-linear representations of $\gam$ are in one to one correspondence with finite dimensional representations of $\gam^\alg$ via the pullback of action by $\psi_\gam$. 
We put $\gam^\red:=(\gam^\alg)^\red$ and call it the {\em pro-reductive completion of} $\gam$. Finite dimensional representation of $\gam^\red$ are in one to one correspondence with finite dimensional semi-simple $k$-representation of $\gam$ via the pullback by $\gam\stackrel{\psi_\gam}{\to}\gam^\alg(k)\to \gam^\red(k)$.         
\\ 

\indent Our notion of model categories is that of \cite{hov}. $\ho(M)$ denotes the homotopy category of a model category $M$. If $M'$ is a full subcategory of $M$ which is stable under weak equivalences, $\ho(M')$ denotes the full subcategory of $\ho(M)$ spanned by $M'$. This is isomorphic to the localization of $M'$ by weak equivalences. $[-,-]_{M'}$ denotes the set of morphisms of $\ho(M')$. $\sset$ (resp. $\ssetp$ denote the category of simplicial sets (resp. pointed simplicial sets). For a group $\gam$, $B\gam$ or $K(\gam,1)$ denotes the simplicial nerve of $\gam$.   
  
  \section{Preliminaries}\label{preliminaries}
 
\subsection{Closed dg-categories}\label{cldgc}
The following is a rewrite of \cite[Def.2.1.1]{moriya}, where we call the same objects closed tensor dg-categories.
\begin{defi}[closed dg-categories, $ \cldgc, \clcat$]\label{defofcldgc}
 \ \\
\textup{(1)} Let $C$ be an object of $\dgc$. A \textup{closed tensor structure} on $C$
is a 11-tuple
\[
((-\ot -),\, \mb{1},\, a,\, \tau ,\, u,\, \inhom,\,  \phi,\,  (-\oplus -),\, s_1,\, s_2,\, \mb{0})
\]
consisting of 
\begin{enumerate}
\item a morphism $(-\ot -):C\bt C\longrightarrow C \in \dgc$,
\item a distinguished object $\mb{1}\in C$,

\item natural isomorphisms 
\begin{align*}
a :((-\ot -)\ot -)&\Longrightarrow  
(-\ot (-\ot -))\,\, \, :(C\bt C)\bt C\cong C\bt (C\bt C)\longrightarrow C,  \\
\tau:(-\ot -)&\Longrightarrow (-\ot -)\circ \mca{T}_{C,C}:C\bt C\longrightarrow C,\\
u:(-\ot \mb{1})&\Longrightarrow id_C\qquad\qquad\ \,:C\longrightarrow C
\end{align*}
satisfying
usual coherence conditions on associativity, commutativity and unity, see \cite[pp.251]{mac},
\item a morphism $\inhom :C^{op}\bt C\longrightarrow C \in \dgc$,
\item a natural isomorphism 
\begin{align*}
\phi :\homo _C(-\ot -,-)\Longrightarrow \homo _C(-&,\inhom (-,-)):C^{op}\bt C^{op}\bt C\longrightarrow \nngc,
\end{align*}
\item a morphism $(-\oplus -):C\times C\longrightarrow C\in \dgc$,
\item two natural transformations 
\[
\xymatrix{P_1\ar@{=>}[r]^{s_1\ \ \ }&(-\oplus -)&
P_2:C\times C\ar@{=>}[l]_{s_2}\ar[r]& C,}
\] where $P_i:C\times C\longrightarrow C$ is 
the $i$-th projection, such that the induced morphism $s_1^*\times s_2^*:\homo _C(c_0\oplus c_1,c')
\longrightarrow \homo_C(c_0,c')\times \homo_C(c_1,c')$ is an isomorphism (i.e., $c_0\oplus c_1$ is a coproduct), and
\item a distinguished object $\mb{0}\in \ob (C)$ such that $\homo _C(\mb{0},c)=0$ for any 
$c\in \ob (C)$. 
\end{enumerate}
We call $(-\otimes -)$ a \textup{tensor functor} and $\inhom$ a \textup{internal hom functor}.\\
\textup{(2)} 
 A \textup{closed  dg-category} is an object $C$ of $\dgc$ equipped with a closed tensor structure.
  For two closed dg-categories $C, D$, a \textup{morphism of closed dg-categories} is a morphism $F:C\rightarrow D$ 
of dgc's which preserves all of the above structures. For example,
$F(c\otimes d)=F(c)\otimes F(d)$ (not only naturally isomorphic), $F(\tau_{c,c'})=\tau_{Fc,Fc'}$ 
and $F(\mb{1})=\mb{1}$.
We denote by $\cldgc$ the category of small closed dgc's.\\
\textup{(3)} A \textup{closed  $k$-category} is a closed dg-category whose complexes of morphisms are concentrated in degree 0 and a \textup{morphism of closed $k$-categories} is the same as a morphism of closed  dg-categories. We denote by $\clcat$ the category of small closed $k$-categories. 
\end{defi}
The important part of this definition is the data concerning $(-\otimes-)$ and $\inhom$. We assume the existence of  coproducts in order to make the initial object of $\cldgc$ be (equivalent to) the closed $k$-categroy of finite dimensional $k$-vector spaces\\

\indent \textbf{Notation.} We set $C(c):=\hom_C(\uni,c)$ for a closed dgc $C$ and an object $c\in\ob(C)$.
\begin{exa}\label{exacldgc}
\textup{(1)} Let $\Gamma$ be a descrete group. The $k$-linear category $\rep (\Gamma)$ of finite dimensional $k$-linear representations has a closed tensor structure. The tensor of two representations is, as usual, the tensor of vector spaces with the diagonal action and the internal hom is similar.\\
 \textup{(2)(The dg-category of flat bundles \cite{sim})} Let $X$ be a $C^\infty$-manifold.
A flat bundle $(V,D)$ on $X$ is a pair of a $C^\infty$-vector bundle $V$ and a flat connection $D:V\to \msc{A}^1_X\otimes V$, where $\msc{A}^1_X$ is the sheaf of $C^\infty$ one-forms on $X$. 
The tensor of two flat bundles $(V, D)$, $(V',D')$ is the pair of the tensor of vector bundles $V\otimes V'$ and the flat connection $D\otimes \id +\id\otimes D'$. The internal hom is similar. \\
\indent The \textup{dg-category $\msc{C}_{dR}$ of flat bundles on} $X$ is defined as follows. Its objects are flat bundles on $X$ and its complex of morphisms between $(V,D)$ and $(V',D')$ is the twisted de Rham complex of forms with coefficients in the internal hom $\inhom((V,D),(V',D'))$. It is easy to see the tensor and the internal hom defined above are extended to a closed tensor structure on $\msc{C}_{dR}$. See \cite{sim} for details. \\
\textup{(3)} The reader may feel the definition of morphism of closed dgc's is non-natural as it does not require natural isomorphisms but equalities.
The motivation of this definition is to ensure that $\cldgc$ is closed under limits and colimits. An example of a morphism of $\cldgc$ is the functor $\rep(\Gamma')\to \rep(\Gamma)$ induced by a group homomorphim $\Gamma\to \Gamma'$.
\end{exa}
We apply the notions of an equivalence and a quasi-equivalence to 
morphisms of $\cldgc$ via the forgetful functor $\cldgc\longrightarrow \dgc$. For example, we say a morphism in $\cldgc$ is an equivalence if it induces an equivalence of underlying categories. Note that an equivalence  in $\cldgc$ does not always have a  
quasi-inverse which is a morphism of $\cldgc$. We say two objects of $\cldgc$ are equivalent if  
they can be connected by a finite chain of equivalences in $\cldgc$.\\ 
\indent We denote the initial object of $\cldgc$ by $\vect$.  $\vect$ is equivalent to the closed $k$-category of all finite dimensional $k$-vector spaces and $k$-linear maps. In fact, $\vect$ is identified with the smallest full subcategory which includes the distinguished objects $\uni$ and $\mb{0}$ and is closed under $\otimes$, $\inhom$, and $\oplus$.\\
\indent In the rest of the paper, we assume a vector space which underlies a finite dimensional representation of a discrete group or an affine group scheme belongs to $\vect$. 
\begin{defi}[closed dg-categories with a fiber functor, $\cldgck, \clcatk$]\label{defacldgc}
\ \\ 
\textup{(1)} The  \textup{category of closed dg-categories with a fiber functor} is  the over category
\[
\cldgc/\vect
\]
and denoted by $\cldgck$. An object $(C,\omega_C)$ of $\acldgc$ consists of a closed dg-category $C$ and a morphism $\omega_C:C\to\vect \in\cldgc$. We call $\omega_C$ the \textup{fiber functor of} $C$. \\
\textup{(2)} A \textup{closed  $k$-category with a fiber functor} is a closed dg-category with a fiber functor whose complexes of morphisms are concentrated in degree 0 and  \textup{morphisms of closed $k$-categories with a fiber functor} are the same as  morphisms of closed tensor dg-categories with a fiber functor. We denote by $\clcatk$ the category of small closed $k$-categories with a fiber functor. \end{defi}
\begin{exa}\label{defrep} 
Let $\gam$ (resp. $G$) be a discrete group (resp. an affine group scheme).  We regard $\rep(\gam)$ (resp. $\rep(G)$) as an object of $\clcatk$ with the forgetful functor to $\vect$.
\end{exa}

The following is proved in \cite{moriya}. 
\begin{thm}[Thm.2.3.2 of \cite{moriya}]\label{thmmodelcldgc}
\textup{(1)}  The category $\cldgc$ has a 
cofibrantly generated model category structure where 
weak equivalences and fibrations are
defined as follows.
\begin{itemize}
\item A morphism $F:C\rightarrow D\in\cldgc$ is a weak equivalence if and only if
it is a quasi-equivalence.
\item A morphism $F:C\rightarrow D\in\cldgc$ is a fibration if and only if it satisfies 
the following two conditions.
\begin{itemize}
\item For $c,c'\in \ob (C)$ the morphism $F_{(c,c')}:\hom _C(c,c')\to \hom _D(Fc,Fc')$ is a levelwise epimorphism. 

\item For any $c\in \ob (C)$ and any isomorphism $f:Fc\to d'\in \coc^0(D)$, there exists an isomorphism 
$g:c\to c' \in \coc^0(C)$ such that $F(g)=f$.
\end{itemize}
\end{itemize}
\textup{(2)} 
$\cldgck$ has a model category structure induced by that of $\cldgc$.
\end{thm}

\subsection{Completeness of dg-category}\label{complete}
We shall recall the notion of completeness of a dg-category from \cite[section 3]{sim}. Let $C$ be a dg-category. An extension in $C$ is a pair of morphisms
\[
c_0\stackrel{a}{\to} c_2\stackrel{b}{\to}c_1
\]
with $a,b\in \homo^0$, $b\circ a=0$ and $d(a)=0$, $d(b)=0$, such that a splitting exists: a splitting is a pair of morphisms of degree 0
\[
c_0\stackrel{g}{\leftarrow}c_2\stackrel{h}{\leftarrow}c_1
\]
such that $ga=id_{c_0}$, $bh=id_{c_1}$ and $ag+hb=id_{c_2}$. We define a morphism $\delta\in\homo^1(c_1,c_0)$ by $\delta = gd(h)$. $d(\delta )=0$ and $\delta$ defines a class  $[\delta]\in H^1(\homo (c_1,c_0))$. It is easy to check that this class is independent of a choice of  splittings. We call $[\delta]$ the  class of the extension $c_0\to c_2\to c_1$. 
\begin{defi}\label{defcomplete}
We say a dg-category $C$ is complete if for each $c_0, c_1\in\ob (C)$ any class in $H^1(\homo (c_1,c_0))$ is a class of some extension.
\end{defi}
\begin{exa}
Let $\msc{C}_{dR}$ be the category of flat bundles defined in Example \ref{exacldgc},(2). As in \cite{sim}, $\msc{C}_{dR}$ is complete. In fact, for a cocycle $\delta\in \homo^1_{\msc{C}_{dR}}((V,D),(V',D'))$, we define a flat bundle $(V'',D'')$ by $V''=V\oplus V'$ and $D''=\left(
\begin{smallmatrix}
D & 0 \\
\delta & D' 
\end{smallmatrix}
\right)$. The sequence $(V',D')\longrightarrow (V'',D'')\longrightarrow (V,D)$ 
is an extension corresponding to $\delta$.
\end{exa}   
We will construct the completion of a dgc $C$ which has finite coproducts. We define a dgc $\overline{C}$ as follows. The objects of $\overline{C}$ are pairs $(c,\eta )$ with $c\in \ob (C)$ and $\eta \in \homo ^1(c,c)$ such that 
\[
d(\eta)+\eta^2=0.
\]
( We call an element $\eta$ satisfying the above equation a {\em Maurer-Cartan (MC) element on} $c$.) 
We set
\[
\homo^n((c,\eta),(c',\eta ')):=\homo^n(c,c')
\] 
and 
\[
d_{\overline{C}}(f):=d_{C}(f)+\eta'\circ f-(-1)^{\deg f}f\circ \eta.
\]
We identify $C$ with the full sub-dg-category of $\overline{C}$ consisting of objects of the form $(c,0)$, $c\in C$ and define the completion $\widehat C$ to be the smallest full sub-dg-category of $\overline{C}$ including $C$ and closed under extensions (and isomorphisms). \\
\indent The following is a rewrite of \cite[Lemma 3.1, Lemma 3.3]{sim}.
\begin{lem}[\cite{sim}]\label{lemcompletion}
\textup{(1)} Let $C$ be a dgc closed under finite coproducts. $\widehat{C}$ is complete. For two objects $(c_0,\eta_0), (c_1,\eta_1)\in\widehat C$, a MC-element corresponding to a cocycle $\alpha\in\coc^1(\hom_{\widehat C}((c_1,\eta_1), (c_0,\eta_2)))$ is given by 
$\left(\begin{array}{cc}
\eta_0&\alpha \\
0&\eta_1
\end{array}\right)$. \\
\textup{(2)} The functor $C\to \widehat C$, $c\mapsto (c,0)$ is initial up to natural isomorphims among dg-functors $C\to D$ with $D$ complete. More precisely,
for such functor $F:C\to D$ there exists a functor $\widehat F:\widehat C\to D$  factorizing $F$ and such $\widehat F$ is unique up to unique natural isomorphisms.\\
\textup{(3)} Let $F:C\to D$ be a quasi-equivalence. The induced functor $\widehat F:\widehat C\to \widehat D$ is also a quasi-equivalence. 
\end{lem}
\subsubsection{Completion and closed dg-categories}\label{modify}
A closed tensor structure on $C$ induces a closed tensor structure on $\widehat C$ as follows. 
\[
\begin{split}
(c_1,\eta_1)\otimes (c_2,\eta_2)&=(c_1\otimes c_2,\eta_1\otimes \id+\id\otimes \eta_2),\\
\inhom((c_1,\eta_1),(c_2,\eta_2))&=(\inhom(c_1,c_2), \inhom(\id, \eta_2)-\inhom(\eta_1,\id)).
\end{split}
\]
The other structures such as homomorphisms between complexes of morphisms and natural isomorphisms are the same as those of $C$.
If $C$ is a closed dgc, we consider $\widehat C$ as a closed dgc with this induced structure. Note that this closed  dgc does not have a universality like Lem.\ref{lemcompletion} in $\cldgc$. We shall modify $\widehat C$. The modification makes the tensor and the internal hom "free". For example, we want to avoid two objects $(c_0\otimes c_1)\otimes c_2$ and $c_0\otimes (c_1\otimes c_2)$ happen to be equal. \\
\indent Let $D$ be a closed dgc. Let $\wcl(\ob(D)\sqcup \{\uni,\mb{0}\})$ denote the set of the words generated by $\ob(D)\sqcup\{\uni,\mb{0}\}$ with operations $\otimes', \inhom'$ and $\oplus'$ ($\uni$ and $\mb{0}$ are formal symbols, see \cite[sub-subsection 2.2.2]{moriya} for an explicit definition). We regard $\ob(D)\sqcup\{\uni,\mb{0}\}$ as a subset of $\wcl(\ob(D)\sqcup \{\uni,\mb{0}\})$. Let  
\[
R:\wcl(\ob(D)\sqcup\{\uni,\mb{0}\})\longrightarrow \ob(D)\sqcup\{\uni,\mb{0}\}
\]
be the function given by 
\begin{enumerate}
\item $RX=X$ for $X\in\ob(D)\sqcup \{\uni,\mb{0}\}$.

\item $R(X\otimes'Y)=RX\otimes RY$, $R(\inhom'(X,Y))=\inhom(RX,RY)$ and $R(X\oplus'Y)=RX\oplus RY$, inductively.
\end{enumerate}
We define a closed dgc $D^c$ as a "pullback" of $D$ by $R$ :
\[
\ob(D^c)=\wcl(\ob(D)\sqcup\{\uni,\mb{0}\}),\ \ \ \ \hom_{D^c}(X,Y)=\hom_D(RX,RY).
\]
Clearly the construction 
\[
\cldgc\ni D\longmapsto D^c\in\cldgc
\]
is functorial and $R$ induces a natural equivalence $R_D:D^c\longrightarrow D\in\cldgc$.
\begin{lem}\label{lemcompletion2}Let $E\in\cldgc$. Let $C, D\in\cldgc/E$. Suppose   $D$ is complete.\\
\textup{(1)}Suppose $D$ is fibrant in $\cldgc/E$. We regard $C^c$  as an object over $E$ whose augmentation $a_{C^c}:C^c\to E$ is  the composition $\xymatrix{C^c\ar[r]^{R_C}&C\ar[r]& E}$.
Suppose an augmentation $a_{\widehat C^c}:(\widehat C)^c\to E\in\cldgc$ such that $a_{\widehat C^c}\circ (I_C)^c=a_{C^c}$ ($I_C:C\to \widehat C$ is the canonical inclusion) is given. Let $F:C^c\longrightarrow D\in\cldgc/E$ be a morphism. There exists a morphism $\widetilde F:(\widehat{C})^c\longrightarrow D\in\cldgc/E$ such that the following diagram is commutative.
\[
\xymatrix{C^c\ar[r]^F\ar[d]_{(I_C)^c}&D \\
\widehat C^c\ar[ur]_{\widetilde F}&}
\] 
Let $\widetilde F':(\widehat{C})^c\longrightarrow D\in\cldgc/E$ be another morphism with $F=\widetilde F'\circ (I_C)^c$. Then there exists a unique natural isomorphism $\varphi:\widetilde F\Rightarrow \widetilde F'$ preserving tensors, such that $\varphi|_{C^c}$  and $(a_D)_*(\varphi)$ are the identities ($a_D:D\to E$ is the augmentation of $D$).\\
\textup{(2)} Assume the inclusion $E\to \widehat E$ is an isomorphism (i.e., MC-elements in $E$ are all zeros).  The pullback by the canonical functor $C\to \widehat C$ gives a bijection:
\[
[\widehat C,D]_{\cldgc/E}\cong [C, D]_{\cldgc/E}.
\]
Here, the augmentation $\widehat C\longrightarrow E$ is given by $\widehat C\to \widehat E\cong E$. 
\end{lem}
\begin{proof} 
(1) We first consider the case of $E=*$. As in \cite{sim}, one can choose a functor $F_1:\widehat C^c\to D$ such that $F=F_1\circ (I_C)^c$ and $F_1$ preserves tensors and internal hom's up to natural isomorphisms which are compatible with coherency isomorphisms. Since the objects of $\widehat C^c$ are freely generated by objects of $\widehat C$ and $\{ \mb{0}, \mb{1}\}$, one can modify $F_1$ so that it becomes a morphism of $\cldgc$. Thus we get $\widetilde F$. The latter part is similar to \cite{sim}. For general $E$, we first take a morphism $F_2:\widehat C^c\to D\in\cldgc$ such that $F=F_2\circ (I_C)^c$.  By the latter claim in the case of $E=*$, we have a unique natural isomorphism $\varphi:a_D\circ F_2\Rightarrow a_{\widehat C^c}$. Let $x\in \ob(\widehat C)\sqcup \{\mb{1},\mb{0}\}$. As the augmentation $a_D:D\to E\in\cldgc$ is a fibration, one can lift $\varphi_x$ to an isomorphism $f_x:F_2(x)\to {}^{\exists}\widetilde F(x)\in D$. Using $f_x$'s, one can modify $F_2$ so that $a_D\circ F_2=a_{\widehat C^c}$. This modified $F_2$ is the required $\widetilde F$. The latter claim follows from the case of $E=*$ \\
(2) We may assume $C$ is cofibrant and $D$ is fibrant in $\cldgc/E$.  Note that $C^c$ and $\widehat C^c$ are cofibrant. Indeed, let $Q(C^c)$ be a cofibrant replacement of $C^c$ with a trivial fibration $Q(C^c)\to C^c$. As the composition $Q(C^c)\to C^c\to C$ is a trivial fibration, we can take a right inverse $C\to Q(C^c)$. Using this morphism, one can see  the map $Q(C^c)\to C^c$ has a right inverse so $C^c$  is cofibrant. For $\widehat C^c$, one can find a right inverse of a trivial fibration $Q(\widehat C^c)\to \widehat C^c$ using (1). We only have to check the map
\[
((I_C)^c)^*:\homo_{\cldgc/E}(\widehat C^c,D)\longrightarrow \homo_{\cldgc/E}(C^c, D)
\]
induces a bijection between the sets of right homotopy classes. The surjectivity follows from (1). Let 
$\widetilde F_1,\widetilde F_2:\widehat C^c\longrightarrow D\in\cldgc/E$ be two morphisms such that 
$\widetilde F_1|_{C^c}$ and $\widetilde F_2|_{C^c}$ are right homotopic. Let $PD$ be a path object of $D$.
We may regard $C^c$, $\widehat C^c$ and $PD$ as objects of $\cldgc/D\times_ED$. Then by (1), there exists a 
morphism $\widetilde H :\widehat C^c\longrightarrow PD\in\cldgc/E$ such that the following diagram commutes.
\[
\xymatrix{C^c\ar[r]^{H}\ar[d]&PD\ar[d]\\
\widehat C^c\ar[ur]^{\widetilde H}\ar[r]_{\widetilde F_1\times \widetilde F_2\ \ \ }&D\times_ED}
\]
Thus, $\widetilde F_1$ and $\widetilde F_2$ are right homotopic.
\end{proof}
\section{Tannakian dg-categories and reductive equivariant  dg-algebras}\label{dgcdga}

\subsection{Tannakian dg-categories}\label{tandgc}
We shall define Tannakian dg-categories. For the definition of neutral Tannakian categories, see Appendix \ref{proobj}.
\begin{defi}[Tannakian dgc's, $\tann$, $\tannp$]\label{deftannakian}
Let $C$ be a closed dg-category. We say $C$ is a Tannakian dg-category  if 
the following conditions are satisfied.
\begin{enumerate} 
\item $Z^0(C)$ is a neutral Tannakian category with respect 
to the closed tensor structure induced from that of $C$ or equivalently, there exist an affine group scheme $G$ over $k$ and a finite chain of morphisms of closed $k$-categories: 
\[
\coc^0(C)\to C_1\leftarrow C_2\rightarrow\cdots\leftarrow C_n\rightarrow \rep(G),
\]
where all arrows are  equivalences of underlying categories. In particular, $\coc^0(C)$ is an Abelian category. 
\item $C$ is complete (see Def.\ref{defcomplete}). 
\item If $0\to c_0 \to c_1 \to c_2\to 0$ is a short exact sequence in $Z^0(C)$ then 
$c_0 \to c_1 \to c_2$ is an extension in $C$ in the sense explained in the beginning of subsection \ref{complete} .
\end{enumerate}
We denote by $\tann$ the full subcategory of $\cldgc$ consisting of Tannakian dg-categories. We denote by $\tannp$ the full subcategory of $\acldgc$ consisting of objects whose underlying closed dgc belongs to $\tann$.\end{defi}
The third condition means that extensions in terms of representations and extensions in terms of a dgc coincide. Let $C$ be a Tannakian dg-category and $C^\ssim$ denote full sub dg-category of $C$ consisting of semisimple objects of $\coc^0C$. The third one is equivalent to the one that the morphism $\widehat{C^\ssim}^c\to C$ induced by the natural inclusion is a quasi-equivalence. So $\tann$ and $\tannp$ are stable under weak equivalences of $\cldgc$ and $\acldgc$, respectively. Note that for a Tannakian dg-category with a fiber functor $(C,\omega_C)\in\tannp$, the functor $\coc^0\omega_C:\coc^0C\to \vect$ is automatically exact and faithful.

\subsection{Equivariant  dg-algebras}\label{dga}
Let $G$ be an affine group scheme. We denote by $\dgalg(G)$ the {\em category of $G$-equivariant  dg-algebras}. An object of $\dgalg(G)$ is a commutative dg-algebra with a $\ring(G)$-comodule structure which is compatible with the grading, the differential and the algebra structure. A morphism of $\dgalg(G)$ is a morphism of dga's which is compatible with the $\ring(G)$-comodule structures. $G$-equivariant  dg-algebra is abbreviated to $G$-dga. We say a $G$-dga $A$ is {\em connected} if $\coh^0 A\cong k$.  $\dgalg(G)_0$ denotes the full subcategory of $\dgalg(G)$ consisting of connected objects.\\ 
\indent The following can be proved by an argument similar to the non-equivariant case.
\begin{prop}
Let $G$ be a reductive affine group scheme. The category $\mathrm{dg}\ms{Alg}(G)$ admits a model category structure such that 
\begin{enumerate}
\item a morphism $f:A\longrightarrow B\in\mathrm{dg}\ms{Alg}(G)$ is a weak equivalence if and only if it is a quasi-isomorphism (of underlying complexes), and
\item a morphism $f:A\longrightarrow B\in\mathrm{dg}\ms{Alg}(G)$ is a fibration if and only if it is a levelwise epimorphism.
\end{enumerate}
\end{prop}
\begin{defi}[minimal $G$-dga's]\label{defminimal}
Let $G$ be an affine group scheme. We say a connected $G$-dga is \textup{minimal} if its underlying dga is minimal in the usual sense (see \cite{bg,grimor}). Let $A$ be a connected $G$-dga. A \textup{minimal model of} $A$ is a minimal $G$-dga $\M$ such that there exists a weak equivalence $\M\to A$ of $G$-dga's. For a minimal $G$-dga $\M$, the $G$-module
\[
[\M/(\M^{\geq 1}\cdot\M^{\geq 1})]^i
\]
is called the \textup{$i$-th indecomposable module of} $\M$. Here, $\M^{\geq1}\cdot\M^{\geq 1}$ is the submodule of $\M$ generated by $\{ x\cdot y|\deg x\geq 1, \deg y\geq 1\}$.     
\end{defi}
The following is well-known. 
\begin{prop}
Let $G$ be a pro-reductive affine group scheme. Then, any connected $G$-dga has a minimal model. such minimal model is unique up to non-unique natural isomorphisms. \end{prop}
\begin{proof} 
The proof is similar to the usual trivial group case. See \cite{bg,grimor,sul}.
\end{proof}
\begin{defi}[reductive dga's, $\credeqalg, \aredeqalg$]\label{defofeqalg}
\ 
\begin{enumerate}
\item A \textup{reductive equivariant dg-algebra} in short, a \textup{reductive dga} is  a pair $(G,A)$ of a pro-reductive affine group scheme $G$ and a $G$-dga $A$. 
\item A \textup{morphism of reductive dga's} $f:(G,A)\rightarrow (H,B)$ is a pair of a morphism of group schemes  
$f^{\mathrm{gr}}:H\rightarrow G$ and a morphism of $H$-dga's $f:(f^{\mathrm{gr}})^*A.\longrightarrow B$. (Note that $f^{gr}$ defines a functor $(f^{gr})^*:\dgalg(G)\longrightarrow \dgalg(H)$ by pulling back the group action.)
\item A morphism of reductive dga's $f:(G,A)\rightarrow (H,B)$ is said to be a \textup{quasi-isomorphism} if $f^{\mathrm{gr}}$ is an isomorphism and $f:(f^{\mathrm{gr}})^*A.\longrightarrow B$ is a quasi-isomorphism.\end{enumerate}
The category of reductive dga's is denoted by  $\redeqalg$. We say a reductive dga $(G,A)$ is connected if $\coh^0(A)\cong k$. We denote by $\redeqalg_0$ the full subcategory of $\redeqalg$ consisting of  connected objects and  by $\aredeqalg$ the over category $\redeqalg _0/(e,k)$, where $e$ is the trivial group. We always identify the category $\dgalg(G)$ with a subcategory of $\redeqalg$ in the obvious way. For an object $(G,A)\in\credeqalg$, a \textup{minimal model} of $(G,A)$ is a minimal model of $A$ as a $G$-dga in the sense of Def.\ref{defminimal} 
\end{defi}

\subsection{Comparison}\label{comparison}
In this subsection, we show a correspondence between reductive equivariant dga's and Tannakian dgc's. The direction from dga to dgc is functorial but the other direction is not so and we only have a function $\ared:\ob(\tann_*)\to\ob(\aredeqalg)$, see Rem.\ref{remnotfunct}.  \\

\indent We shall define two functors 
\[
\Tss:\credeqalg\longrightarrow \cldgc,\quad\T:\credeqalg\longrightarrow \tann .
\]
Let $A=(G,A)\in \credeqalg$. 
For a comodule $M\in\rep^\infty(G)$, we define the module of invariants $M^G$ by 
\[
M^G:=\ker( \id_M\otimes k-\rho_M:M\longrightarrow M\otimes\ring(G)), 
\]
where $k:k\to\ring(G)$ is the unit map and $\rho_M:M\to M\otimes\ring(G)$ is the coaction. We set 
\[
\ob(\Tss A):=\ob(\mathrm{Rep}(G)),\ \ \ \ \homo _{\Tss A}(V,W):=(\inhom(V,W)\otimes A)^G.
\]
Here, $\inhom$ is the internal hom of $\rep(G)$, and $\inhom(V,W)\otimes A$ is considered as a complex of comodules. $(\inhom(V,W)\otimes A)^G$ is defined by taking  invariants in the degreewise manner.  We define the composition and  closed tensor structure of $\Tss(A)$ using corresponding structures of $\repg$ and the multiplication of $A$. A morphism $f:(G,A)\rightarrow (H,B)$ of $\credeqalg$ gives a functor 
$(f^{gr})^*:\repg \rightarrow \reph $ and $f$ induces a morphism 
$\Tss f:\Tss A\rightarrow \Tss B$ of closed dg-categories.  We set 
\[
\T A:=\widehat{\Tss A}.
\]
The right hand side is the completion (see subsection \ref{complete}).  As $\T(e,k)\cong\Tss(e,k)\cong \vect$, a morphism $(G,A)\to (e,k)$ defines  morphisms $\Tss(G,A)\to\vect, \T(G,A)\to\vect$ so we obtain functors between augmented categories: 
\[
\Tss:\aredeqalg\longrightarrow \cldgck,\quad\T:\aredeqalg\longrightarrow \tannp .
\]
We omitt the proof of the following. See \cite[Lem.1.3]{prid3}.the reductivity of $G$ is necessary for (2).
\begin{lem}\label{lembasic}
\textup{(1)} For $V\in\ob(\rep^\infty(G))$ the  coaction $V\to V\otimes_k\mca{O}(G)$ induces an isomorphism of vector spaces: $V\to (V\otimes \mca{O}(G)_r)^G$. 
Similarly, the coaction also induces an isomorphism $V\to V\otimes^G \ring(G)_l$. Here, 
\[
V \otimes^G \ring(G)_l=\ker( (\id_V\otimes \tau)\circ(\rho_V\otimes\id_{\ring(G)})-\id_V\otimes\lambda_G:V\otimes \ring(G)_l\to V\otimes\ring(G)_l\otimes\ring(G)),
\]
where $\tau:\ring(G)\otimes\ring(G)\to \ring(G)\otimes\ring(G)$ is the isomorphism given by $\tau(x\otimes y)=y\otimes x$ and  $\lambda_G$ is the coaction of the left translation.(In short, $V \otimes^G \ring(G)_l=\{\Sigma_iv_i\otimes l_i;\Sigma_ig\cdot v_i\otimes l_i=\Sigma_iv_i\otimes g\cdot_ll_i \forall g\in G\}$) In particular, as submodules of $V\otimes \ring(G)$, $(V\otimes \ring(G)_r)^G=V\otimes^G\ring(G)_l$.\\
\textup{(2)} If $f:A\to B\in \dgalg(G)_0$ is a quasi-isomorphism, the induced morphism  
$\ms{T}f:\ms{T}A\to \ms{T}B\in\cldgc$ is a quasi-equivalence. \\
\textup{(3)} $\ms{T}A$ is a Tannakian dgc  for any $A\in \credeqalg$. 
\end{lem}
\begin{rem}\label{remnotfunct}
If we define the \textup{homotopy category $\ho(\aredeqalg)$ of $\aredeqalg$} as the localization of $\aredeqalg$ obtained by inverting quasi-isomorphisms, by  Lemma \ref{lembasic} $\T$ induces a functor between homotopy categories:
\[
\T:\ho(\aredeqalg)\longrightarrow \ho(\tann).
\]
Unlike the finite group case \cite{moriya}, this is not an equivalence simply because morphisms between Tannakian dgc's do not always preserve semisimple objects. By the same reason the construction
\[
\tann\ni T\longmapsto \ared T\in \aredeqalg 
\]
is not functorial. If we restrict morphisms of $\tann_*$ to those which preserves semi-simple objects, this function  is extended to a functor which induces an equivalence of homotopy categories, see Prop.\ref{thmbasic} and Lem.\ref{lemmor}.
\end{rem}
\begin{defi}[the reductive dga associated to $T$; $\ared T$]
\ \\ 
Let $T=(T,\omega_T)$ be a Tannakian dg-category with a fiber functor. Let $\pi_1(\coc^0T)$ denote the affine group scheme which represents $\autbar^{\otimes}(\coc^0\omega_T)$ (see Appendix \ref{proobj}). We regard $\ring(\pi_1(\coc^0T)^\red)$ as a $\pi_1(\coc^0T)^\red$-representation by the right translation. We define an augmented $\pi_1(\coc^0T)^\red$-equivariant dg-algebra $\ared(T)\in\aredeqalg$ as follows. Let 
\[
\mca{O}(\pi_1(\coc^0T)^{\red})=\bigcup_{\lambda}V_{\lambda}
\] 
be the presentation  as the union of finite dimensional $\pi_1(\coc^0T)^{\red}$-subrepresentations. We regard $V_\lambda$'s as objects of $T$ via the equivalence $\widetilde{\coc^0\omega_T}:\coc^0T\to\rep(\pi_1(T)^\red)$ (see Appendix \ref{proobj}). As a complex, we set
\[
\ared(T)=\operatorname{colim}_{\lam}T(V_{\lam})
\]    
(see the notation under Def.\ref{defofcldgc}). The colimit is taken in the category $\nngc$ (in this case, this is the set-theoretic union). The $k$-algebra structure on $\mca{O}(\pi_1(\coc^0T)^{\red})$  defines a structure of dga on $\ared(T)$, the left translation of $\pi_1(\coc^0T)^{\red}$ on $\mca{O}(\pi_1(\coc^0T)^{\red})$ defines an action on $\ared(T)$ and the counit map of $\mca{O}(\pi_1(\coc^0T)^{\red})$ and the fiber functor of $T$ defines an augmentation of $\ared(T)$. Note that we do \textup{not} say the construction $T\longmapsto\ared(T)$ is functorial. \end{defi}
\begin{prop}\label{thmbasic}
\textup{(1)} For an augmented 0-connected reductive dga $A$, $\ared\T(A)$ is isomorphic to $A$ in $\aredeqalg$.\\
\textup{(2)} For a Tannakian dgc $(T,\omega_T)$ with a fiber functor, $\T\ared (T)$ is equivalent to $T$ in $\cldgck$.\\
\textup{(3)} For a Tannakian dgc $T\in\tann$, there exists a connected reductive dga $A\in\credeqalg$ such that $\T A$ is equivalent to $T$ as a closed dg-category. 
\end{prop}
By this proposition, the functor $\T:\aredeqalg\to \tannp$ and the function $\ared:\ob(\tannp)\to\ob(\aredeqalg)$ induce a bijection between the isomorphism classes (resp. quasi-isomorphism classes) of dga's and the equivalence classes (resp. quasi-equivalence classes) of closed dgc's.
\begin{proof}
(1) This follows from Lem.\ref{lembasic},(1) and the fact that finite limit and filtered colimit (in the category of modules) commutes.\\
\indent (2) Let $A=\ared T$. We may assume $(\coc^0T)^{\ssim}=\rep(G)$ where $G=\pi_1(\coc^0T)^{\red}$. In the following we deal with $\ring(G)_r$ as if it is an object of $\rep(G)$. As the action of $G$ on $\ring(G)_r$ is locally finite, this does not matter. \\
For an object $V\in\ob(\rep (G))$ we have an exact sequence of $G$-modules:
\[
\xymatrix{0\ar[r]& V\ar[r]^{\rho_V\qquad }&V_u\otimes \ring(G)_r\ar[r]^{\phi \qquad}&V_u\otimes \ring(G)_u\otimes \ring(G)_r},
\]
where $(-)_u$ denotes the corresponding trivial module and  $\phi:=(\id_V\otimes \tau)\circ(\rho_V\otimes\id_{\ring(G)})-\id_V\otimes\lambda_G$, see Lem.\ref{lembasic}.
 We have a map between  sequences of complexes
\[
\xymatrix{0\ar[r]& (V\otimes A)^G\ar[r]\ar@{-->}[dd]&(V_u\otimes\ring(G)_r\otimes A)^G\ar[r]^{\phi\otimes \id_A\qquad}&(V_u\otimes \ring(G)_u\otimes \ring(G)_r\otimes A)^G \\
&& V_u\otimes A\ar[r]\ar[d]\ar[u]^{\rho_A}&V_u\otimes \ring(G)_u\otimes A\ar[d]\ar[u]^{\rho_A} \\
0\ar[r] &\hom_T(\uni,V)\ar[r]&\hom_T(\uni,V_u\otimes\ring(G)_r)\ar[r]^{\phi_*\qquad}&\hom_T(\uni,V_u\otimes \ring(G)_u\otimes\ring(G)_r).} 
\]
As the both horizontal sequences are levelwise exact and the vertical arrows are isomorphisms by Lem.\ref{lembasic}, the dotted arrow uniquely exists. We define a dg-functor $F:\Tss A\longrightarrow T$
by $F(V)=V$ for $V\in\ob(\rep(G))$ and $F_{(V,W)}:\hom_{\Tss A}(V,W)\longrightarrow\hom_C(FV,FW)$ being the composition
\[
\hom_{\Tss A}(V,W)=(\inhom_{\rep(G)}(V,W)\otimes A)^G\longrightarrow \hom_T(\uni,\inhom_T(V,W))\cong\hom_T(V,W).
\]
One can check this is a morphism of $\cldgck$. By definition of Tannakian dgc's, $F$ is extended to an equivalence $F:\T A^c\longrightarrow T\in\acldgc$ (see Lem.\ref{lemcompletion2}). The proof of (3) is similar to that of (2). 
\end{proof}
The following lemmas are used later
\begin{lem}\label{lempushout}
\textup{(1)}Let $G$ be a reductive affine group scheme. Then, $\Tss(G,k)^c$ is cofibrant in $\acldgc$. If $f:A\to B\in\mathrm{dg}\ms{Alg}(G)$ be a cofibration betwwen 0-connected $G$-dga's, then $\Tss(f)^c:\Tss(G,A)^c\to\Tss(G,B)^c$ is a cofibration in $\cldgc$. \\
\textup{(2)} Let $(G,A)\in\aredeqalg$. 
There exists a commutative square
\[
\xymatrix{\Tss(G,k) ^c\ar[d]\ar[r]&\vect\ar[d] \\
\Tss(G,A) ^c\ar[r]& \Tss(e,A)}
\]
where $e$ denotes the trivial group, the left vertical morphism is induced by the unit $k\to A$ and the bottom horizontal morphism is induced by the morphism $(G,A)\to (e,A)$ which is the identity on dg-algebras.  This diagram is a pushout square in $\dgclp$ and  a homotopy pushout square.
\end{lem}
\begin{proof}
(1) The proof is similar to that of \cite[Thm.3.2.10]{moriya}.\\ 
(2) We prove the first assertion.
Let 
\[
\xymatrix{
\Tss(G,k)^c\ar[d]\ar[r]^{\omega}&\vect\ar[d] \\
\Tss(G,A)^c\ar[r]^F&C}
\]
be a commutative square in $\cldgc$. For $V,W\in\ob(\Tss(e,A))$ we define a morphism of complex 
\[
\widetilde F_{(X,Y)}:\homo_{\Tss(e,A)}(V,W)\longrightarrow \homo_C (V,W)
\]
as the following composition.
\[
\begin{split}
\homo(V,W)\otimes_k A&\cong(\inhom(V,W)\otimes\ring(G)\otimes A)^G \\
&\cong\homo_{\Tss A^c}(V,W\otimes\ring(G))\stackrel{F}{\to}\homo_C(V,W\otimes F(\ring(G)))\stackrel{Fu}{\to}\homo_C(V,W).
\end{split}
\]
Here, $u:\ring(G)\longrightarrow \mb{1}\in\vect$ denotes the counit map. It is easy to check $F_{(V,W)}$'s form a morphism $\widetilde F:\Tss(e,A)\longrightarrow C\in\cldgck$ and this is the unique morphism making appropriate diagram commutative.\\
 By (1) and \cite[Lem.5.2.6]{hov}, this is a homotopy pushout square so the former one is. 

 \end{proof}
\begin{rem}\label{remcofibrant}
We can give an explicit cofibrant model for any Tannakian dg-category. In fact, by Lem.\ref{lemcompletion2} and Lem.\ref{lempushout}, for a cofibrant connected $G$-dga $A$, $\T A^c$ is cofibrant in $\cldgc$ (see the proof of Lem.\ref{lemcompletion2},(2)). 
\end{rem}
The following gives a description of hom-sets of closed dg-categories in terms of dg-algebras and closed $k$-categories. This is used to translate a property of schematic homotopy types into a property of equivariant dg-algebras, see the proof of Prop.\ref{proploop}.
\begin{lem}\label{lemmor}
\textup{(1)} Let $E$ be a closed dgc and $C\in\cldgc/E$ be a closed dgc over $E$. Let $A=(G,A)\in\credeqalg$ be an object and $\Tss A^c\longrightarrow E\in\cldgc$ be a morphism. We consider $\Tss A^c$ as an object of $\cldgc/E$. \\
\indent For a morphism $\alpha:\coc^0\Tss A^c\to\coc^0C\in \clcat$ We define a dg-algebra $C_{\alpha}\in\dgalg(G)$ as follows. We fix a presentation of $\ring(G)_r$ as a union of finite dimensional $G$-representations: $\ring(G)_r\cong\bigcup_{\lam} V_{\lam}$. We set 
\[
C_{\alpha}:=\colim_{\lam}C(\alpha(V_{\lam})),
\]
see the notation under Def.\ref{defofcldgc}. The algebra structure is defined from that of $\ring(G)$ and the action of $G$ is defined from the left translation on $\ring(G)$. Then, there exists a bijection: 
\[
\homo_{\cldgc/E}(\Tss(A)^c,C) 
\cong \left\{(\alpha,\beta)\left|
\begin{array}{c}
\alpha:\coc^0\Tss A^c\to\coc^0C\in \clcat/\coc^0E,\\
\beta:A\to C_{\alpha}\in\dgalg(G)/E_{\alpha}
\end{array}
\right. \right\}
\]
If the augmentation $a_C:C\to E\in\cldgc$ is a fibration, this bijection induces a bijection:
\[
[\Tss A,C]_{\cldgc/E}\cong\{(\alpha,\beta)|\alpha\in[\coc^0\Tss A,\coc^0C]_{\clcat/\coc^0E},\ \beta\in [A,C_{\alpha}]'_{\dgalg(G)/E_{\alpha}}\}.
\]
Here, 
\begin{enumerate}
\item $[\coc^0\Tss A,\coc^0C]_{\clcat/\coc^0E}:=\homo_{\clcat}(\coc^0\Tss A^c,\coc^0C)/\sim$, where $\alpha_1\sim\alpha_2$ if and only if there exists a natural isomorphism $t:\alpha_1\Rightarrow \alpha_2$ such that $t$ preserves tensors, i.e., $t(X\otimes Y)=t(X)\otimes t(Y)$, and $(a_C)_*t:a_C\circ \alpha_1\Rightarrow a_C\circ \alpha_2:\coc^0\Tss A^c\to \coc^0E$ is the identity transformation. 
\item $[A,C_{\alpha}]'_{\dgalg(G)/E_{\alpha}}:=[A,C_{\alpha}]_{\dgalg(G)/E_{\alpha}}/\sim$, where $f_1\sim f_2$ if and only if there exists a tensor preserving natural automorphism $t:\alpha\Rightarrow\alpha$ over $E$ such that $t_{\alpha(\ring(G))}\circ f_1=f_2$.
\end{enumerate}
\textup{(2)} We use the notation of the previous part. Let $(G,A), (H,B)\in\aredeqalg$. Let $\alpha:\coc^0\Tss(G,A)^c\to\coc^0\T(H,B)\in \clcatk$ be a morphism. Suppose $\alpha$ preserves semi-simple objects so that it induces a morphism $\alpha^*:H\to G$.
Then, in the case $E=\vect$,
\[
[A,\T(H,B)_\alpha]_{\dgalg(G)/\vect_{\alpha}}\cong [\alpha A,B]_{\dgalg(H)/k}.
\]
Here, $\alpha A$ denotes the pullback of $A$ by $\alpha^*:H\to G$. In particular, if $H$ is the one-element group, 
\[
[\T(G,A),\T(H,B)]_{\acldgc}\cong [A,B]_{\dgalg/k}.
\]   
In the case $E=*$ (the terminal object of $\cldgc$), 
\[
[A,\T(H,B)_\alpha]'_{\dgalg(G)}\cong[\alpha A,B]_{\dgalg(H)}/f\sim f*g \ \ (=:[\alpha A,B]'_{\dgalg(H)}).
\]
Here $f* g$ denotes the morphism $A\ni a\mapsto g\cdot a\mapsto f(g\cdot a)\in B$ and $g$ runs through $C_{G(k)}(\alpha^*H(k))$, the centralizer of the image $\alpha^*(H(k))$ in $G(k)$. 
\end{lem}
 \begin{proof}
(1) The bijection is defined by $(F:\Tss(A)^c\to C)\longmapsto (\coc^0F,\colim_{\lambda}F_{(\uni,V_{\lam})})$. Here $\colim_{\lambda}F_{(\uni,V_{\lam})}$ denotes the following composition: $A\cong (\ring (G)_r\otimes A)^G\cong\colim_{\lam}(V_{\lam}\otimes A)^G\cong\colim\homo_{\Tss A^c}(\uni,V_{\lam})\stackrel{F_{V_{\lam}}}{\to}\colim_{\lam}\homo_C(\uni,\alpha(V_{\lam}))=C_{\alpha}$. The proof of the former claim is similar to Prop.\ref{thmbasic}. The latter part follows from an explicit description of  path objects in $\cldgc/E$ (see \cite[2.3.2]{moriya}) and Lem.\ref{lempushout}, (1). In fact, right homotopies of closed dg-categories of the above forms can be decomposed into natural transformations of $\coc^0$ and right homotopies of dg-algebras. Note that if $\alpha_1\sim\alpha_2$, $E_{\alpha_1}=E_{\alpha_2}$ and $C_{\alpha_1}\cong C_{\alpha_2}$ in $\dgalg(G)/E_{\alpha_1}$. \\
(2) We shall show morphisms $F:(\alpha^*)^*A\to \T(H,B)_\alpha\in\dgalg(G)$ are in one-to-one correspondence with morphisms $f:(\alpha^*)^*A\to B\in\dgalg(H)$. For given $F$, we set $f=\ring(\alpha^*)_* \circ F$, where $\ring(\alpha^*):\alpha(\ring(G))\to\ring(H)$ is the induced $H$-module morphism and $\ring(\alpha^*)_*:\T(H,B)_\alpha\to B$ is corresponding push-forward. In the other direction, for given $f$, We set $F=\T f:A\cong \T A(\ring(G))\to \T B(\alpha(\ring(G))=\T B_\alpha$. The construction $f\mapsto F\mapsto f$ is clearly the identity. We shall show $F\mapsto f\mapsto F$ is the identity. Let $a\in A$ and put $F(a)=\sum_iF_i(a)\otimes b_i(a)$, $F_i(a)\in\ring(G), b_i(a)\in B$. Put $\rho_A(a)=\sum_ja_j\otimes l_j$ and $\lambda_G(F_i(a))=\sum_kF_{i,k}(a)\otimes m_{i,k}(a)$, where $\lambda_{G}$ denotes the left translation, $a_j\in A$, $l_j, F_{i,k}(a),$ and $m_{i,k}(a)\in\ring(G)$. As $F$ is $G$-equivariant, $\sum_{i,j}F_i(a_j)\otimes b_i(a_j)\otimes l_j=\sum_{i,k}F_{i,k}(a)\otimes b_i(a)\otimes m_{i,k}(a)$.  The constructed $F$ is $\sum_{i,j}u\circ F_i(a_j)\cdot b_i(a_j)\otimes l_j$, where $u:\ring(G)\to k$ is the counit map, and this is equal to $\sum_{i,k}u\circ F_{i,k}(a)\cdot b_i(a)\otimes m_{i,k}(a)=\sum_iF_i(a)\otimes b_i(a)$.  This correspondence  clearly preserves right homotopy relation so the former pointed claims follow from Lem.\ref{lemcompletion2}.  For the unpointed ($E=*$) claim, similarly to the above, we have a bijection: $[A,\T(H,B)_\alpha]_\dgalg(G)\cong[\alpha A, B]_{\dgalg(H)}$. Under Tannakian duality, a (tensor-preserving) natural isomorphism $t:\alpha\Rightarrow \alpha$ corresponds an element $g$ of $C_{G(k)}(\alpha^*H(k))$ and $t_{\alpha(\ring(G))}\circ f$ does $f*g$ under the above bijection.
\end{proof}
\subsubsection{Iterated Hirsch extensions}
We shall introduce the notion of iterated Hirsch extensions of dg-algebras. The corresponding notion in the trivial group case was considered by  Sullivan, see \cite[P.279-280]{sul}. In fact, an iterated Hirsch extension is an iteration of Hirsch extensions as we see below, but this notion is useful because its classifying data directly correspond to homotopy invariants of a space. Let $l\geq 1$. Let $G$ be a reductive affine group scheme and $A\in\dgalg(G)_0$ be a connected $G$-equivariant dg-algebra. Let $(W,\eta)\in\T A$ be an object, where $W\in\ob(\rep(G))$ and $\eta$ is a MC element on $W$ (see subsection \ref{complete}), and $\alpha\in \coc^{l+1}(\T A (W,\eta))$  be a cocycle of degree $l+1$ (see the notation under Def.\ref{defofcldgc}). We define a $G$-dga $A\otimes_{(\alpha,\eta)}\bigwedge(W^\vee,l)\in\credeqalg$ as follows
\begin{enumerate}
\item As an equivariant commutative gaded algebra, 
\[
A \otimes _{(\alpha,\eta)}\bigwedge(W^\vee,l)=A\otimes _k\bigwedge(W^{\vee},l),
\]
where $\bigwedge(W^{\vee},l)$ is the free commutative graded algebra generated by the dual space $W^\vee$ with degree $l$ and regarded as an $G$-equivariant algebra with the action induced by that on $W$.

\item The differential $d$ is determined by its restrictions to $A$ and $W^\vee$.   $d|_A$ is equal to the differential of $A$. $d|_{W^\vee}$ is given by  
\[
\alpha\oplus (-^{\mathrm{t}}\eta) : W^{\vee}\longrightarrow A^{l+1}\oplus A^1\otimes_kW^{\vee}.
\]
Explicitly, if we express $\alpha$ and $\eta$ as 
\[
\eta=\sum_i  f_i\otimes a_i \ \ \  f_i\in\inhom(W,W),\ a_i\in A^1,\quad \alpha =\sum_j v_j\otimes b_j \ \ \  v_j\in W,\ b_j\in A^{l+1},
\]
we set $[\alpha \oplus (-^{\mathrm{t}}\eta)](u)=(\sum_ju(v_j)\cdot b_j, -\sum _ia_i\otimes (u\circ f_i))$.
Maurer-Cartan condition on $\eta$ and the condition $d_{\eta}\alpha=0$ ensures $d^2=0$.
\end{enumerate}
We always identify $A$ with a subalgebra of $A\otimes_{(\alpha,\eta)}\bigwedge(W^\vee, l)$ by $a\mapsto a\otimes 1$. 
\begin{defi}\label{defghe}
We use the above notations. Let $f:A\to B$ be a morphism of $G$-dga's. We say $f$ is an \textup{iterated Hirsch extension of} $A$ if there exist data $(W,\eta)\in\ob(\T A)$, $\alpha\in\coc^{l+1}(\T A(W,\eta))$, and  an isomorphism $\varphi:B\to A\otimes_{(\alpha,\eta)}\bigwedge(W^\vee,l)$ such that the following triangle is commutative.
\[
\xymatrix{A\ar[r]^f\ar[rd]^i&B\ar[d]^\varphi\\
& A\otimes_{(\alpha,\eta)}\bigwedge(W^\vee,l),}
\]
where $i$ is the natural inclusion.  We say a pair 
\[
\{\, (W,\eta),\, [\alpha]\, \} 
\]
of an object $(W,\eta)\in\ob(\T A)$ and a class  $[\alpha]\in\coh^{l+1}(\T A(W,\eta))$ is the \textup{classifying data of }$f$ (or $B$) if the above commutative triangle exists for $(W,\eta)$, $\alpha$, and for some isomorphism $\varphi$.  Classifying data is well-defined up to isomorphisms by Lem.\ref{lemghe} below. If $\eta$ can be chosen  as $0$, we say $f$ is a \textup{Hirsch extension of} $A$.  
\end{defi}
 An iterated Hirsch extension extension $A\to A\otimes_{(\alpha,\eta)}\bigwedge(W^\vee,l)$ is decomposable into a sequence of Hirsch extensions. Precisely, there exists a finite sequence of $G$-dga's
\[
A=A_0\to A_1\to\cdots\to A_m=A\otimes_{(\alpha,\eta)}\bigwedge(W^\vee,l)
\]
such that $A_{i}$ is a Hirsch extension of $A_{i-1}$ for $1\leq i\leq m$.
This is because the MC element $\eta$ can be replaced with a MC-element of the upper triangular form with zero diagonal, up to isomorphism (see Lem.\ref{lemcompletion}). In particular, if $A$ is minimal, $A\otimes_{(\alpha,\eta)}\bigwedge(W^\vee,l)$ is also minimal.\\
\indent The following is a variant of \cite[Thm.2.1]{sul} or \cite[Lem.9.3]{grimor}.
\begin{lem}\label{lemghe}
Let $G$ be a reductive affine group scheme.\\
\textup{(1)} Let $A$ be a connected $G$-dga. Let $A\otimes_{(\alpha_i,\eta_i)}\bigwedge(W_i^\vee,l)$, $i=0,1$ be two iterated Hirsch extensions of $A$. There exists an  isomorphism of $G$-dga's $\varphi:A\otimes_{(\alpha_0,\eta_0)}\bigwedge(W^\vee_0,l)\cong A\otimes_{(\alpha_1,\eta_1)}\bigwedge(W^\vee_1,l)$ fixing $A$ if and only if there exists an isomorphism  $\phi:(W_1,\eta_1)\cong (W_0,\eta_0)\in\coc^0\T A$ such that  $[\phi\circ \alpha_1]=[\alpha_0]\in\coh^{l+1}(\T A(W_0,\eta_0))$. \\
\textup{(2)} Let $\M$ be a minimal $G$-dga generated by elements of degree $\leq l$. If the $l$-th indecomposable module  $V$ of $\M$ is finite dimensional, $\M$ is an iterated Hirsch extension  of  the dg-subalgebra $\M(l-1)$ generated by elements of degree $\leq l-1$. Let $\{(W,\eta),[\alpha]\}$ be its  classifying data. We identify $V$ with a submodule of $\M^l$. As $\M$ is minimal, its differential $d|_V$, restricted to $V$, has a unique decomposition $d|_V=d^{\M(l-1)}\oplus d^{\M^1\otimes V}$ consisting of maps 
\[
d^{\M(l-1)}:V\longrightarrow \M(l-1)^{l+1},\quad d^{\M^1\otimes V}:V\longrightarrow \M^1\otimes V \in\rep^\infty(G).  
\]
 Then, 
\begin{enumerate}
\item the dual $V^\vee$ is isomorphic to $W$,
\item $d^{\M^1\otimes V}$ corresponds to the MC element $\eta$, and
\item $d^{\M(l-1)}$ corresponds to a cocycle $\alpha$ which gives the class $[\alpha]$.
\end{enumerate} 
\end{lem}
\begin{proof}
See \cite[Thm.2.1]{sul}.

\end{proof}

\subsection{de Rham functor}\label{derhamfunct}
\indent In this subsection, we define a Quillen adjoint pair 
\[
\xymatrix{
\tdr:\sset \ar@<3pt>[r]& (\dgcl)^{\op}:\langle-\rangle \ar@<3pt>[l]}
\] 
and prove that $\tdr(K)$ is a Tannakian dg-category if $K$ is connected.\\

\indent We first recall the notion of standard simplicial commutative
dga $\nabla (*,*)$ over $k$ from \cite[Section 1]{bg}. 
Let $p\geq 0$ and $\nabla  (p,*)$ be the commutative graded algebra over $k$ 
generated by indeterminates $t_0,\dots,t_p$ of degree 0 and $dt_0,\dots,dt_p$
of degree 1 with relations
\[
t_0+\cdots +t_p =1, \ \ \
dt_0+\cdots +dt_p=0.
\]

We regard $\nabla(p,*)$ as a dga with the differential given by $d(t_i):=dt_i$. We
can define simplicial operators
\[
d_i:\nabla (p,*) \rightarrow \nabla (p-1,*),\ \ \  
s_i:\nabla(p,*) \rightarrow \nabla (p+1,*),\ \ 0\leq i\leq p 
\]
(see \cite{bg}) and we also regard $\nabla(*,*)$ as a simplicial commutative dga.\\
\indent The following definition is adopted in \cite{nsrat}
\begin{defi}[\cite{nsrat}]
Let $\vect^{\iso}$ be the subcategory of $\vect$ consisting of
all objects and isomorphisms. Let $K$ be a simplicial set.
\begin{enumerate}
\item A \textup{local system} $\msc{L} $ \textup{on} $K$ is a functor 
$(\Delta K)^{op}\rightarrow \vect ^{\iso}$.
\item A \textup{morphism of local systems} $\msc{L}\rightarrow \msc{L}'$ is a 
natural transformation $I \circ \msc{L} \Rightarrow 
I\circ \msc{L}':(\Delta K)^{op}\rightarrow  \vect$, where 
$I:\vect ^{\iso}\rightarrow \vect$ is the natural inclusion functor.
\end{enumerate} 
 We define the
 \textup{tensor} $\msc{L}\otimes \msc{L}'$, the \textup{internal 
hom object} $\mf{Hom}(\msc{L},\msc{L}')$ and the \textup{coproduct} $\msc{L} \oplus \msc{L}'$
 of two local systems $\msc{L}$, $\msc{L}'$  by using those of $\vect $. For example, $(\aloc\otimes\aloc')(\sigma)=\aloc(\sigma)\otimes \aloc'(\sigma)$.
We denote by $\loc(K)$ the closed $k$-category of local systems on $K$. 
If $K$ is pointed, $\loc(K)$ is regarded as a closed $k$-category with a fiber functor.  The fiber functor $\loc (K)\rightarrow\vect$ 
is given by the evaluation at the base point.
\end{defi}
It is well-known that for a pointed connected simplicial set $K$, there exists an equivalence of closed $k$-categories $\loc(K)\stackrel{\sim}{\to}\rep(\pi _1(K))$ which is functorial in $K$. In the following, we sometimes identify $k$-local systems with representations of the fundamental group, fixing such an equivalence.
\begin{defi}
Let $K$  be a simplicial set and $\msc{L}$ be a local system on $K$. The \textup{de Rham complex of $\msc{L}$-valued 
polynomial forms} $\cdr(K,\msc{L})\in \ms{C}^{\geq 0}(R)$ is defined as follows. For each $q\geq 0$, the degree $q$ part is given by  
\[
\cdr^q(K,\msc{L} )=\lim\nolimits_{\Delta K^{op}}\nabla(*,q)\ot _k\msc{L} .
\]
Here $\nabla(*,q)$ is regarded as a functor from $\Delta K^{op}$ to the category of $k$-vector spaces by composed with the functor 
$\Delta K^{op}\rightarrow \Delta^{op}$, the limit is taken in the 
category of possibly infinite dimensional $k$-vector spaces. For $q\leq -1$, we set $\cdr^q(K,\msc{L} )=0$. The differential $d:\cdr^q(K,\msc{L} )\to \cdr^{q+1}(K,\msc{L} )$ is defined from
that of $\nabla(*,q)$.
\end{defi}
 We shall define the generalized de Rham functor
\[
\tdr:\sset\longrightarrow (\dgcl )^{\op}.
\]
This  
is a natural generalization of the de Rham functor of \cite[Definition 2.1]{bg}. 
For $K\in\sset$ we define 
a closed  dgc $\tdr(K)$ as follows. An object is a local system on $K$ and 
$\hom_{\tdr(K)}(\msc{L},\msc{L} ')=\cdr(K,\mf{Hom}(\msc{L},\msc{L} '))$. The composition is defined 
from that of $\vect$ and the multiplication of $\nabla(*,*)$, i.e.,
\[
(\eta \cdot b)\circ (\omega \cdot a):=(\eta\cdot\omega )\cdot (b\circ a)
\]
for $\omega,\eta\in\nabla(*,*)$, $a\in \mf{Hom}(\msc{L},\msc{L}')$ and $b\in \mf{Hom}(\msc{L} ',\msc{L} '')$. 
The additional structures $\otimes$, $\inhom$ and $\oplus$ are defined similarly. (We agree that)  
$\tdr (\emptyset )$ is a terminal object of $\dgcl$.  
For each morphism $f:K\rightarrow L\in\sset$ we associate a morphism of closed dgc's 
$f^*:\tdrl\rightarrow \tdrk$ by 
\[
(\Delta (L)^{op}\stackrel{\msc{L}}{\to}\vect ^{\iso})\longmapsto 
(\Delta (K)^{op}\stackrel{\Delta f^{op}}{\to}\Delta (L)^{op}\stackrel{\msc{L}}{\to}\vect^{\iso}).
\]
Thus we have defined a functor 
$\tdr:\sset\longrightarrow (\dgcl )^{op}$. \\ 
\indent Let $C\in\dgcl$. We define a functor $\langle-\rangle:(\dgcl)^{\op}\rightarrow \sset$ 
by
\[
\langle C\rangle_n=\homo _{\dgcl}(C,\tdr (\Delta ^n))
\]
with obvious simplicial operators. Clearly $\langle-\rangle$ is a right adjoint of $\tdr$.\\
\indent We define an adjoint pair between pointed categories:
\[
\xymatrix{
\tdr:\ssetp \ar@<3pt>[r]& (\dgclp)^{op}:\langle-\rangle.\ar@<3pt>[l]}
\] 
For $K\in\ssetp$, we define $\tdr(K)\in\acldgc$ by the following pullback square:
\[
\xymatrix{
\tdr(K)\ar[r]\ar[d]&\tdr(K_u)\ar[d]^{\tdr(pt)}\\
\vect\ar[r]&\tdr(*),}
\]
where $K_u$ is the unpointed simplicial set underlying $K$. For $(C,\omega_C)\in\acldgc$, we set $\langle C,\omega_C\rangle:=\langle C\rangle$ whose base point is given by $*\cong\langle\vect\rangle\stackrel{\langle\omega_C\rangle }{\to}\langle C\rangle$.
\begin{lem}[\cite{moriya}]\label{lemqpair}
The above two adjoint pairs $(\tdr,\langle-\rangle):\sset\to\cldgc$ and $(\tdr,\langle-\rangle):\ssetp\to\acldgc$ are Quillen pairs. \end{lem}
\begin{defi}[$\ared(K)$]
 Let $K$ be a pointed connected simplicial set. We set $\ared(K):=\ared(\tdr(K))\in\aredeqalg$. We always identify  the affine group scheme of $\ared(K)$ with $\pi_1(K)^\red$ (see Appendix \ref{proobj}). When we want to clarify the field of definition, we write $\ared(K,k)$.  
\end{defi}
We shall show the closed dg-category $\tdr(K)$ of a connected simplicial set $K$ is a Tannakian dg-category. Recall that there is a natural equivalence $\T A^c\to\T A$ of closed dg-categories for each equivariant dga $A$ (see sub-subsection \ref{modify}).
\begin{thm}\label{propcomplete1}
\textup{(1)} Let $K\in\sset$ (resp. $\ssetp$). 
\begin{enumerate}
\item $\tdr K$ is complete.
\item $0\to\aloc_0\to \aloc_2\to \aloc _1\to 0$ is a short exact sequence in $\loc(K)\cong\coc^0\tdr(K)$ if and only if $\aloc_0\to\aloc_2\to\aloc_1$ is an extension in $\tdr K$ in the sense of subsection \ref{complete}.
\end{enumerate}
 In particular, if $K$ is connected, $\tdr K\in\tann$ (resp. $\tannp$).\\
\textup{(2)} Let $K$ be a pointed connected simplicial set.  There exists a morphism of $\acldgc$ which is an equivalence between underlying categories:
\[
 \T\ared(K)^c\longrightarrow \tdr K.
\]
 In particular, if $f:A\to \ared(K)\in\dgalg(\pi_1^\red)$ is a quasi-isomorphism ($\pi_1=\pi_1(K)$),  finite dimensional representations $V$ of $\pi_1$ determine and are determined by objects $(W,\eta)\in\ob(\T A)$ up to isomorphisms via the following equivalences:
\[
\rep(\pi_1)\simeq\loc(K)\simeq \coc^0\tdr(K)\simeq\coc^0\T \ared(K)^c\simeq\coc^0\T\ared(K)\stackrel{f_*}{\simeq}\coc^0\T A.
\] 
 
In this correspondence,  $W$ is isomorphic to the semi-simplification of $V$. 
\end{thm}

\begin{proof}
First note that the following facts.
\begin{enumerate}
\item The completeness is stable under homotopy pullbacks and homotopy limits of towers.
\item For a dgc $C$, a sequence of chain morphisms $c_0\stackrel{a}{\to}c_2\stackrel{b}{\to}c_1$ in $C$ is an extension if and only if for any object $c\in\ob(C)$ the two sequences of complexes
\[
\begin{split}
&0\to\homo_C(c,c_0)\stackrel{a_*}{\to}\homo_C(c,c_2)\stackrel{b_*}{\to}\homo_C(c,c_1)\to 0
, \\
& 0\to\homo_C(c_1,c)\stackrel{b^*}{\to}\homo_C(c_2,c)\stackrel{a^*}{\to}\homo_C(c_0,c)\to 0
\end{split}
\]  
are both levelwise exact.

\end{enumerate}
By these  facts and small object argument, all we have to do is to prove the propositon for $X=\Delta^n,  \partial\Delta^n$ for  $n\geq 1$. 
We only show the case of $\tdr (\partial\Delta^2)$ and the others are clear.
We may replace $\partial\Delta^2$ by $S^1 =\Delta ^1/0\!\sim \!1$ and we identify local systems on $S^1$ with representation of the free group $\bb{Z}$. Let $(V,g),(V',g')\in 
\rep(\bb{Z})$.
Let $\Sigma _iP_i(t)dt\! \cdot\! f_i\in\hom _{\tdr S^1}^1((V,g),(V',g'))$ where $t=t_0$, $P_i(t)$ is a polynomial of $t$, and $f_i
\in \hom (V,V')$. Put  
\[
g_0=\left(
\begin{array}{cc}
g' &\Sigma _i\int _0^1P_i(t)dt\! \cdot \! f_ig \\
0& g
\end{array}\right).
\]    
Then the obvious sequence $(V',g')\to (V'\oplus V,g_0)\to (V,g)$ is an extension whose class is
equal to $[\Sigma _iP_i(t)dt\! \cdot\! f_i]$. 

So $\tdr(S^1)$ is complete and the third condition of Def.\ref{deftannakian} is proved by a similar argument.\\
\indent (2) follows from (1) and Prop.\ref{thmbasic}.
\end{proof}
The following lemma is used in the proof of Thm.\ref{thmspace}.
\begin{lem}[cf. Rem.4.43 of \cite{prid0}]\label{lemhtpygr}
Let $G$ be a reductive affine group scheme and $\M$ be a minimal $G$-dga. Let $V^i$ denote the $i$-th indecomposable module of $\M$. For $i\geq 2$, there exists an isomorphism of groups $\pi_i(\R\langle \T\M\rangle)\cong (V^i)^\vee$. This isomorphism is functorial about morphisms between minimal equivariant dga.
\end{lem}
\begin{proof}
By Lem.\ref{lemcompletion2} and \ref{lemmor}, $\pi_i(\R\langle \T\M\rangle)\cong [\T\M,\tdr S^i]_{\acldgc}\cong[\M,\adr S^i]_{\dgalg/k}\cong(V^i)^\vee. $
\end{proof}

\section{The de Rham homotopy theory for general spaces}\label{derhamhtpy}
In this section, we see how the minimal models describe algebraic topological invariants of spaces and provide some examples of minimal models. We also prove an equivalence between Tannakian dg-categories with subsidiary data and fiberwise rationalizations. In the proofs of results of this section, we use the correspondence between Tannakian dg-categories and schematic homotopy types (Thm.\ref{thmsht} and Cor.\ref{thmunpointedsht}). As for logical order,  section \ref{sht} is previous to this section. 
\subsection{Homotopy invariants}
We shall recall the notion of algebraically goodness  introduced by To\"en \cite{champs}. \\
\indent Let $\gam$ be a discrete  group. 
Let $\coh^i(\gam,-)$ be the $i$-th derived functor of invariants
\[
\rep^{\infty}_k(\gam)\longrightarrow k-\ms{Mod},\ \ \ \ V\longmapsto V^{\gam},
\]
and $\coh^i(\gam^{\alg},-)$ be the $i$-th derived functor of the functor
\[
\rep^{\infty}_k(\gam^{\alg})\longrightarrow k-\ms{Mod},\ \ \ \ V\longmapsto V^{\gam^{\alg}}.
\] 
 Any $\gam^{\alg}$-module can be regarded as a $\gam$-module by pulling back by the canonical map $\gam\to\gam^{\alg}(k)$ so there exists a canonical natural transformation
\[
\coh^i(\gam^{\alg},-)\Longrightarrow \coh^i(\gam,-):\rep^{\infty}_k(\gam^{\alg})\longrightarrow k-\ms{Mod}.
\]
\begin{defi}[algebraically good, \cite{champs,kpt2,prid0}]\label{defalggood}
Under above notations, we say $\gam$ is algebraically good over $k$ if for each $i\geq 0$ and  each finite dimensional representation $V\in \rep_k(\gam^{\alg})$, the canonical map
\[
\coh^i(\gam^{\alg},V)\longrightarrow \coh^i(\gam,V)
\]
is an isomorphism. 
\end{defi}
The following was proved by Pridham \cite{prid0}.
\begin{thm}[\cite{prid0}]\label{propalggood}
Let $\gam$ be a discrete  group. $\gam$ is algebraically good over $k$ if and only if  the minimal model of $\ared(K(\gam,1),k)$ is generated by elements of  degree 1.   
\end{thm}
\begin{proof}
Clearly, $\gam$ is algebraically good if and only if the canonical map $\ared(K(\gam^\alg,1))\to\ared(K(\gam,1))$ is a quasi-isomorphism (see subsection \ref{equivsht}). But by \cite[Prop.4.12]{kpt2}, for any discrete group $\Gamma$ the map $\coh^i(\ared(K(\gam^\alg,1))\to\coh^i(\ared(K(\gam,1)))$ is an isomorphism for $i=0,1$ and a monomorphism for $i=2$ and by Cor.\ref{corsht}, the minimal model of $\ared(K(\gam^\alg,1))$ is generated by degree one elements so the claim follows.
\end{proof}

\indent  We shall show how the minimal model describes homotopy theory of a space.  
\begin{thm}\label{thmfdhtpygr}
Let $K$ be a pointed connected simplicial set.  Put $\pi_i:=\pi_i(K)$. For $i\geq 2$, we regard $\pi_i$ as a $\pi_1$-module by the canonical action. Let 
\begin{enumerate}
\item $K\to\cdots\stackrel{p_i}{\to} K_{(i-1)}\stackrel{p_{i-1}}{\to}\cdots\stackrel{p_1}{\to} K_{(1)}$ be the Postnikov tower of $K$, 
\item $\M$ be the minimal model of $\ared(K)$, 
\item $\M(i)$ be the dg-subalgebra of $\M$ generated by $\oplus_{j\leq i}\M^j$,
 and 
\item $n\geq 2$ be an integer.
\end{enumerate}
Suppose  $\pi_1$ is algebraically good over $k$ and $\pi_i$ is of finite rank as an Abelian group for each $2\leq i\leq n$. \\
 \textup(1) There exists a commutative diagram in $\tannp$
\[
\xymatrix{
\T\M(1)^c\ar[r]^{\T l_1^c}\ar[d]_{q_1}&\T\M(2)^c\ar[r]^{\T l_2^c}\ar[d]_{q_2}&\cdots\ar[r]^{\T l_{n-1}^c\ }&\T\M(n)^c\ar[d]_{q_n}  \\
\tdr K_{(1)}\ar[r]^{p_1^*}& \tdr K_{(2)}\ar[r]^{p_2^*}&\cdots\ar[r]^{p_{n-1}^*}&\tdr K_{(n)}
}
\]
such that $l_i:\M(i)\to\M(i+1)$ is the inclusion and all the vertical arrows are quasi-equivalences. \\
\textup{(2)} For each $2\leq i\leq n$, the inclusion $l_{i-1}:\M(i-1)\to\M(i)$ is an iterated Hirsch extension (see Def.\ref{defghe}). Let $\{(W_i,\eta_i),[\alpha_i]\}$ be its classifying data. Then,   
\begin{enumerate}
\item The object $(W_i,\eta_i)\in\ob(\T\M(i-1))=\ob(\T\M)$ is isomorphic to the $\pi_1$-representation $\pi_i\otimes_\Z k$ under the correspondence of  Thm.\ref{propcomplete1},(2).
\item The class $[\alpha_i]\in\coh^{i+1}(\T\M(i-1)(W_i,\eta_i))$ corresponds to the k-invariant tensored with $k$ via the isomorphism:  
\[
 \coh^{i+1}[\T\M(i-1)(W_i,\eta_i)]\cong\coh^{i+1}[\tdr(K_{(i-1)})(\pi_i\otimes_{\Z}k)]\cong\coh^{i+1}(K_{(i-1)};\pi_i\otimes_{\Z}k).
\]
induced by $q_{i-1}$ in the above diagram. Here, $\pi_i\otimes_\Z k$ is considered as a local system on $K_{(i-1)}$. \end{enumerate}
\end{thm}
See  Lem.\ref{lemghe},(2).  See also \cite[Thm.1.58, Rem.4.43]{prid0}.\\
\indent Before we begin the proof, we state a corollary which may be useful for computations.
\begin{cor}\label{corhirsch}
Under the assumption of Thm.\ref{thmfdhtpygr}, suppose the action of $\pi_1$ on $\pi_n\otimes_\Z k$ is semisimple. Then, $\M(n)$ is a Hirsch extension of $\M(n-1)$. (Note that we do not assume $\M^1=0$.)
\end{cor}
To prove the theorem, we need to show a variant of the Hirsch Lemma \cite[Thm.11.1]{grimor} formulated in the following. \\ 
\indent Let $m\geq 2$. We shall consider a fibration
\[
p:E\longrightarrow B
\]
between pointed connected simplicial sets  whose  fiber $F$ satisfies $\pi_0(F)\cong *$, $\pi_i(F)=0$ for $i\leq m-1$, and $\pi_m(F)$ is an Abelian group of finite rank. Put $\pi_1:=\pi_1(B)\cong\pi_1(E)$. Let $\pi$ denote the local system of the $n$-th $k$-tensored homotopy groups of fibers of $p$ or the corresponding representation of $\pi_1(B)$. Then
\[
\coh ^{m+1}(B,E;\pi)\cong\Hom_{\rep(\pi_1(B))}(\pi _{m+1}(B,E),\pi)\cong \Hom_{\rep(\pi_1(B))}(\pi,\pi). 
\]
See \cite[P.344]{gj} or \cite[P.289]{white}. Take the element $\tilde k$ of $\coh ^{m+1}(B,E;\pi)$ corresponding to the identity on $\pi$ via this isomorphism. By definition, the image of $\tilde k$ in $\coh^{m+1}(B;\pi)$ is the k-invariant of $p$ tensored with $k$. Let $\M_B$ be a minimal model of $\ared(B)$. We may replace $\tdr (B)$ by $\T\M_B$ and we have $\coh^{m+1}(B,E;\pibar)\cong \coh^{m+1}(\cone (p^*:\cdr (B;\pibar)\to \cdr (E;\pibar))\cong \coh^{m+1}(\cone(p^*:\T\M_B(\pi^\ssim,\eta)\to \cdr (E;\pi)))$. Here, $(\pi^\ssim,\eta)$ is a pair of the semisimplification of $\pi$ and a MC-element $\eta$, see Thm.\ref{propcomplete1}. We take a cocycle $(\alpha,\beta )\in\cone ^{m+1}(p^*:\T\M_B(\pi^\ssim,\eta)\to \cdr (E;\pi))$, where $\alpha \in \T\M_B(\pi^\ssim,\eta)^{m+1}$ and $\beta\in \cdr^m(E;\pibar)$ which represents $\tilde k$.  \\
\indent Take the iterated Hirsch extension $ \M_B\otimes _{(\alpha,\eta)}\bigwedge((\pi^\ssim)^{\vee},m)$. $p$ and $\beta$ define a morphism of $\pi_1^\red$-dga's:
\[
\varrho: \M_B\otimes _{(\alpha,\eta)}\bigwedge((\pi^\ssim)^{\vee},m)\longrightarrow \ared (E)
\]
by $\varrho|_{\M_B}=p^*$ and $\rho|_{(\pi^\ssim)^{\vee}}=\beta$ (see Def.\ref{defghe}). Here $\beta$ is considered as a $\pi_1^{\red}$-module homomorphism $(\pi^\ssim)^{\vee}\longrightarrow \ared E$.\\
\indent  The essence  of the proof of the following is the same as that of \cite[Thm.11.1]{grimor}. 
\begin{lem}\label{hirschlem}
We use the above notations. \\
\textup{(1)} $\varrho$ induces an isomorphism between $i$-th cohomology groups for each $i\leq m$ and a monomorphism between $m+1$-th cohomology groups. \\
\textup{(2)} If a fiber $F$ of $p$ satisfy $\pi_i(F)=0$ for $i\geq m+1$, the map $\varrho$ is a quasi-isomorphism. 
\end{lem}
\begin{proof}
We prove (2). The proof of (1) is similar and easier. In the following, we use cubical sets instead of simplicial sets. For details about cubical de Rham theory, see Appendix \ref{derhamthm}. We use the same notation as the case of simplicial sets for the corresponding notion in the cubical case. Put $A_B:=\M_B\otimes _{(\alpha,\eta)}\bigwedge((\pi^\ssim)^{\vee},m)$. We define a descending filtration 
\[
A_B=F^0_B\supset F^1_B\supset\cdots\supset F^p_B\supset\cdots
\]
by $F^p_B=\oplus_{i\geq p}\M_B^p\otimes _k\bigwedge((\pi^\ssim)^{\vee},n)$ i.e., $F^p_B$ is the ideal of $A_B$ generated by $\oplus_{i\geq p}\M_B^p$. This filtration induces a filtration of closed dg-category consisting of 
ideals closed under tensor and internal hom: $\T A_B=\T F^0_B\supset \T F^1_B\supset\cdots\supset \T F^p_B\supset\cdots$
Note that $\T F^p_B\circ \T F^{q}_B, \ \T F^p_B\otimes \T F^{q}_B,\ \inhom (\T F^p_B,\T F^q_B)\subset \T F^{p+q}$. On the other hand, we define a filtration $\{F^p(\tdr E)\}_{p\geq 0}$ of $\tdr (E)$ by $F^p(\tdr E)$ being the $(\otimes,\inhom)$-closed ideal  generated by images of homogeneous morphism in $\T\M_B$ of degree $\geq p$, by $p^*:\tdr B\longrightarrow \tdr E$. We may describe $F^p(\tdr E)$ as follows. Let $\sigma\in E$ be a non-degenerate $l$-cube. We identify a cube with the sub-cubical set generated by it. For notational simplicity, we assume $p|_{\sigma}:\sigma\to p(\sigma)$  is the projection to the first $k$ components $q_k:\square^l\to\square^k$. Indeed, after change of coordinate, $p|_{\sigma}$ is isomorphic to $q_k$ for some $k$. Then 
\[
\omega\in \homo_{F^p(\tdr E)}(\aloc,\aloc')\iff\forall \sigma \ \omega|_{\sigma}\in\square(k,\geq p)\otimes_k\square(l-k,*)\otimes_k\inhom(\aloc,\aloc')(\sigma).
\]
Here, we identify $\square(l,*)$ with $\square(k,*)\otimes_k\square(l-k,*)$. 
Note that $F^p(\tdr E)\circ F^{p'}(\tdr E)$, $F^p(\tdr E)\otimes F^{p'}(\tdr E)$, 
$\inhom(F^p(\tdr E),F^{p'}(\tdr E))$ $\subset  F^{p+p'}(\tdr E)$. Let $\tcube(E)\in\dgc$ be a dgc defined as follows.
\begin{enumerate}
\item $\ob(\tcube(E))=\loc (E)$. 
\item $\hom_{\tcube(E)}(\aloc,\aloc')=\ccube(E;\inhom(\aloc,\aloc'))$ (see Appendix \ref{derhamthm}) and the composition is given by the cup-product.
\end{enumerate}
We define a filtration $\{F^p(\tcube(E))\}$ on $\tcube(E)$ as usual, by 
\[
\hom_{F^p(\tcube(E))}(\aloc,\aloc')=\ker(i^*:\hom_{\tcube E}(\aloc,\aloc')\to \hom_{\tcube E^{(p-1)}}(\aloc,\aloc')),
\]
where $E^{(p-1)}=p^{-1}(B^{p-1})$ and $i: E^{(p-1)}\to E$ is the inclusion ($B^{p-1}$ is  the $p-1$-th skeleton of $B$).
 $\varrho$ induces filtration-preserving morphism
\[
\T \varrho:\T A_B\longrightarrow \T\ared E\simeq \tdr E.
\]
Stokes map 
\[
\rho:\tdr E\to \tcube E
\] (this is a morphism of dg-graphs) also preserves filtration. \\
\indent In the following, for a (closed) dgc $C$ with a filtration $\{F^p(C)\}$, $E_r(C)$ denotes a (closed) dgc defined  by
\begin{enumerate}
\item $\ob E_r(C)=\ob C$.
\item $\homo_{E_r(C)}(-,-)=E_r(\hom_C(-,-);\homo_{F^{\bullet}(C)}(-,-))$, the $E_r$-term of the spectral sequence, with the differential $d_r$.
\end{enumerate}
To prove the lemma, it is enough to prove $E_2(\rho\circ \T\varrho )$ induces isomorphisms between each hom-complexes.
Note that there is a diagram 
\[
\xymatrix{E_1\tdr(E)\ar[r]^{\varphi}\ar[rd]_{E_1\rho}&\tdr (B;\mca{H}_F)\ar[d]^{\rho'}\\
&E_1\tcube E}
\]
Here, 
\begin{enumerate}
\item $\tdr (B;\mca{H}_F)$ is a closed dgc defined by
\begin{enumerate}
\item $\ob\tdr(B;\mca{H}_F)=\ob(\loc(B))$,

\item $\hom_{\tdr(B;\mca{H}_F)}(-,-)=\bigoplus_{p+q=n}\cdr^p(B;\mca{H}_F^q\otimes\inhom(-,-))$, where $\mca{H}_F^q$ is a local system on $B$ given by $\tau\mapsto \coh^q(p^{-1}(\tau(0,\dots,0),\Q)$,
\end{enumerate}
\item $\rho'$ is the Stokes map with $\mca{H}^q_F\otimes\inhom(-,-)$-coefficients ( Here, we used the well-known identification $\hom_{E_1\tcube E}^{p,q}(-,-)\cong \ccube^p(B;\mca{H}^q_F\otimes\inhom(-,-))$, and 
\item $\varphi$ is a morphism of closed dgc's  defined as follows. An element $x\in\hom_{E_0(\tdr E)}^{p,q}(-,-)$ defines a form $\omega$ on $E$ such that $\omega|_{\sigma}\in\square(k,p)\otimes_k\square(n-k,q)$ (with the above notations). Let $\tau\in B$ be a non-degenerate $k$-cube. We take a lift $\tilde{\tau}$ of the following diagram
\[
\xymatrix{F_{\tau}\ar[r]\ar[d]&E\ar[d] \\
\square^k\times F_{\tau}\ar@{-->}[ur]^{\tilde{\tau}}\ar[r]^{\tau}&B}
\]
Here $F_{\tau}=p^{-1}(\tau(0,\dots ,0))$. 
Fixing basis of $\square(k,p)$, $\alpha_1,\dots,\alpha_N$, we can write 
$\tilde{\tau}^*(\omega)=\alpha_1\otimes \beta_1+\cdots+\alpha_N\otimes\beta_N$, $\beta_i\in\cdr^q(F_{\tau};\inhom(-,-)|_{\tau})$. So if $x$ is a cocycle, it gives an element $\varphi(x)\in\cdr^p(B;\mca{H}^q_F\otimes\inhom(-,-))$ defined as $\alpha_1\otimes [\beta_1]+\cdots+\alpha_N\otimes[\beta_N]$ on $\tau$. \end{enumerate}
The Fubini's theorem ensures the diagram is commutative. It is easy to see $\coh^*\rho':\coh^*\tdr(B;\mca{H}_F)\to E_2\tcube E$ induces bijections of hom-sets, so all we have to do is to prove $\coh^*(\varphi\circ E_1\T\varrho)=\coh^*(\varphi)\circ E_2(\T\varrho):E_2(\T A_B)\to \coh^*\tdr(B;\mca{H}_F)$ is an equivalence. We prove this by using the fact that the map $\coh^*(\varphi)\circ E_2(\T\varrho)$ is a morphism of closed dgc's. 
(This is  analogous  to the fact that in simply connected case, the corresponding claim was proved by using the fact that the corresponding map is a morphism of algebras.)\\
\indent We need the following sub-lemma. We use the notation that $E^{p,q}_r(C)(X):=\hom^{p,q}_{E_r(C)}(\uni, X)$.
\begin{sublem}
Let $U^q\in\ob\T A_B$ be $\mca{H}_F^q$ regarded as an object of $\T A_B$.
\begin{enumerate}
\item There exists a bijection $E_2^{p,q}(\T A_B)(V)\cong\coh^p[\T \M_B(  (\bigwedge \pibar^{\vee})^q\otimes V)]$ for each $V\in \ob(\T A_B)$.
\item $\coh^*(\varphi)\circ E_2(\T\varrho)$ induces isomorphisms on $(p,0)$-terms  of each hom-complex for $p\geq 0$.
\item $e\in E_2(\T A_B)^{0,m}( (U^m)^{\vee})\cong\hom_{\rep(\pi_1)}(U^m,  \pibar^{\vee})$ be an element corresponding to an isomorophism, and $ev_{U^q}\in\hom^{0,0}_{E_2(\T A_B)}({U^q}^{\vee}\otimes U^q,\uni)$ be the evaluation. Then  the maps
\[
\begin{aligned}
E_2^{p,0}\T A_B( U^{km}\otimes V)&\longrightarrow E_2^{p,km}\T A_B(V)&\quad x\longmapsto& ev_{U^{km}}\circ(e^k\otimes x) \\
\coh^p[\tdr(B;\mca{H}_F^0)( \psi (U^{km})\otimes \aloc)]&\longrightarrow \coh^p[\tdr(B;\mca{H}^{km}_F)(\aloc)]&\quad y\longmapsto& ev_{\psi(U^{km})}\circ(\psi(e^k)\otimes y) 
\end{aligned}
\]
are bijections, where $k\geq1$, $e^k\in  E_2(\T A_B)^{0,km}( (U^km)^{\vee})$ is the $k$-times  $e$, and $\psi=H^*(\varphi)\circ E_2(\T \varrho)$.
\end{enumerate}
\end{sublem}
\begin{proof}[Proof of Sub-lemma]
 $E^{p,q}_0(\T A_B)(V)$ and $E^{p,q}_1(\T A_B)(V)$ are naturally isomorphic to $[((\bigwedge \pibar^{\vee})^q\otimes V)^{\ssim}\otimes\M_B^p]^{\pi_1^{\red}}$ and $d_1$ is equal to $d_{\M_B}+\eta_{(\bigwedge \pibar^{\vee})^q\otimes V}$ ($\eta_{(\bigwedge \pibar^{\vee})^q\otimes V}$ is the MC-element of $(\bigwedge \pibar^{\vee})^q\otimes V$) so the first part follows. The second part is clear. For the third part, the first map is identified with the pushforward  
 $e^k_*:\coh^p[\T \M_B((\bigwedge \pibar^{\vee})^{km}\otimes V)]\to \coh^p[\T\M_B(U^{km}\otimes V)]$ via the bijection of part 1, so this is a bijection. To see the second map, note that $\T \varrho (e)={}^te\circ\beta\in \cdr^m(E;(\mca{H}_F^m)^{\vee})$. The restriction of this element to a fiber is a cocycle in $\cdr^m(F;(\mca{H}_F^m|_F)^{\vee})$ and it represents an isomorphism in $\hom_{\rep(\pi_1)}(\mca{H}_F^m|_F,\coh^m(F))\subset \coh^m(F;(\mca{H}_F^m|_F)^{\vee})$. Thus $\psi(e)$ corresponds to an isomorphism via the identification $\coh^0[\tdr(B;\mca{H}_F^m)((\mca{H}_F^m)^{\vee})]\cong\coh^0(B;\mca{H}_F^m\otimes(\mca{H}_F^m)^{\vee})\cong\hom_{\loc(B)}(\mca{H}_F^m,\mca{H}_F^m)$ so we can see the second map is a bijection similarly to the first one. \end{proof}
As $H^*(\varphi)\circ E_2(\T \varrho)$ is a morphism of closed dgc's, the following diagram is commutative.
\[
\xymatrix{E_2^{p,0}\T A_B( U^{km}\otimes V)\ar[r]\ar[d]^{\psi}& E_2^{p,km}\T A_B(V)\ar[d]^{\psi}\\
\coh^p[\tdr(B;\mca{H}_F^0)( \psi (U^{mk})\otimes \psi V)]\ar[r]& \coh^p[\tdr(B;\mca{H}^{mk}_F)(\psi V)],}
\]
where the horizontal maps are the ones in the sub-lemma. This  implies that $\psi$ is an equivalence of categories.
\end{proof}

\begin{proof}[Proof of Thm.\ref{thmfdhtpygr}]
The proof is successive applications of Thm.\ref{propalggood} and Lem.\ref{hirschlem}. Lem.\ref{hirschlem},(1) is necessary to prove $\M(n)$ is isomorphic to the minimal model of $\ared(K_{(n)})$, see \cite[Prop.7.10]{bg}.
\end{proof}
For the first infinite higher homotopy group, we have the following.
\begin{thm}\label{thmifdhtpygr}
Let $K$ be a pointed connected simplicial set with $\pi_1(K)$ algebraically good. We use the notation of Thm.\ref{thmfdhtpygr}. Let $n\geq 2$ and suppose $\pi_i(K)$ is of finite rank as an abelian group for $2\leq i\leq n-1$. As the action of $\pi_1(K)^\red$ on $V^n$ is locally finite and as $\M$ is minimal, there exist finite dimensional $\pi_1(K)^\red$-submodules $\{V_\lam^n\}_\lam$ of $V^n$ such that $\bigcup_\lambda V_\lambda=V^n$, $d^{\M^1\otimes V^n}(V^n_\lam)\subset M^1\otimes V^n_\lam$, and $(\{V_\lam^n\}_\lam,\subset)$ forms a filtered  system.  The restriction of $d^{\M^1\otimes V^n}$ to $V_\lam^n$ defines a MC-element $\eta_\lam$ on $(V^n_\lam)^\vee$ in $\Tss \M$ (see Lem.\ref{lemghe}). Thus we obtain an inverse system of finite dimensional $\pi_1(K)$-representations corresponding to $\{((V^n_\lam)^\vee,\eta_\lam)\} $. \\

\indent Then, the limit of the inverse system is isomorphic to the pro-finite dimensional completion of $\pi_n(K)\otimes_\Z k$. Here, the pro-finite dimensional completion of (possibly infinite dimensional) $\pi_1(K)$-representation $X$ is the limit of the inverse system $\{W_\mu\in\rep(\pi_1(K))|\mu:X\to W_\mu\in\rep^\infty(\pi_1(K))\}$ taken in $\rep^\infty(\pi_1(K))$. 
\end{thm}
\begin{proof}
The proof is similar to that of Thm.\ref{thmfdhtpygr} so we omit. 
\end{proof}
\subsection{Examples}
\begin{defi}
We denote by $\mred(K)$ (resp. $\mdr(K)$) the minimal model of $\ared(K)$ (resp. $\adr(K)$). Here $\adr(K)$ is the usual polynomial de Rham algebra over $k$. When we want to clarify the field of definition, we write $\ared(K,k),\mred(K,k)$ and $\mdr(K,k)$.
\end{defi}
\begin{exa}
Let $X$ be a pointed connected simplicial set with $\pi_1(X)$ finite. Then 
$\mred(X)\cong\mdr(\widetilde{X})$, where $\widetilde{X}$ denotes the universal covering of $X$ and the action of $\pi_1$ on $\mdr(\widetilde{X})$ is induced from the one on $\widetilde{X}$ (see \cite{moriya}).
\end{exa}

The method of proof of Lem.\ref{hirschlem} provide some examples.

\begin{thm}\label{thmnilpotent}
Let $K\in\ssetp$ be a nilpotent simplicial set of finite type. Then $\mred(K)\cong\mdr(K)$. Here $\mdr(K)$ is considered as $\pi_1(K)^\red$-dga with the trivial action.
\end{thm}
\begin{proof}
It is enough to show $\mred(K(N,1))\cong\mdr(K(N,1))$ for a nilpotent group $N$. The proof is similar to that of Lem.\ref{hirschlem} and we use the tower of fibrations associated to nilpotent extensions instead of the Postnikov tower. A slight difference is that $\pi_1(B)\not=\pi_1(E)$ in this case. But in fact, for a finite dimensional $\pi_1(E)$-module $V$, the $E_2$-terms concerning $V$ is naturally isomorphic to those concerning the $\pi_1(B)$-module $V^{\pi_1(F)}$ of $\pi_1(F)$-invariants, where $F$ denotes the fiber, so the proof of Lem.\ref{hirschlem} still works. 
\end{proof}
\begin{rem}\label{remnilpotent}
If we express Thm.\ref{thmnilpotent} in the language of schematic homotopy types, there exists an isomorphism of schematic homotopy types:
\[
(L\otimes k)^\sch\cong (L\otimes k)^{uni}\times K(\pi_1(L)^\red,1)\in\ho(\sht)
\]
for nilpotent $L$ of finite type (see \cite{champs} for notation).
\end{rem}
\begin{exa}\label{exaknmn}
Suppose $k$ is algebraically closed. Let $n\geq 2$ be an integer and  $N$ be the free abelian group of rank $l$ and Let $M$ be a finite rank abelian group with $N$-action. Let $K=K(N,M,n):=K(N\ltimes  K(M,n-1),1)$. Here, $K(M,n-1)$ is an  Eilenberg Maclane space realized as simplicial abelian group with the induced action of $N$. Let $g_j\in GL(M\otimes_ {\Z}k)$ be  the action of $j$-th generater of $N$ and $g_j=g_j^s+g_j^n$ be a Jordan decomposition commutative with each other, where $g_j^n$ is nilpotent and $g_j^s$ is semisimple. Note that $\sum_{j=1,\dots, l}g_j^n\cdot s_j$ is a MC-element corresponding to $M\otimes _{\Z}k$ and the  k-invariant is zero as the section $K(N,1)\to K(N,M,n)$ exists.
 When we denote the module of $i$-dimensional generators of $\mred(K,k)$ by $V^i$, by Thm.\ref{thmfdhtpygr},
\[
V^i=
\left\{ 
\begin{array}{cc}
\bigoplus_{j=1,\dots ,l}k\cdot s_j& (i=1)\\
((M\otimes_{\Z} k)^{\ssim})^{\vee}& (i=n)\\
0   & (otherwise)
\end{array}\right.
\]   
and $d(s_j)=0$, $d(x)=\sum_{j=1,\dots, l}{}^tg_j^n(x)\cdot s_j$ for $x\in V^n$. Here, $(M\otimes_{\Z}k)^\ssim$ is the semisimplification of $N$-representation $M\otimes_\Z k$, i.e., the $j$-th generator  acts on it by $g_j^s$.  $N$ acts  on $V^1$ trivially.
\end{exa}

\begin{exa}[cell attachment]\label{exacell}
We shall give an explicit model of cell attachment which is a natural generalization of \cite[Prop.13.12]{fht}. Let $X$ be a pointed connected CW complex. Let $\pi_i:=\pi_i(X)$ ($i\geq 1$), $\mred(X)=\M =\bigwedge(V^i,d_\M)$ and $n\geq 2$.  Take $a\in\pi_n(X)$. We also denote by $a:V^n\to k$ the corresponding image by the map $\pi_n(X)\to [\tdr(X),\tdr(S^n)]\cong (V^n)^\vee$ (see Lem.\ref{lemhtpygr}). Let $X\cup_aD^{n+1}$ be the space obtained by attaching a $n+1$-cell to $X$ along $a$. A model of $X\cup_aD^{n+1}$ (i.e. a $\pi_1^\red$-dga quasi-isomorphic to $\ared(X\cup_aD^{n+1})$), $\bigwedge(V^i)\oplus_a\ring(\pi_1^\red)_l u$ is given as follows.
\begin{enumerate}
\item As a graded module, it is $\bigwedge (V^i)\oplus\ring(\pi_1^\red)_l u$, where $\ring(\pi_1^\red)_l u$ is a copy of $\ring(\pi_1^\red)_l$ whose degree is $n+1$.
\item The algebra structure is determined by that of $\M$ and $\M\cdot (\ring(\pi_1^\red)_lu)=(\ring(\pi_1^\red)_lu)^2=0$
\item The differential $d$ is determined by the derivation property from the formula
\[
dx=\left\{
\begin{array}{ll}
0& \textrm{if } x\in \ring(\pi_1^\red )_lu \\
d_\M x, & \textrm{if } x\in V^i, i\not= n\\
d_\M x+a_*(x)u & \textrm{if } x\in V^n, 
\end{array}
\right.
\]
where $a_*:V^n\to \ring (\pi_1^\red)_l$ is the composition:
$\xymatrix{V^n\ar[rr]^{\textrm{coaction of }V^n\quad}&&V^n\otimes \ring(\pi_1^\red)_l\ar[r]^{\quad a\otimes\id}&\ring(\pi_1^\red)_l .}$. 
 \end{enumerate} 
In fact, as $\pi_1(X\cup_aD^{n+1})\cong \pi_1(X)$, $\ared(X\cup_aD^{n+1})$ is isomorphic to $\ared(X)\times_{\adr(S^n)\otimes\ring(\pi_1^\red)_l}\adr(D^{n+1})\otimes\ring(\pi_1^\red)_l$. So by an argument similar to \cite[Prop.13.12]{fht} we see the latter is quasi-isomorphic to the above model (see also Lem.\ref{lemmor}).\\
\indent We shall present some concrete examples. Suppose $k$ is algebraically closed.\\
1) Let $X_0=S^1\times S^2$ and $a_0$ be a generator of $\pi_3(X_0)\cong \Z$. Let $X_1=X_0\cup_{a_0}D^4$. Using the above model, we can compute the minimal model $\M_1=\mred(X_1)$. Note that $\ring(\Z^\red)$ is isomorphic to the group ring $k\langle k^*\rangle$ of the discrete group $k^*=k-\{0\}$.  the fifth stage of $\M_1$ is presented as 
\[
\begin{split}
\M_1(5)=\bigwedge &(t,s,v_{\alpha,i},w_\alpha)_{\alpha\in k^*, i\geq 1}, \deg t=1, \deg s=2, \deg v_{\alpha,i}=4, \textrm{and }\deg w_\alpha=5 \\
\textrm{with } dv_{\alpha, 1}&=\left\{
\begin{array}{ll}
ts^2 &\textrm{for } \alpha= 1 \\
0& \textrm{otherwise}  
\end{array}
\right.,  \quad 
dw_\alpha =\left\{
\begin{array}{ll}
t^3& \textrm{for } \alpha = 1\\
tv_{\alpha,1} & \textrm{otherwise}
\end{array}
\right.\\
dv_{\alpha,i}&=tv_{\alpha,i-1} \textrm{ for all } \alpha \in  k^* \textrm{ and }i\geq 2
\end{split}
\]
Here, the fixed generator of $\pi_1\cong \Z$ acts as $v_{\alpha,i}\mapsto \alpha v_{\alpha,i}$ and $w_\alpha\mapsto \alpha w_\alpha$. In particular, we see the fourth k-invariant of $X_1$ is non-zero.
Independently, it is easy to see $\pi_4(X_1\cup_{a_1}D^{n+1})$ is the free $\pi_1$-module generated by one element, and the pro-finite dimensional completion of this has the form $(\oplus_{\alpha\in k^*}U_\alpha)^\vee$ where $\pi_1$ acts on the infinite dimensional vector space $U_\alpha =k\langle u_{\alpha,1},u_{\alpha,2},\dots\rangle$ by the infinite size Jordan block of eigen-value $\alpha$.\\
2)Next, we take an element $a_1\in \pi_4(X_1)$ and put $X_2=X_1\cup_{a_1}D^5$. We can identify $\pi_4(X_1)$ with the Laurent polynomial ring $\Z[x,x^{-1}]$, where the action of the fixed generater of $\pi_1$ corresponds the multiplication of $x$. we regard $a_1$ as a Laurent polynomial $P(x)$. Let $\phi:\pi_4(X_1)\to (\oplus_{\alpha\in k^*}U_\alpha)^\vee$ be the structure map of the completion and put $\phi_{\alpha,i}=\phi(1)(u_{\alpha,i})$ where $1\in \Z[x,x^{-1}]$. The differential of $\M_1(5)\oplus_{a_1}\ring(\pi_1^\red)u$ has the following expression. 
\[
d(\sum_{i=1}^Nc_iv_{\alpha, i})=\delta_{\alpha ,1}c_1ts^2+\sum_{i=2}^Nc_itv_{\alpha, i-1}+
(\phi_{\alpha,1},\dots ,\phi_{\alpha,N})P(A_{\alpha,N}^{-1})\left(
\begin{array}{c}
c_1\\
\vdots \\
c_N
\end{array}
\right)
\cdot \underline{\alpha}u.
\]
Here, $\delta_{\alpha,1}$ is the Kronecker delta, $A_{\alpha,N}$ is the Jordan block of size $N$ and $\underline{\alpha}$ is the element of $k\langle k^*\rangle$ corresponding to $\alpha$. So if we let $R$ be the set of non-zero distinct roots of $P(x)$ and $p_\alpha$ be the multiplicity of $\alpha\in R$, $\M_2(5)=\bigwedge (t,s,v_{\alpha,i},w_{\beta} )_{\alpha^{-1}\in R,\ 1\leq i\leq p_\alpha,\  \beta^{-1} \in R\cup\{1\}}$
with the degree and differential given by the same formula as $\M_1(5)$. It is easy to see $\pi_4(X_2)\cong \Z[x,x^{-1}]/(P(x))$ (finite rank) so we can say the profinite dimensional completion of $\pi_5(X_2)$ is the dual of $k\langle w_{\beta}\rangle_{\beta^{-1}\in R\cup\{1\}}$ (Thm.\ref{thmifdhtpygr}). 
 \end{exa}
\begin{lem}\label{lemdivisible}
Let $\gam$ be a commutative divisible group. Then the $\Q$-pro-algebraic completion $\gam^{\alg}_{\Q}$ is pro-unipotent. In particular, if $L$ is a pointed connected simplicial set with $\pi_1(L)$ commutative and divisible, $\mred(L;\Q)\cong \mdr(L;\Q)$.
\end{lem}
\begin{proof}
Left to the reader.
\end{proof}
\begin{thm}\label{thmext}
Let 
$0\longrightarrow M\longrightarrow \widetilde{\gam}\longrightarrow \gam \longrightarrow 1$
be an extension of groups such that $M$ is a finitely generated abelian group or a finite dimensional $\Q$-vector space. We regard $M$ as a $\gam$-module with the action induced by the extension. Then 
\[
\mred(K(\widetilde{\gam},1);k)\cong\mred(K(\gam,1);k)\otimes_{(\alpha, \eta)}\bigwedge[((M\otimes k^{\ssim}))^{\vee},1].
\] 
See Def.\ref{defghe} for the notation. Here, $\eta$ is the MC-element corresponding to the $\gam$-module $M\otimes k$ and $\alpha\in \coc^2\T\mred(K(\gam,1))((M\otimes k)^{\ssim},\eta)$ is a cocycle which represents the class $[\widetilde{\gam}]\otimes k\in\coh^2(\gam,M\otimes k)$. In particular, if $\gam$ is algebraically good over $k$, $\widetilde{\gam}$ is also algebraically good over $k$. (In these statements, when $M$ satisfies the second condition, we assume $k=\Q$.)\end{thm}
\begin{proof}
When we use Thm.\ref{thmnilpotent} or Lem.\ref{lemdivisible}, the proof is similar to that of Thm.\ref{thmnilpotent}.
\end{proof}

\subsubsection{components of  free loop spaces}
Let $K$ be a connected fibrant simplicial set and $\pi_1$ be the fundamental group of $K$ with respect to a fixed base point $*$. Let $\Lambda K$ be the free loop space of $K$, i.e., the internal hom-object $\inhom(S^1,K)$ in the category of unpointed simplicial sets. Choose a point $b\in S^1$. The evaluation at $b$ defines a fibration $p:\Lambda K\to K$ whose fiber is the pointed loop space $\Omega K$ of $K$. Let $\gamma\in\pi_1$. We let $\gamma$ denote its representative loop and $\Lambda_\gamma K$ do the connected component of $\Lambda K$ containing $\gamma$. $p$ induces a fiber sequence $\Omega_{\gamma}K\to \Lambda_\gamma K\to K$. Let   $\partial :\pi_i(K)\to \pi_{i-1}(\Omega_\gamma K)\cong \pi_i(K)$ be the boundary map of the long exact sequence of this fiber sequence. One can easily see for $\alpha\in \pi_i(K)$, $\partial(\alpha)$ vanishes if and only if the Whitehead product $[\gamma,\alpha]$ vanishes. So there exist  exact sequences of groups 
\[
\begin{split}
0\to& \pi_2(K)_{\gamma}\to \pi_1(\Lambda_\gamma K)\to C_{\pi_1}(\gamma)\to 1, \\
0\to&\pi_{i+1}(K)_{\gamma}\to \pi_i(\Lambda_\gamma K)\to \pi_i(K)^\gamma\to 0
\end{split}
\]
for $i\geq 2$, where $\pi_i(K)^{\gamma}=\ker(\id-\gamma)$,  $\pi_i(K)_\gamma=\coker(\id-\gamma)$ ($\gamma$ means the action on the homotopy groups), and $C_{\pi_1}(\gamma)$ denotes the centralizer of $\gamma$ in $\pi_1$. Thus, as is well-known, different components of $\Lambda K$ have different homotopy types. In this sub-subsection, we give   a model of  $\Lambda_\gamma K$ under the assumption that $\gamma$ is in the center of $\pi_1$. For the nilpotent case, the result here is already included in \cite{brszc} and we use their argument  in the proof. The difference form the nilpotent one is that the  category of unpointed schematic homotopy types is not equivalent to that of unaugmented equivariant dga's.\\
\indent    Let $\M =\mred(K,*)$.  
Let  $\Lambda\M$ be a $\pi_1^{\red}$-dga defined as follows. 
\begin{enumerate}
\item $\Lambda\M= \bigwedge(V^i, \bar V^i)=\bigwedge(V^i)\otimes\bigwedge(\bar V^i)$, where  for $i\geq 1$, $\bar V^i$ is a copy of $V^i$ with  $\deg \bar V^i=i-1$.  We identify $\M$ with a subalgebra of $\Lambda\M$ via the natural identification $:\M\supset V^i=V^i\subset \Lambda\M$.

\item We define a derivation $i:\M\to \Lambda\M$ of degree $-1$ by $i(x)=\bar x$ ($x\in V^i$, $\bar x\in \bar V^i$ is the copy of $x$), $i(xy)=i(x)y+(-1)^{\deg x}xi(y)$. Then, the differential $d_{\Lambda\M}$ on $\Lambda\M$ is defined  by $i\circ d_{\M}+d_{\Lambda\M}\circ i=0$.
\end{enumerate}
Suppose $\gamma\in\pi_1$ be in the center of $\pi_1$.  Let $\bar{\gamma}$ be the image of $\gamma$ under the map $\pi_1\to\pi_1^\alg(k)$ and $\bar{\gamma}=(u,s)$ be the decomposition such that $u\in\ru(\pi_1^\alg)(k)$, $s\in\pi_1^\red(k)$. Under the identification of Cor.\ref{corsht}, we regard $u$ as a linear map $\bar V^1\cong  V^1\to k$ Let $I\subset \Lambda\M$ be the homogeneous ideal generated by 
\[
\{x-u(x), dx\mid x\in\bar V^1\}\cup\{y-s\cdot y\mid y\in\Lambda\M\}.
\] This ideal is closed under $\pi_1^\red$-action and differential. The first part of the above set is the same as $K_u$ in \cite[P.4946]{brszc}. We define a $\pi_1(\Lambda_{\gamma}K,\gamma)^\red$-dga $\Lambda_\gamma\M$ by
\[
\Lambda_\gamma\M:=\Lambda\M/I,
\]
where the action of $\pi_1(\Lambda_{\gamma}K,\gamma)^\red$ is given by the pullback of the action of $\pi_1$ by the evaluation $\Lambda_\gamma K\to K$ at some fixed point of $S^1$. $\Lambda_\gamma \M$ is not necessarily minimal. 
The following is a non-simply connected version of \cite[Theorem]{vs}.  
\begin{prop}\label{proploop}
Suppose $k=\Q$. Let $K$ be a connected simplicial set such that $\pi_1(K)$ is algebraically good, $\pi_2(K)$ is a finitely generated abelian group or finite dimensional $\Q$-vector space, and $\pi_i(K)$ is of finite rank for each $i\geq 3$. Let $\gamma$ be an element of the center of $\pi_1(K)$.   Then, under the above notations, there exists a quasi-isomorphism $\mred(\Lambda_{\gamma}K,\gamma)\stackrel{\sim}{\to}\Lambda_{\gamma}\M$.
\end{prop}
\begin{proof}
See also Rem.\ref{remloop} below. In the following, $(-\otimes \Q)^\sch$ is abbreviated to $(-)^\sch$. Let $(K^{\sch})^{S^1}\in\sprq$ be an object given by $(K^{\sch})^{S^1}(R)=(K^{\sch}(R))^{S^1}$ for $R\in\Q-\Alg$, where the right hand side is the exponential in $\sset$ and $K^{\sch}$ is taken to be fibrant in $\sprl{\Q}$. Let $(K^\sch)^{S^1}_\gamma$ be the connected component of $(K^\sch)^{S^1}$ containing $\gamma$. In general, we have a map $f:(\Lambda_\gamma K)^\sch\to (K^\sch)^{S^1}_\gamma$ induced by the map $K\to K^\sch(k)$.  \\
\indent We first consider the case that $V=\pi_2(K)$ is a $\Q$-vector space. In this case, as $V^\alg\cong V$ by Lem.\ref{lemdivisible}, by comparing the long exact sequences associated to the fibrations $p^\sch:(\Lambda_\gamma K)^\sch\to K^\sch$, $q:(K^\sch)^{S^1}_\gamma\to K^\sch$, the evaluation at the base point of $S^1$, we can see $f$ is a weak equivalence and $\pi_1(\Lambda_\gamma K)^\red\cong\pi_1^\red$. Let $G=\pi_1^\red$. Let $(H,B,a_B)\in\aredeqalg$ be a reductive dga and $\psi:H\to G$ be a homomorphism. 
We fix three sets $H_1$, $H_2$, and $H_3$ as follows. 
\[
\begin{split}
H_1&=\{\beta\in[\psi'^*\M, \wedge\xi\otimes B]_{\dgalg(\Z^\red\times H)}| 
 a_B\circ\beta =u\in[s^*\M,\wedge\xi ]_{\dgalg(\Z^\red)} \} \\
H_2&=\{\beta'\in [\psi'^*\M,\wedge\xi\otimes B]'_{\dgalg(\Z^\red\times H)}\mid a_B\circ\beta'=u\in[s^*\M,\wedge\xi]'_{\dgalg(\Z^\red) }\}\\
H_3&=\left\{ \beta'\in [\psi'^*\M,\wedge\xi\otimes B]'_{\dgalg(\Z^\red\times H)}\left| 
\begin{split}
\text{for a representative } \tilde \beta'\in& [\psi'^*\M,\wedge\xi\otimes B] \text{ of } \beta',\\  
a_B\circ\tilde\beta'=u&\in[s^*\M,\wedge\xi]  
\end{split}
\right.\right\}
\end{split}
\]
For the notation $[-,-]'$, see Lem.\ref{lemmor} and below. (we put $E=*$).
Here, $(\Z^\red,\wedge\xi)$ is the minimal model of $\ared(S^1)$ and $\xi$ denotes a generator of degree 1, $\psi'=s\times \psi$, and $u$ is considered as a morphism $s^*\M\to \wedge\xi\in\dgalg(\Z^\red)$. 
By an argument similar to the proof of \cite[Thm.6.1]{brszc}, we see that
there exists a natural bijection:
\[
[\psi^* \Lambda_\gamma \M , B]_{\dgalg(H)}\cong H_1 \tag{*}
\]
(Here we use $H$-equivariant affine stacks\cite{kpt2} instead of rational simplicial sets in \cite{brszc}. and note that there is a natural bijection 
\[
[A,U(C)]_{\dgalg( H\times \Z^\red)}\cong [A_{\Z},C]_{\dgalg( H)}
\]
,where for an $H$-dga $A$, and an $H\times \Z^\red$-dga $C$, $A_\Z$ denotes $\Z$-coinvariants of $A$, and $U(C)$ is $C$ considered as $H\times \Z^\red$-dga with trivial $\Z^\red$-action.)\\
\indent  We shall see what universal property $\mred((K^\sch)^{S^1}_{\gamma})$ have, using Lem.\ref{lemmor}.   We work on the intermediate category whose objects are pointed schematic homotopy types but whose morphisms are those of $\ho(\sht)$, the unpointed homotopy category between underlying unpointed schematic homotopy types. $(K^\sch)^{S^1}_{\gamma}$ have a universal property as follows. Let $(Y,y)$ be a pointed schematic homotopy type. The following two sets of morphisms in $\ho(\sht)$ is naturally bijective. 
\begin{enumerate}
\item morphisms $\phi:Y\to (K^\sch)_\gamma^{S^1}\in\ho(\sht)$ such that $\pi_0(\phi(\Q)):\pi_0(Y(\Q))\to\pi_0((K^\sch)^{S^1}_\gamma(\Q))$ maps $[y]$ to $[\gamma]$.
\item morphisms $\varphi:Y\times (S^1)^{\sch}\to K^{\sch}$ such that  the composition $y\times S^1\to Y\times (S^1)^\sch(\Q)\to K^{\sch}(\Q)$ is freely homotopic to $\gamma$.  
\end{enumerate}
As $\gamma$ is in the center, this bijection gives the following bijection.
\begin{equation*}
\begin{split}
\{\alpha\in [\rep (G)^c,&\rep(H')]\mid \omega_{H'}\circ\alpha=\omega_G\in[\rep(G)^c,\vect]\} \\
&\cong\{\alpha' \in [\rep(G)^c,\rep(H'\times \Z^\alg)]\mid i_{\Z}^*\circ \alpha'=s\in[\rep(G)^c,\rep(\Z^\alg)]\},
\end{split}\tag{**}
\end{equation*}
see Lem.\ref{lemmor}. Here $[-,-]$ means $[-,-]_{\clcat}$ so "$=$" means naturally isomorphic, respecting tensors, and in the left hand side $G$ is identified with $\pi_1(\Lambda_{\gamma}K)^\red$ while it is identified with $\pi_1(K)^\red$ in the right hand side ($H'=\pi_1(Y)$). So by Lem.\ref{lemmor} and the universal property of $(K^\sch)^{S^1}_\gamma$, we get a natural bijection:
\[
[\psi^*\mred( (K^\sch)^{S^1}_\gamma),B]'_{\dgalg(H)}\cong H_2.
\]
On the other hand, (*) implies $[\psi^*\Lambda_\gamma\M, B]'_{\dgalg(H)}\cong H_3$.  Now, elements in $[s\M,\wedge\xi]$ which is identified with $u$ in $[s\M,\wedge\xi]'$ are $\{u* g|g\in G(\Q)\}$ but by assumption these are equal to $u$. So $H_2=H_3$, which implies $\Lambda_\gamma\M\simeq \mred((K^\sch)^{S^1}_{\gamma})$.\\
\indent For the case $\pi_2(K)$ finitely generated abelian, $f$ is not a weak equivalence. Let $r:K\to K_\Q$ be a fiberwise rationalization so that $\pi_1(r):\pi_1(K)\cong\pi_1(K_{\Q})$ and $\pi_i(r)\otimes \Q:\pi_i(K)\otimes \Q\cong\pi_i(K_{\Q})$.  Consider the induced map $r_* :\Lambda_\gamma K\to\Lambda_\gamma K_\Q$. By Lem.\ref{lemdivisible}, the pull-back by $\pi_1(r_*)$ preserves semi-simple $\Q$-representations so  maps $\pi_1(r_*)^\red:\pi_1(\Lambda_\gamma K)^\red\to\pi_1(\Lambda_\gamma K_\Q)^\red$ and $\ared(r_*):\ared(K_\Q)\to \ared(K)$ are induced. By \ref{lemmor}, $(\alpha_*)^*\tdr(K_\Q)\to\tdr(K)$ is equivalent to $\T \ared(\alpha_*):\T\ared(K_\Q)\to\T\ared(K)$. So by Thm.\ref{thmfdhtpygr}, Thm.\ref{thmext} and naturality of the construction of the iterated Hirsch extensions, we see $\mred(\Lambda_\gamma K_\Q)\cong\mred(\Lambda_\gamma K)$ as underlying dga's.

\end{proof}
\begin{rem}\label{remloop}
If we define a "category of dg-algberas over pro-reductive groupoids" appropriately, it is equivalent to the sub-category of $\ho(\sht)$ whose morphisms are those which preserve semisimple local systems. We can prove Prop.\ref{proploop}, using this category similarly.  To deal with more general base loops than the center,especially loops which do not preserve semisimple local systems, it will be inevitable  to consider Tannakian dgc's as the universality of the component is not contained in the semisimple subcategory.


\end{rem}

\begin{exa}\label{exaloop}
Let $K$ be a connected simplicial set  such that $\pi_1(K)=\Z$ and \\
$\M(=\mred(K))=\wedge(t,s_1,s_2,s_3,s_4,u_1,u_2,u_3)$ $\deg t=1,\deg s_i=2, \deg u_j=3$, where 
\begin{enumerate}
\item the generator $g=1\in\Z$ act on $\M$ by $g\cdot t=t$,  $g\cdot (s_i,s_{i+1})=
(s_i,s_{i+1})\left(
\begin{array}{cc}
0&1 \\
-1&0
\end{array}\right)
$ ($i=1,3$),  $g\cdot u_j=-u_j$ ($j=1,2$), $g\cdot u_3=u_3$, and 
\item the differential is given by
$ds_1=ds_2=0$, $ds_3=ts_1$, $ds_4=ts_2$,  $du_1=s_1s_2$, $du_2=2tu_1-s_1s_3-s_2s_4$, $du_3=s_1s_4-s_2s_3$.
\end{enumerate}
For example, the action of $\pi_1$ on $\pi_2\otimes\Q$ ($\pi_i=\pi_i(K)$) is as follows (see the proof of Thm.\ref{propcomplete1}).
\[
g\cdot (s_4^*,s_3^*,s_2^*,s_1^*)=(s_4^*,s_3^*,s_2^*,s_1^*)
\left(
\begin{array}{cccc}
0&1&0&-1\\
-1&0&1&0\\
0&0&0&1\\
0&0&-1&0
\end{array}\right)
\]
Here,  $(s_i^*)$ is the dual basis of $(s_i)$.\\
\textup{(i) $\gamma=e$ (the unity).} $\Lambda_e\M=\wedge(t, \bar s_i, s_i, \bar u_j, u_j)_{i=1,\dots,4,j=1,2,3}$, which is already minimal. For example, $d\bar u_2=2t\bar u_1+\bar s_1s_3+s_1\bar s_3+\bar s_2 s_3+ s_2\bar s_3$. In this case, $\pi_1(\Lambda_eK)\cong \pi_2\rtimes \pi_1$ and by the description of the minimal model, we can compute the action of $\pi_1(\Lambda_eK)$ on the homotopy group. For example, the action of $\bar s_1$ (considered as an element of $\pi_2\subset \pi_2\rtimes \pi_1$) on $\pi_2(\Lambda_eK)\otimes\Q\cong \Q\langle \bar u_3^*,\bar u_2^*, \bar u_1^*, s_4^*, s_3^*, s_2^*, s_1^*\rangle$ is given by 
\[
\left(
\begin{array}{ccccccc}
1&&&1&&& \\
&1&&&-1&& \\
&&1&&&1&\\
&&&1&&&\\
&&&&1&&\\
&&&&&1&\\
&&&&&&1
\end{array}
\right)
\]\\
\textup{(ii) $\gamma=g$.} $\Lambda_g\M=\wedge(t,\bar u_3,u_3)$ with $d=0$, which is minimal\\
\textup{(iii) $\gamma=g^2$.} $\Lambda_{g^2}\M$ is not minimal. A minimal model of $\Lambda_{g^2}\M$ is $\N_{2}:=\wedge(t,\bar u_1,\bar u_3, u'_2,u_3)$, $\deg t=1$, $\deg \bar u_1=\deg\bar u_3=2$, $\deg u'_2=\deg u_3=3$, $d=0$.
 $\N_2$ is identified with a subalgbra of $\Lambda_{g^2}\M$ whose inclusion is a quasi-isomorphism, by  $t\mapsto t$, $\bar u_j\mapsto \bar u_j$, $u'_2\mapsto u_2-t\bar u_2/2$, $u_3\mapsto u_3$.\\
 \textup{(iv) $\gamma=g^4$.} A minimal model of $\Lambda_{g^4}\M$ is $\N_4:=\wedge(t,\bar s_1,\bar s_2, s'_3, s'_4,\bar u'_1,\bar u'_3,u'_2, u'_3)$ with $\deg s'_i=\deg \bar u'_j=2$, $\deg u'_j=3$, and $ds'_3=ds'_4=d\bar u'_1=0$, $d\bar u'_3=-\bar s_1s'_4+\bar s_2s'_3$, $du'_2=du'_3=0$. $\N_4$ is identified with a subalgbra of $\Lambda_{g^4}\M$ whose inclusion is a quasi-isomorphism, by $x\mapsto x$ ($x=t, \bar s_1,\bar s_2$), $s'_i\mapsto s_i-t\bar s_i/4$ ($i=3,4$), $\bar u'_1\mapsto \bar u_1+(\bar s_1\bar s_4-\bar s_3\bar s_2)/4$, $\bar u'_3\mapsto \bar u_3-\bar s_3\bar s_4/4$, $u'_2\mapsto u_2-(\bar s_3s'_3+\bar s_4s'_4+t\bar u_2)/4$. $u'_3\mapsto u_3-t\bar u_3/4-13t\bar s_3\bar s_4/16+(\bar s_3s_4-\bar s_4s_3)/4$.
\end{exa}
\begin{rem}
Let $[\gamma]\in\coh^2(\pi_1;(\pi_2)_\gamma)$ be the class associated to the extension $0\to (\pi_2)_\gamma\to\pi_1(\Lambda_\gamma K)\to\pi_1\to 1$. When $\gamma$ is the unity, clearly $[\gamma]=0$. In the above example, for all $\gamma$, the class $[\gamma]\otimes\Q\in \coh^2(\pi_1;(\pi_2)_\gamma\otimes \Q)$ is zero. But this is not true in general, and a counterexample exists even in the nilpotent case. For example, let $\M=\wedge(t_1,t_2,t_3,s)$ with $\deg t_i=1$, $\deg s=2$, $dt_i=0$, and $ds=t_1t_2t_2$, and let $\gamma=(t_1=1,t_2=0,t_2=0)$. Then $\Lambda_\gamma\M=\wedge(t_1,t_2,t_3,\bar s, s)$ with $d\bar s=-t_2t_3$, which imply $[\gamma]\otimes \Q=[-t_2t_3]\not= 0$ (see Thm.\ref{thmext}).
\end{rem}

\subsection{Equivalence with algebraically good spaces}
In this subsection, we will show a homotopy category of Tannakian dgc's with subsidiary data is equivalent to a homotopy category of some spaces. We restrict our discussion  to the pointed case.
Let $\ssetpc$ denote the category of pointed connected simplicial sets. We say a connected pointed simplicial set $K$ is {\em algebraically good} if  $\pi_1(K)$  is algebraically good and $\pi_i(K)$ is a finite dimensional $\Q$-vector space for each $i\geq 2$. we denote by $\ssetpgd$ the full subcategory of $\ssetpc$ consisting of algebraically good spaces.
\begin{defi}
\textup{(1)} Define a category $\tannplus$ as follows.
\begin{enumerate}
\item An object is a triple $(T,\gam,\phi)$ consisting of a Tannakian dgc $T$, a discrete group $\gam$ and an equivalence of closed $k$-categories $\phi:\coc^0T\stackrel{\sim}{\to}\rep(\gam)$ such that the following diagram is commutative.
\[
\xymatrix{\coc^0T\ar[r]^{\phi}\ar[dr]_{\coc^0\omega_T}&\rep(\gam)\ar[d]^{\omega_{\gam}}\\
&\vect ,}
\]
where $\omega_{\gam}$ is the forgetful functor.
\item A morphism $(T,\gam,\phi)\longrightarrow (T',\gam',\phi')$ is a pair $F=(F,F^{gr})$ of a morphism of Tannakian dgc's $F:T\longrightarrow T'$ and a group homomorphism $F^{gr}:\gam'\longrightarrow \gam$ such that the following diagram is commutative.
\[
\xymatrix{\coc^0T\ar[r]^{\coc^0F}\ar[d]^{\phi}&\coc^0T'\ar[d]^{\phi'}\\
\rep(\gam)\ar[r]^{(F^{gr})^*}&\rep(\gam')}
\] 
\end{enumerate}
We say a morphism $F:(T,\gam,\phi)\longrightarrow (T',\gam',\phi')\in\tannplus$ is a weak equivalence if $F:T\to T'$ is a quasi-equivalence and $F^{gr}$ is an isomorphism. $\ho(\tannplus)$ denotes the localization of $\tannplus$ obtained by inverting all weak equivalences.  \\
\textup{(2)} $\tannplusgd$ denotes the full subcategory consisting of $(T,\gam,\phi)$'s such that $\pi_i(T):=[ T,\tdr S^i]_{\acldgc}$ is a finite dimensional $\Q$-vector space for $i\geq 2$ and $\gam$ is algebraically good (By Lem.\ref{lemhtpygr}, $\pi_i(T)$ is isomorphic to the dual of the $i$-th indecomposable module of the minimal model of the corresponding dga). The corresponding full subcategory of $\ho(\tannplus)$ is denoted by $\ho(\tannplusgd)$.
\textup{(3)} Let $F_0, F_1:(T,\gam,\phi)\longrightarrow (T',\gam',\phi')\in\tannplus$ be two morphisms.  We say $F_1$ and $F_2$ are { \em right homotopic} (resp. {\em left homotopic}), written $\sim_r$ (resp. $\sim_l$) if $F_1^\gr=F_2^\gr$ and $F_0$ and $F_1:T\to T'$ are right homotopic (resp. left homotopic) as morphisms of $\acldgc$.
\end{defi}
We shall show $\ho(\ssetpgd)$ and $\ho(\tannplusgd)$ are equivalent. \\
\indent We do not claim $\tannplus$ has a model category structure but we have the following lemma.
\begin{lem}\label{lemhomotopic}
Let $(T,\gam,\phi)$, $(T',\gam',\phi')\in \tannplus$. Suppose  $T$ is  cofibrant and $T'$ is  fibrant in $\acldgc$. Then, for two morphism $F_0,F_1:(T,\gam,\phi)\to(T',\gam',\phi')\in\tannplus$, $F_0\sim_r F_1$ if and only if $F_0\sim_l F_1$. $\sim_r$ (so $\sim_l$) is an equivalence relation on $\hom_{\tannplus}((T,\gam,\phi),(T',\gam',\phi'))$ and  there exists a natural bijection:
\[
\homo_{\ho(\tannplus)}((T,\gam,\phi),(T',\gam',\phi'))\cong \homo_{\tannplus}((T,\gam,\phi),(T',\gam',\phi'))/\sim_r .
\]  

\end{lem}

\begin{proof}
Note that if $F_0, F_1:T\to T'\in\tannp$ are right or left homotopic in $\acldgc$, the induced morphisms $(\coc^0F_0)^*, (\coc^0F_1)^*:\aut^\otimes(\omega_{\coc^0T'})\to \aut^\otimes(\omega_{\coc^0T})$ are equal. By using this fact, we see right or left homotopic morphisms represent the same morphism in $\ho(\tannplus)$. So the lemma is standard, see \cite{quiha} or \cite[Thm.1.2.10]{hov}. 
\end{proof}
\begin{rem}
The author does not know whether the statement corresponding to Lem.\ref{lemhomotopic} in the unpointed case is true. In this case, we will consider groupoids instead of groups.  As morphisms of groupoids have non-trivial homotopic relation, if we define a right or left homotopic relation similarly, the proof does not go on similary.
\end{rem}
Define two functors
\[
\tdr^+:\ssetpc\longrightarrow (\tannplus)^{\op}, \ \ \ \R\langle-,-,-\rangle:(\tannplus)^{\op}\longrightarrow \ssetpc
\]
as follows. \\
\indent For $K\in\ssetpc$, $\tdr^+(K)=(\tdr(K),\pi_1(K),\phi_K)$, where $\phi_K$ is a fixed functorial equivalence $\coc^0(\tdr(K))\stackrel{\sim}{\to}\loc(K)\stackrel{\sim}{\to}\rep(\pi_1(K))$. For $(T,\gam,\phi)\in \tannplus$, $\R\langle T,\gam,\phi\rangle$ is the following pullback.
\[
\xymatrix{\R\langle T,\gam, \phi\rangle\ar[r]\ar[d]&\langle QT\rangle\ar[d]^{p}\\
B\gam \ar[r]^{\phi*}&B\Pi_1\langle QT\rangle}
\]
Here, $QT$ is a fixed functorial cofibrant replacement of $T$, $\Pi_1\langle QT\rangle$ is the fundamental groupoid,  $p$ is the canonical map which gives an isomorphism of fundamental groupoid, and $\phi^*$ is the following composition.
\[
B\gam\to B\aut^{\otimes}(\omega_{\gam})\stackrel{\phi^*}{\cong} B\aut^{\otimes}(\coc^0\omega_{QT})\cong B\pi_1\langle QT\rangle \to B\Pi_1\langle T\rangle ,
\]
see Appendix \ref{proobj}, where the first map is induced by the natural "evaluation" $\gam\to \aut^\otimes(\omega_\gam)$ and the inverse of the isomorphism $\aut^\otimes(\coc^0\omega_T)\cong \pi_1\langle QT\rangle$ is the following composition:
\[
\pi_1\langle QT\rangle\cong[QT,\tdr S^1]_{\acldgc}\stackrel{\coc^0}{\to}[\coc^0(QT), \rep(\Z)]_{\clcatk}\cong\aut^\otimes(\coc^0\omega_{QT}),
\] 
where the last isomorphism sends an action of $1\in\Z$ to an automorphism. This is in fact, an isomorphism by Thm.\ref{thmchamps} and \ref{thmsht}. Note that $\tdr^+$ preserves weak equivalences so it induces a functor $\tdr^+:\ho(\ssetpc)\longrightarrow \ho(\tannplus)^{\op}$.  As the map $p:\langle QT\rangle \longrightarrow B\Pi_1\langle QT\rangle$ is a fibration, the above pullback  is a homotopy pullback. So $\R\langle-,-,-\rangle$ also preserves weak equivalences and it defines a functor between homotopy categories.
 We have an obvious natural map
\[
\Phi:\homo_{\ssetpc}(K, \R\langle T,\gam,\phi\rangle )\longrightarrow \homo_{\tannplus}((QT,\gam,\phi),\tdr^+(K)).
\]
We can easily see if $K$ is reduced, i.e., $K_0\cong *$, this map is a bijection and preserves and reflects homotopy relation so $\Phi$ induces a bijection
\[
\Phi:\homo_{\ho(\ssetpc)}(K, \R\langle T,\gam,\phi\rangle )\longrightarrow \homo_{\ho(\tannplus)}((T,\gam,\phi),\tdr^+(K)) 
\]
by Lem.\ref{lemhomotopic}. Thus the pair $(\tdr^+,\R\langle -,-,-\rangle)$ is an adjoint pair.

\begin{thm}\label{thmspace}
\textup{(1)} The above adjoint pair induces an equivalence between full subcategories :
\[
\xymatrix{
\tdr^+:\ho(\ssetpgd) \ar@<3pt>[r]^{\sim \qquad  }& \ho(\tannplusgd)^{\op}:\R\langle -,-,-\rangle.\ar@<3pt>[l]}
\]
\textup{(2)} Let $K$ be a connected pointed simplicial set whose fundamental group is algebraically good and whose higher homotopy groups are of finite rank as Abelian groups. Then, the unit 
$K\to \R\langle\tdr(K),\pi_1(K),\phi_K\rangle$ is  the fiberwise rationalization.
\end{thm}
\begin{proof}
(1) For $K\in\ssetpgd$, we shall show the unit map $u_K:K\longrightarrow \R\langle\tdr^+(K)\rangle$ is a weak equivalence. we use the notation of Thm.\ref{thmfdhtpygr}. By definition, the map $\pi_1(u_K):\pi_1(K)\to\pi_1(\R\langle\tdr^+(K)\rangle)$ is an isomorphism. For $i$-th homotopy group ($i\geq 2$), we may assume $K$ is weak equivalent to the $i$-th level of the Postnikov tower of $K$. Consider the fiber sequence $F\stackrel{f}{\to} K\stackrel{p_{i-1}}{\to}K_{(i-1)}$, where $F$ is the fiber of type $K(\pi_i(K),i)$. By the proof of Lem.\ref{hirschlem}, the corresponding morphism $f^*:\tdr K\to \tdr F$ is equivalent to the morphism $\T(\pi_1^\red,\M)\to\T(e,\bigwedge(\pi_i^\vee,i))\in\tannp$ induced by a morphism $ (\pi_1^\red, \M)\to (e,\bigwedge(\pi_i^\vee,i))\in\aredeqalg$ which maps $V^j$ to $0$ ($j< i$) and $V^i$ to $\pi_i^\vee$ isomorphically.  So by Lem.\ref{lemhtpygr}, $\pi_i(\R\langle f^*\rangle):\pi_i(\R\langle \tdr F\rangle)\to\pi_i(\R\langle \tdr K\rangle)$ is an isomorphism. As the unit $F\to \R\langle \tdr F\rangle$ is a weak equivalence by the classical Sullivan's theory, we see the map $\pi_i(K)\to\pi_i(\R\langle \tdr K\rangle)$ is an isomorphism. This implies $\pi_i(u_K)$ is an isomorphism. Thus, $u_K$ is a weak equivalence and $\tdr^+$ is fully faithful.\\
\indent By a similar argument, we see a map $F:(T,\gam,\phi )\longrightarrow (T',\gam',\phi')\in\tannplusgd$ is a weak equivalence  if the induced map $\R F: \R\langle T,\gam,\phi \rangle\longrightarrow \R\langle T',\gam',\phi'\rangle$ is a weak equivalence so $\tdr^+$ is essentially surjective. One can prove (2) by a similar argument.  
\end{proof}

\section{Correspondence with schematic homotopy types}\label{sht}
\subsection{Definitions}\label{definitions}
We recall the notion of schematic homotopy types from \cite{champs, kpt2}. Let $R-\Alg$ denote the category of (discrete) commutative unital $R$-algebras for a commutative unital $k$-algebra $R$.
\begin{defi}\label{defsht}
\textup{(1)} $\sprk$ denotes the \textup{category of simplicial presheaves on} $\kaff:=(\kalg)^{\op}$. In other words, $\sprk$ is the category of functors $\kalg\longrightarrow \sset$ and natural transformations. 
\begin{enumerate}
\item For a simplicial presheaf $X\in\ob(\sprk)$ and a $R$-point $s\in X(R)_0$, ($R\in\kalg$) the \textup{$i$-th homotopy presheaf $\pi^{pre}_i(X,s)$ with base point $s$} is a presheaf on $\aff_R=(R-\Alg)^{\op}$ defined by $\pi^{pre}_i(X,s)(S):=\pi_i(X(S),u(s))$ for $S\in R-\Alg$. Here, $u(s)$ denotes  the image of $s$ by the unit $u: R\to S$, and if  $i=0$, we ignore $s$ and suppose $R=k$. The \textup{$i$-th homotopy sheaf $\pi_i(X,s)$ with base point $s$} is the sheafification of $\pi_i^{pre}(X,s)$ with respect to the faithfully flat quasi-compact topology on $\aff_R$. We say $X\in\sprk$ is connected if the $0$-th homotopy sheaf $\pi_0(X)$ is one point sheaf.
\item We regard $\sprk$ as a model category via the \textup{objectwise projective model structure}, denoted by $\sprkobj$, where a morphism $f:X\to Y\in\sprk$ is a weak equivalence (resp. a fibration) if and only if $f_R:X(R)\to Y(R)\in\sset$ is a weak equivalence (resp. fibration) of simplicial sets for any $R\in k-\Alg$. We call the above weak equivalence an \textup{objectwise equivalence}.
\item We also consider the \textup{local projective model structure \cite{blander,dhi} on} $\sprk$ , denoted by $\sprkl$, where weak equivalences and cofibrations are defined as follows. 
\begin{enumerate}
\item A morphism $f:X\to Y$ is a weak equivalence, called a  \textup{local equivalence}, if and only if it induces  isomorpisms of $0$-th homotopy sheaves $\pi_0(X)\cong \pi_0(Y)$ and of $i$-th homotopy sheaves $\pi_i(X,s)\cong \pi_i(Y,f(s))$ for any $i\geq 1$, any $R\in\kalg$, and any base point $s\in X(R)_0$.
\item A morphism is a cofibration if and only if it is a cofibration in the objectwise projective model structure. \end{enumerate}
We say an object $X\in\sprk$ is a local object if a fibrant replacement $X\to RX$ in $\sprkl$ is an objectwise equivalence.

\item We set $\sprkp:=*/\sprk$, $\sprkobjp:=*/\sprkobj$, and $\sprklp:=*/\sprkl$.
\end{enumerate}
\textup{(2)}(\cite{champs,kpt2}) We say a pointed simplicial presheaf $X\in\sprkp$ is a \textup{pointed schematic homotopy type} if the following  conditions are satisfied. 
\begin{enumerate}
\item $X$ is a connected local object. 
\item The homotopy sheaf $\pi_i(X,*)$ is represented by an affine group scheme over $k$ for any $i\geq 1$ and  it is unipotent for $i\geq 2$.
\end{enumerate}
We say a pointed schematic homotopy type $X\in\sprkp$ is a \textup{pointed affine stack} if $\pi_1(X,*)$ is represented by a unipotent affine group scheme over $k$. We denote the full subcategory of $\sprkp$ consisting of pointed schematic homotopy types by $\shtp$.  $\ho(\shtp)$ denotes the corresponding full subcategory of $\ho(\sprkobjp)$ (or $\ho(\sprklp)$). We say an object $U\in\sprk$ is a \textup{schematic homotopy type} if there exists a point $*\in U(k)_0$ such that the pointed object $(U,*)\in\sprkp$ is a pointed schematic homotopy type. We denote by $\sht$ the full subcategory of $\sprk$ consisting of schematic homotopy types.\\
\textup{(3)}(\cite{champs}) Let $K\in\sset$ (resp. $\ssetp$) be a connected simplicial set. We denote by $\underline{K}$ be a constant simplicial presheaf given by $\underline{K}(R)=K$ for $R\in k-\Alg$. The \textup{schematization of} $K$ is a morphism 
\[
\underline{K}\longrightarrow (K\otimes k)^\sch\in\ho(\sprkobj)
\] (resp. $\ho(\sprkobjp)$) which is initial in the category $\underline{K}/\ho(\sht)$ (resp. $\underline{K}/\ho(\shtp)$). Here, $\underline{K}/\ho(\sht)$ is the full subcategory of the over category $\underline{K}/\ho(\sprk)$ consisting of objects $\underline{K}\to X$ such that $X$ is a schematic homotopy type. $\underline{K}/\ho(\shtp)$ is similar. 
\end{defi}
The above definition of pointed schematic homotopy types is different from the original one but they are shown to be equivalent, see \cite[Cor.3.16]{kpt2}. The existence of the schematization in the pointed category was proved in \cite{champs} (we will reprove this later, based on the above definition). The existence of the unpointed schematization follows from Cor.\ref{thmunpointedsht}. \\
\indent We shall define an adjoint pair $(\tdr,\langle-\rangle):\sprk\to(\cldgc)^\op$ (or $\sprkp\to(\acldgc)^\op$). This is an analogue of the adjoint pair $(\tdr,\langle-\rangle):\sset\to(\cldgc)^\op$ defined in subsection \ref{derhamfunct}. \\
\indent We first define a notion of local systems on a simplicial presheaf. In order to ensure some functoriality, we take care about the definition of the category of projective modules. The following definition of (1) can be found in \cite[1.3.7]{hag2}, for example. Let $R-\Mod$ denote the category of $R$-modules and $R$-linear maps.
\begin{defi}
\textup{(1)} For $R\in k-\Alg$ we define a category $R-\proj$ as follows.
\begin{enumerate}
\item An object is a pair $(M_-,\phi_-)$ consisting of a function $M_-$ which assigns each $R$-algebra $S$ a finitely generated projective $S$-module $M_S$ and a function $\phi_-$ which assigns each morphism $f:S\to S'\in R-\Alg$ an isomorphism $\phi_f:M_S\otimes _SS'\to M_{S'}\in S'-\Mod$ such that for each sequence $S\stackrel{f}{\to}S'\stackrel{f'}{\to}S''\in R-\Alg$ the following diagram is commutative.
\[
\xymatrix{(M_S\otimes_S S')\otimes _{S'}S''\ar[r]^{\qquad \phi_f\otimes S''} \ar[d]&M_S'\otimes _{S'}S''\ar[d]^{\phi_{f'}}\\
M_S\otimes_S S''\ar[r]^{\phi_{f'\circ f}}&M_S''}
\]
Here, the left vertical arrow is induced from the associativity isomorphism and unit isomorphism. Sometimes we omit $\phi$.
\item A morphism $F:(M,\phi)\to (N,\varphi )$ is a function assigning each $R$-algebra $S$ a morphism $F_S:M_S\to N_S\in S-\Mod$ such that for any $f:S\to S'\in R-\Alg$ the following diagram is commutative
\[
\xymatrix{M_S\otimes _SS'\ar[r]^{F_S\otimes S'}\ar[d]^{\phi_f}&N_S\otimes_SS'\ar[d]^{\varphi_f}\\
M_{S'}\ar[r]^{F_{S'}}&N_{S'}}
\]
\end{enumerate}
We regard $R-\proj$ as a closed $k$-category in the obvious way. For example the internal hom $\inhom (M, N)$ is given by $\inhom(M,N)_S=\inhom(M_S,N_S)$, the right hand side is the internal hom in $S-\Mod$. For $f:R\to R'\in \kalg$ we have a morphism $f_*:R-\proj\to R'-\proj\in R-\clcat$ defined by $f_*(M)_S'=M_{f^*S'}$, where $S'\in R'-\Alg$. Note that $f'_*\circ f_*=(f'\circ f)_*$ \\
\textup{(2)} Let $X$ be an object of $\sprk$. Let $R-\proj^{\iso}$ be the full subcategory of $R-\proj$ consisting of all objects and all isomorphisms. \\
A \textup{local system} $\msc{L} $ \textup{on} $X$ is a collection $\{\aloc_R\}_{R\in\kalg}$ of functors $\aloc_R:[\Delta (X(R))]^{\op}\to R-\proj^{\iso}$ such that for each morphism $f:R\to R'\in\kalg$ the following diagram is commutative.
\[
\xymatrix{[\Delta (X(R))]^{\op}\ar[r]^{\aloc_R}\ar[d]^{(\Delta X_f)^\op}&R-\proj\ar[d]^{f_*} \\
[\Delta (X(R'))]^{\op}\ar[r]^{\aloc_{R'}}&R'-\proj}
\]
A \textup{morphism} $\alpha:\aloc\to \aloc'$ \textup{of local systems} is a collection $\{\alpha_R\}$ of natural transformations $\alpha_R:\aloc_R\Rightarrow \aloc'_R:[\Delta (X(R))]^{\op}\to R-\proj$ compatible with ring homomorphisms. We denote by $\loc(X)$ the category of local systems on $X$. We regard $\loc(X)$ as a closed $k$-category in the obvious way.

\end{defi}
Let $X\in\sprkp$ be a connected object such that $\pi_1(X,*)$ is represented by an affine group scheme $G$ over $k$. It is easy to see $\loc(X)$ is equivalent to  the category of finite dimensional representations of $G$.  
\begin{defi}
Let $X$  be a simplicial presheaf and $\msc{L}$ be a local system on $X$. The \textup{de Rham complex of $\msc{L}$-valued 
polynomial forms} $\cdr(X;\msc{L})\in \ms{C}^{\geq 0}(k)$ is defined as follows. For each $q\geq 0$, the degree $q$ part is the submodule of 
\[
\prod_{R\in k-\Alg}\, \prod_{p\geq 0}\, \prod_{\sigma\in X(R)_p}\nabla(p,q)\otimes_k\aloc_R(\sigma)_R
\]  
consisting of $\{\omega_{\sigma}\}_{R\in\kalg, \sigma\in \Delta X(R)}$'s such that 
\begin{enumerate}
\item for each  algebra $R\in\kalg$ and  morphism $ a:\tau\rightarrow \sigma \in \Delta X(R)$, $a^*\omega_{\sigma}=\omega_{\tau}$, and 
\item for each map $f:R\to R'\in\kalg$ and simplex  $\sigma\in X(R)$, $(\id \otimes_kf_*)(\omega_\sigma)=\omega_{X_f\sigma}$.

\end{enumerate}
Here,  $a^*$ denotes the tensor of $a^*:\nabla(|\sigma|,q)\to \nabla(|\tau|,q)$ and $\aloc_R(a):\aloc_R(\sigma)\to \aloc_R(\tau)$ ($|\cdot|$ denote the dimension of a simplex), and $(\id \otimes_kf_*)(\omega_\sigma)$ denotes the image of $\omega_\sigma$ by the tensor of the identity of $\nabla(|\sigma|,q)$ and the following composition.
\[
\aloc _R(\sigma)_R\to\aloc_R(\sigma)_R\otimes_RR'\stackrel{\phi_f}{\to}\aloc_R(\sigma)_{R'}=\aloc_{R'}(X_f\sigma).
\]  
  For $q\leq -1$, we set $\cdr^q(X,\msc{L} )=0$. The differential $d:\cdr^q(X,\msc{L} )\to \cdr^{q+1}(X,\msc{L} )$ is defined from
that of $\nabla(*,q)$.
\end{defi}
We  define a functor
\[
\tdr:\sprk\longrightarrow (\dgcl )^{\op}
\]
by setting 
\[
\ob (\tdr(X))=\ob (\loc(X)),\qquad
\homo_{\tdr(X)}(\aloc,\aloc')=\cdr(X;\inhom(\aloc,\aloc'))
\]
for $X\in\sprk$. The composition and the closed tensor structure are defined in the obvious way. \\
\indent A right adjoint 
\[
\langle-\rangle:(\cldgc)^\op\longrightarrow \sprk
\] 
of $\tdr$ is defined by 
\[
\langle C\rangle (R)_n:=\homo_{\cldgc}(C,\tdr(h_R\times\Delta^n))
\]
for $C\in\cldgc$, $n\geq 0$  and  $R\in k-\Alg$. Here $h_R$ denotes the Yoneda embedding of $Spec R$, and $h_R\times \Delta^n$ is a simplicial presheaf given by $h_R\times \Delta^n(S):=h_R(S)\times \Delta^n\in\sset$.
For a pointed simplicial presheaf $X\in\sprkp$, we define a closed dgc with a fiber functor $\tdr(X)\in\cldgck$ by the following pullback square: 
\[
\xymatrix{\tdr(X)\ar[r]\ar[d]&\tdr(X_u)\ar[d]^{\tdr(pt)}\\
\vect\ar[r]&\tdr(*),}
\]
where $X_u\in\sprk$ denotes the underlying unpointed simplicial presheaf. On the other hand, for an object $C\in\cldgck$, $\langle C\rangle$ is naturally pointed as $\langle \vect \rangle\cong *$ so we have an adjoint pair $(\tdr, \langle-\rangle):\sprkp\to(\acldgc)^\op$. For connected object $X\in\sprkp$, $\tdr(X)$ is equivalent to $\tdr(X_u)$ as closed dgc's.

\begin{lem}\label{lemqpair2}
The adjoint pairs $(\tdr,\langle-\rangle):\sprkobj\to(\cldgc)^\op$ and $(\tdr,\langle-\rangle):\sprkobjp\to(\acldgc)^\op$
are  Quillen pairs. \end{lem}
\begin{proof}
When we use Prop.\ref{propderham2}, the proof is similar to that of Lem.\ref{lemqpair}. Note that for a local system $\aloc$ on $h_R\times \Delta^n$, $\cdr(h_R\times \Delta^n,\aloc')\cong \cdr(\Delta^n;\aloc_R|_{\id_R\times\Delta^n})$.
\end{proof}

\subsection{Equivalence between schematic homotopy types and Tannakian dg-categories}\label{equivsht}
In  this subsection, we prove the following theorem. Note that for connected $K\in\ssetp$ $\L\tdr(\underline{K})$ is naturally (quasi-)equivalent to $\tdr(K)$ which is defined in \ref{derhamfunct}.
\begin{thm}[cf. Cor.3.57 of \cite{prid0}]\label{thmsht}
\textup{(1)} The derived adjunction $(\L\tdr,\R\langle-\rangle):\ho(\sprkobjp)\longrightarrow\ho(\dgclp)$ induces an equivalence between full subcategories:
\[
\xymatrix{
\bb{L}\tdr:\ho(\shtp) \ar@<3pt>[r]^{\sim}& \ho(\tannp)^{\op}:\R\langle -\rangle.\ar@<3pt>[l]}
\]
\textup{(2)} For connected pointed simplicial set $K$, the unit of the adjunction
\[
\underline{K}\longrightarrow \R\langle \tdr(K)\rangle\in\ho(\sprkobjp)
\]
is the schematization over $k$.
\end{thm}
We prove this using the following theorem of \cite{champs}.
\begin{thm}[Thm.0.0.2 and Cor.2.2.3 of \cite{champs}]\label{thmaffine}
Let $\calg/k$ be the category of commutative and unital cosimplicial $k$-algebras with augmentations. 
Let $\ring :\sprklp\to (\calg/k)^{\op}$ be the functor such that  for a simplicial presheaf $X\in \sprklp$, $\ring(X)^n=\hom(X_n,\ring_0)$, where $\ring_0$ is the tautological presheaf $\kalg\ni R\mapsto R\in\set$ and $\hom$ is the set of morphisms between presheaves.
Then $\ring$ is a left Quillen functor  and its derived functor $\L\ring :\ho(\sprklp)\to \ho(\calg/k)^\op $ induces an equivalence between the full subcategory of affine stacks and one of connected (i.e.$\coh^0\cong k$) cosimplicial algebras. 
\end{thm}
This theorem can be formulated using dg-algebras instead of cosimplicial algebras as follows.
Let $\adr:\sprklp\to (\dgalg/k)^\op$ be the functor defined by $\adr(X)=\cdr(X,\uni)$. This is a left Quillen functor whose right adjoint $|-|$ is given by $|A|(R)_n:=\hom_\dgalg(A,\adr(h_R\times \Delta^n))$
It is well-known that the Thom-Sullivan cochain functor $Th:\calg/k\to \dgalg/k$ induces equivalence of homotopy categories, and it is easy to see the following diagram commutative up to natural quasi-isomorphism.
\[
\xymatrix{
\sprklp\ar[r]^{\ring}\ar[rd]^{\adr}&(\calg/k)^\op\ar[d]^{Th}\\
  &(\dgalg/k)^\op.}
\]
So we may rephrase Thm.\ref{thmaffine} that the functor $\L\adr:\ho(\sprklp)\to\ho(\dgalg/k)^\op$ induces an equivalence between affine stacks and 0-connected dga's.\\   
\indent We shall state some fundamental results.
\begin{defi}
For a pointed schematic homotopy type $X$, we put $\ared(X):=\ared(\L\tdr(X))$
\end{defi}
\begin{prop}\label{lemiroiro}
\textup{(1)} Let $X\in\sprkp$. $\L\tdr X$ is complete. If $X$ is connected, $\L\tdr X$ is a Tannakian dg-category with a fiber functor. $\L\tdr X$ and $\T\ared (X)$ are equivalent. \\
\textup{(2)} For $T\in\tannp$, let $T^{\ssim}$ denote the full sub dgc of $T$ consisting of semisimple objects of $\coc^0T$. The inclusion $T^{\ssim}\longrightarrow T$ induces an isomorphism $\R\langle T\rangle\longrightarrow \R\langle T^{\ssim}\rangle\in \ho(\sprkobjp)$. \\
\textup{(3)} For any object $T\in\tannp$, $\R\langle T\rangle\in\ho(\sprkobjp)$ is connected.\\
\textup{(4)} Let $X, Y\in\sprkp$ be connected objects which is cofibrant in $\sprkobjp$. For a local equivalence $f:X\to Y$, the corresponding morphism $f^*:\tdr Y\to\tdr X$ is a quasi-equivalence. \\
\textup{(5)} Let $F\in\sprkp$ be a pointed affine stack. The unit of the adjoint
\[
F\longrightarrow\R\langle\bb{L}\tdr(F)\rangle\in\ho(\sprkobjp)
\]
is an isomorphism.
\end{prop}
\begin{proof}
The proof of (1) is similar to that of Thm.\ref{propcomplete1}. For (2), by (1) and Lem.\ref{lemcompletion2}, for any $L\in\sprk$,
\[
[L,\R\langle T^{\ssim}\rangle]_{\sprk}\cong [T^{\ssim},\bb{L}\tdr(L)]_{\cldgc}\cong[T,\bb{L}\tdr(L)]_{\cldgc}\cong[L,\R\langle T\rangle]_{\sprk}.
\]
(3) follows from \cite[Thm.3.2]{dmos}. (4) follows from Prop.\ref{propderham2} as invariance of $\cspl(-;\aloc)$ under local equivalence between connected objects are well-known (see e.g.\cite{champs}).
 (5) follows from the reformulation of Thm.\ref{thmaffine} and (2) as $\R\langle\L\tdr F\rangle\cong\R\langle\Tss\adr (QF)\rangle\cong\R |\adr (QF)|$ $\cong F$.
\end{proof}
For an affine group scheme $G$, define a simplicial presheaf $K(G,1)\in\sprkp$ by $K(G,1)(R):=K(G(R),1)$.
\begin{lem}\label{lemfiberseq}
\textup{(1)} Let $G$ be a reductive affine group scheme. Let $K\in\sprkp$ be a fibrant model of $K(G,1)$ in $\sprklp$. The unit of the adjoint 
\[
K\longrightarrow\R\langle\L\tdr(K)\rangle \in\ho(\sprkobjp)
\]
is a local equivalence. (Note that local equivalences are stable under objectwise equivalences so we apply the notion to the morphism of $\ho(\sprkobjp)$)\\
\textup{(2)} Let $X$ be a pointed connected simplicial presheaf such that $\pi_0(X)\cong *$ and $\pi_1(X)$ is represented by an affine group scheme $G$. Suppose $X$ is fibrant in $\sprklp$. Let $K\in\sprqp$ be a fibrant model of $K(G^{\red},1)$ in $\sprklp$ and $p:X\longrightarrow K\in\sprkp$ be a fibration such that 
$\pi_1(p):\pi_1(X)\longrightarrow\pi_1(K)$ is isomorphic to the canonical map $G\longrightarrow G^{\red}$. Let $F\in\sprkp$ be the fiber of $p$ at the base point. Then the sequence 
\[
\R\langle \L\tdr(F)\rangle\longrightarrow \R\langle \L\tdr (X)\rangle\longrightarrow\R\langle \L\tdr (K)\rangle
\]
is a homotopy fiber sequence in $\sprkobjp$.
\end{lem}
\begin{proof}
Let $QF$ be a cofibrant replacement of $F$. We first show $\ared(QF)$ ($\cong \cdr(QF;\mb{1})$) is quasi-isomorphic to $\ared(X)$ with trivial group action as a dga.
As an explicit model of $F$ we use $\widetilde X$ defined as follows. We denote by $Fib(\loc (X)^{\ssim})$ the presheaf of groupoid of fiber functors of $\loc (X)^{\ssim}$, the category of semi-simple local systems on $X$. Precisely, for an algebra $R\in k-\Alg$, An object of $Fib(\loc(X)^\ssim)(R)$ is a morphism $\omega:\loc(X)^\ssim\to R-\proj\in \clcat$. and a morphism is a natural transformation preserving tensors. By the definition of $R-\proj$ the correspondence $R\mapsto   Fib(\loc(X)^\ssim)(R)$ is functorial.
Let $N(Fib(\loc (X)^{\ssim}))$ be the corresponding simplicial presheaf obtained by taking nerve in the objectwise manner. This is a fibrant model of $K(G^\red,1)$ in $\sprklp$ ($G=\pi_1(X)$). We have a natural morphism 
\[
p':X\longrightarrow N(Fib(\loc (X)^{\ssim}))
\]
such that $\pi_1(p'):\pi_1(X)\longrightarrow \pi_1N(Fib(\loc (X)^{\ssim}))$ is isomorphic to $G\longrightarrow G^{\red}$. We set   
\[
\widetilde X(R)_n=\{(\sigma,l)|\sigma\in X(R)_n,l\in Fib(\loc (X)^{\ssim})(\omega_{\sigma (0)},\omega_*)\}
\]
where $\omega_x:\loc (X)^\ssim\to R-\proj$ denotes the $R$-fiber functor comes from $x\in X(R)_0$. Note that $\widetilde X$ has a $G^{\red}$-action given by $g\cdot (\sigma, l)=(\sigma, g\circ l)$. We shall show $\ared(\widetilde X)$ is isomorphic to $\ared(X)$. $\ring(G^{\red})$ is considered as a local system on $X$ by
\[
X(R)_n\ni\sigma\longmapsto \ring(\underline{\homo}_{Fib(\loc (X)^{\ssim})}(\omega_{\sigma(0)},\omega_*)).
\]
Here, $\underline{\homo}_{Fib(\loc (X)^{\ssim})}(-,-)$ denotes the sheaf of natural transformations which is represented by an affine scheme over $R$. Note that if $X$ is cofibrant, $\widetilde X$ is also cofibrant.\\
A map $\phi:\ared(X)\longrightarrow \ared(\widetilde X)$ is defined as follows. An element of $\ared(X)$ is presented as $\sum_i\alpha_i\otimes f_i$, $\alpha_i\in\nabla (n,*)$, $f_i\in\ring(\underline{\homo}_{Fib(\loc (X)^{\ssim})}(\omega_*,\omega_{\sigma(0)}))$ locally on $\sigma\in X(R)_n$. 
We define an element $\phi(\sum_i\alpha_i\otimes f_i)$ to be $\sum_i\alpha_i\otimes l(f_i)\in\nabla(n,*)\otimes R$ on $(\sigma, l)$. Here, $l$ is considered as a map $\ring(\underline{\homo}_{Fib(\loc (X)^{\ssim})}(\omega_{\sigma(0)},\omega_*))\longrightarrow R$.\\
A map $\varphi:\ared(\widetilde X)\longrightarrow \ared(X)$ is defined as follows. $\ared(\widetilde X)$ has an $G^{\red}$-action induced from the action on $\widetilde X$. The corresponding comoduled map induce an isomorpism $\rho:\ared(\widetilde X)\cong(\ring (G^{\red})\otimes\ared(\widetilde X))^{G^{\red}}$. Via this isomorphism, an element of $\ared(\widetilde X)$ are presented as $\sum_jv_j\otimes a_j$, $v_j\in \ring(G^{\red})$, $a_j\in\ared(\widetilde X)$. We define an element $\varphi(\sum_jv_j\otimes a_j)$ to be $\sum_ja_j(\sigma,l)\otimes_R[(l\circ\iota\otimes\id)\circ \mu^*(v_j)]$ locally on $\sigma$ ($l\in Fib(\loc (X)^{\ssim})(\omega_{\sigma (0)},\omega_*)$ is arbitrarily fixed. ) Here,
\[
\mu^*:\ring(G^{\red})\cong \ring(\underline{\homo}(\omega_*,\omega_*))\longrightarrow \ring(\underline{\homo}(\omega_*,\omega_{\sigma(0)}))\otimes_R\ring(\underline{\homo}(\omega_{\sigma(0)},\omega_*))
\]
is the cocomposition and 
\[
\iota :\ring(\underline{\homo}(\omega_*,\omega_{\sigma(0)}))\longrightarrow\ring(\underline{\homo}(\omega_{\sigma(0)},\omega_*))
\] is the coinverse map. $\phi$ and $\varphi$ are morphisms of $G^{\red}$-equivariant dga and inverse to each other. Thus, $\ared(\widetilde X)\cong\ared(X)$.  Consider the case that $G$ is reductive and $X$ is a fibrant model of $K(G,1)$ in $\sprklp$. Clearly $F$ is contractible so $\ared(X)\simeq (G,k)$. and $\L\tdr(K)\cong \Tss(G^{\red},k)$ ($\cong \rep(G^\red)$).
Then, 
\[
\pi_n^{pre}(\R\langle\tdr(K)\rangle)(R)\cong[h_R\wedge S^n, \R\langle\tdr(K)\rangle]_{\sprkobjp}
\cong[\Tss(G,k),\tdr(h_R\wedge S^n)]_{\acldgc}
\]
Here, $h_R\wedge S^n\in\sprkp$ be an object defined by the following pushout diagram
\[
\xymatrix{h_R\ar[r]\ar[d]^{\id\times v}&{*} \ar[d]\\
h_R\times S^n\ar[r]&h_R\wedge S^n,}
\]
where $v$ is the base point of $S^n$. If $n\geq 2$, this is clearly $*$. If  $n=1$, The last term is bijective to the set 
\[
\hom_{\clcat/R-\proj}(\Tss(G,k)^c,\rep_R(\Z))/\sim
\]
Here $\rep_R(\Z)$ is the closed $k$-category defined as follows. 
\begin{enumerate}
\item An object of $\rep_R(\Z)$ is a pair $(V,r)$ of $V\in\vect$ and $R$-linear representation $r:\Z\to \gl_R(V\otimes_k R)$ of $\Z$, and
\item a morphism $f:(V_0,r_0)\to (V_1,r_1)$ is a morphism $f:V_0\to V_1$ of $\vect$ such that $f\otimes_k R$ is compatible with the actions.
\end{enumerate}
The augmentation $\omega_{\rep_R(\Z)}:\rep_R(\Z)\to R-\proj$ is given by the forgetful functor. $\sim$ is an equivalence relation such that $F_1\sim F_2$ if and only if there is a natural isomorphism $i:F_1\Rightarrow F_2$ such that $i$ preserves tensors and $(\omega_{\rep_R(\Z)})_*i$ is the identity. This set is bijective to $\aut^{\otimes}(\omega_{\Tss(G,k)^c}:\Tss(G,k)^c\to R-\proj)$ in the way that an action of $1\in \Z$ corresponds a natural automorphism. Thus we have $\pi_1^{pre}(\R\langle\tdr(K)\rangle)\cong\autbar^{\otimes}(\omega)\cong G$ and obtain (1) (see Appendix \ref{proobj}). \\
\indent By (1), Thm.\ref{thmbasic} and Lem.\ref{lempushout}, there is a commutative diagram in $\cldgck$:
\[
\xymatrix{
\Tss(G,k)\ar[r]\ar[d]&\Tss(G,\ared X)\ar[r]\ar[d]&\Tss(e,\ared X)\ar[d] \\
\tdr(K)^{\ssim}\ar[r]&\tdr(X)^{\ssim}\ar[r]&\tdr(QF)^{\ssim}}
\]
such that  the left and middle vertical arrows are quasi-equivalences. By the above assertion, the right vertical arrow is also a quasi-equivalence so the bottom horizontal sequence is a homotopy cofiber sequence and this imply the claim of (2).  
 
\end{proof}

\begin{proof}[Proof of Thm.\ref{thmsht}]
 Prop.\ref{lemiroiro} (5), Lem.\ref{lemfiberseq} and the long exact sequence of homotopy sheaves 
show the unit 
\[
u_X: X\longrightarrow\R\langle\L\tdr(X)\rangle
\]
is a local equivalence for $X\in\shtp$ Consider the commutative diagram in $\ho(\sprkobjp)$:
\[
\xymatrix{ 
\R\langle\L\tdr(X)\rangle\ar[d]^{i}\ar[r]^{u\qquad}&\R\langle\L\tdr(\R\langle\L\tdr(X)\rangle)\rangle\ar[d]^{\R\langle\L\tdr(i)\rangle}\\
R_{loc}\R\langle\L\tdr(X)\rangle \ar[r]^{u'\qquad}&\R\langle\L\tdr( R_{loc}\R\langle\L\tdr(X)\rangle)\rangle,}
\] 
where $u=u_{\R\langle\L\tdr(X)\rangle}$, $u'=u_{R_{loc}\R\langle\L\tdr(X)\rangle}$, and $i$ is a fibrant replacement in $\sprklp$. By Prop.\ref{lemiroiro},(3), $\R\langle\L\tdr(i)\rangle\circ u$ is an isomorphism so $u'\circ i$ is and we may regard $\R\langle\L\tdr(X)\rangle$ as a retract of $R_{loc}\R\langle\L\tdr(X)\rangle$ in $\ho(\sprkobjp)$. By a characterization of local objects, this imply $\R\langle\L\tdr(X)\rangle$ is also a local object. As a local equivalence between local objects is an objectwise equivalence, we see $u_X$ is an isomorphism. Thus $\L\tdr:\ho(\shtp)\longrightarrow \ho(\tann)$ is fully faithful. By Lem.\ref{lempushout} and \ref{lemfiberseq}, we see a morphism $f:T\to T'\in \tann$ is a quasi-equivalence if and only if $\R\langle f\rangle:\R\langle T'\rangle\to\R\langle T\rangle$ is a weak equivalence so $\L\tdr$ is essentially surjective. 
\end{proof}
The following is a corollary of Thm.\ref{thmsht}.
\begin{cor}[cf. Rem.4.43 of \cite{prid0}]\label{corsht}
Let $X\in\shtp$ be a schematic homotopy type and $\M$ be a minimal model of $\ared(X)$. Let $V^i$ denote the $i$-th indecomposable module of $\M$.  For $i\geq 2$, there exists an isomorphism of affine group schemes : $\pi_i(X)\cong(V^i)^\vee$. For  $i=1$, there is an isomorphism of affine schemes (without group structure) $\ru(\pi_1(X))\cong (V^1)^{\vee}$ Here, we regard $(V^i)^\vee$ as an affine group scheme whose coordinate ring is the polynomial ring on $V^i$. 
\end{cor}
\begin{proof}
Let $i\geq 2$. Then, by Thm.\ref{thmsht} 
\[
\pi_i^{pre}(X)(R)\cong \pi_i^{pre}(\R\langle\L\tdr(X)\rangle)(R)\cong [h_R\wedge S^i,\R\langle\L\tdr(X)\rangle]_{\sprkobjp}\cong[\L\tdr(X),\L\tdr(h_R\wedge S^i)]_{\acldgc},
\]
see the proof of Lem.\ref{lemfiberseq} for the notation. If we define $A^R_i\in\dgalg/k$ by $(A^R_i)^0=k, (A^R_i)^i=R,$ and $(A^R_i)^j=0$ for $j\not=0,i$, we easily see $\T(e, A^R_i)\simeq \L\tdr(h_R\wedge S^i)$ so by Lem.\ref{lemmor}, $\pi_i^{pre}(X)(R)\cong[\M,A^R_i]_{\dgalg/k}\cong \hom_{k-\Mod}(V^i,R)$ and claim for $i\geq 2$ follows. The case of $i=1$ follows from a similar argument using the homotopy fiber $F$ of the map $X\to K(\pi_1(X)^\red,1)$. Note that $\pi_1(F)\cong\ru(\pi_1(X))$ and $\ared(F)\simeq\ared(X)$ (see the proof of Lem.\ref{lemfiberseq}).  
\end{proof}
We obtain the unpointed version of Thm.\ref{thmsht}
\begin{cor}\label{thmunpointedsht}
\textup{(1)} The derived adjunction $(\L\tdr ,\R\langle-\rangle):\ho(\sprkobj)\longrightarrow\ho(\dgcl)^\op$ induces  an equivalence between subcategories:
\[
\xymatrix{
\bb{L}\tdr:\ho(\sht) \ar@<3pt>[r]^{\sim}& \ho(\tann)^{\op}:\R\langle -\rangle.\ar@<3pt>[l]}
\]
\textup{(2)} Any connected  simplicial set $K\in\sset$ admits the schematization $\underline{K}\to (K\otimes k)^\sch$ in the unpointed category. It is realized as the unit of the adjunction
$\underline{K}\longrightarrow \R\langle \tdr(K)\rangle\in\ho(\sprkobj)$.
\end{cor}
\begin{proof}
This is clear from Thm.\ref{thmsht}.
\end{proof}
\appendix
\section{Appendix}

\subsection{Variants of simplicial de Rham theorem }\label{derhamthm}
In this subsection, we state some variants of de Rham theorem for simplicial sets, which is used in sections \ref{derhamhtpy} and \ref{sht}. Their proofs are much similar to the proof in \cite[section 2, section 3]{bg}, so we only indicate  how to modify.

 \subsubsection{Twisted  de Rham theorem} 
 Let $K$ be a simplicial set and $\aloc$ be a local system on $K$. We denote by $\cspl(K;\aloc)$ the normalized  cochain complex of $K$ with $\aloc$ coefficients. An element of the degree $q$-part $\cspl^q(K;\aloc)$ is a function $u$ which assigns each $q$-simplex $\sigma\in K_q$ an element $u(\sigma)\in\aloc(\sigma)$ such that for any degenerate simplex $\sigma$, $u(\sigma)=0$. \\
\indent We have the Stokes map as usual:
\[
\rho_{K,\aloc}:\cdr(K;\aloc)\longrightarrow \cspl(K;\aloc)\in\nngc .
\]
The twisted de Rham theorem is the following 
\begin{prop}\label{propderham}
For a simplicial set $K$ and a local system $\aloc$ on $K$, $\rho_{K,\aloc}$ is a quasi-isomorphisms. 
\end{prop}
\begin{proof}
We have two functors:
\[
\cdr(-;\aloc), \ \cspl(-;\aloc):\sset/K\longrightarrow \nngc
\]
defined by 
\[
(\phi:L\to K)\longmapsto \cdr(L;\phi^*\aloc),\ \cspl(L;\phi^*\aloc).
\]
respectively.
We take $\{\Delta^n\to X|n\geq 0\}$ as models instead of $\{\Delta^n\}$ and apply the argument in \cite[section 2, section 3]{bg}. It is clear $\cdr(-;\aloc)$ and $\cspl(-;\aloc)$ is corepresentable and acyclic on models (in a suitably modified sense).
\end{proof}
\subsubsection{Twisted de Rham theorem for simplicial presheaves}
We use notations defined in subsection \ref{definitions}. 
Let $X\in\sprk$ be a simplicial presheaf and $\aloc $ be a local system on $X$. We define a non-negatively graded cochain complex $\cspl(X;\aloc)$ as follows. An element of $\cspl^q(X;\aloc)$ is a collection $\{u_R\}_{R\in k-\Alg}$ such that $u_R\in\cspl^q(X(R);\aloc_R)$ and for $f:R\to R'\in k-\Alg$ and $\sigma\in X(R)$, $f_*(u_R(\sigma))=u_{R'}(X_f(\sigma ))$. the differential is defined in the component-wise manner.\\
\indent We have a Stokes map
\[
\rho_{X,\aloc}:\cdr(X;\aloc)\longrightarrow \cspl(X;\aloc)\in\nngc.
\]
Let $I=\{h_R\times \partial\Delta^n\to h_R\times \Delta^n|R\in k-\Alg, n\geq 0\}$, where $h_R$ denotes the Yoneda embedding  of $Spec R$.
\begin{prop}\label{propderham2}
For an $I$-cell object $X\in\sprk$ and a local system $\aloc$ on $X$, the Stokes map $\rho_{X,\aloc}$ is a quasi-isomorphism.
\end{prop}
\begin{proof}
The proof is similar to that of Prop.\ref{propderham}. In this case, we take $\{h_R\times \Delta^n\to X\}$ as models. Note that for a local system $\aloc'$ on $h_R\times \Delta^n$, $\cspl(h_R\times \Delta^n,\aloc')\cong \cspl(\Delta^n;\aloc_R|_{\id_R\times\Delta^n})$ so $\cspl(-;\aloc)$ is acyclic on models.
\end{proof}
\subsubsection{de Rham theorem for cubical sets}
Let $\cset$ be the category of cubical sets. An object of $\cset$ consists of a collection of sets $\{K_n\}_{n\geq 0}$, face maps $\partial^{\epsilon}_i:K_n\to K_{n-1}$ ($0\leq i\leq n$, $\epsilon=0, 1$), and degeneracy maps $s_i:K_n\to K_{n+1}$ ($0\leq i\leq n$) which satisfy the standard cubical identities. We denote by $\square ^n\in \cset$ the standard $n$-cube. We regard $\cset $ as a model category with the model structure given in \cite{cisinski,cube}, where trivial fibrations are precisely those which have right lifting property with respect to the maps $\partial \square^n\to\square^n\in\cset$ ($n\geq 0$, of course, $\partial\square^n$ is the boundary of $\square^n$).\\
\indent We denote by $\square (n,*)$ the  dg-algebra of $k$-polynomial forms on $\square ^n$. This is the commutative graded algebra over $k$ freely generated by $t_1,\dots ,t_n$ and $dt_1,\dots, dt_n$ with $\deg t_i=0$, $\deg dt_i=1$. (It is isomorphic to $\nabla(n,*)$.) For each $q\geq 0$, We regard $\square(*,q)$ as a cubical set (or cubical abelian group). Face maps and degeneracy maps are defined by the pullback of corresponding maps between the standard cubes. For example, $\partial _i^{\epsilon}:\square(n,q)\to\square(n-1,q)$ is defined as follows.
\[
\partial_i^{\epsilon}(t_j)=
\left\{
\begin{array}{ll}
t_j&j< i \\
\epsilon&j=i \\
t_{j-1}&j>i
\end{array}
\right.
\] 
We need the following.
\begin{lem}\label{lemcube}
For each $q\geq0$, the unique morphism $\square(*,q)\to *\in\cset$ is a trivial fibration.
\end{lem}
\begin{proof}
We imitate the proof of fibrancy of simplicial abelian groups (see for example, \cite[Lem.3.4]{gj}). We shall show the morphism of the claim has right lifting property with respect to the maps $\partial \square ^n\to \square^n$ ($n\geq 0$). Suppose $2n$ elements $x_i^{\epsilon}\in \square(n-1,q)$, $i=1,\dots,n$, $\epsilon=0,1$ such that $\partial _i^{\epsilon_1}x_j^{\epsilon_2}=\partial_{j-1}x_i^{\epsilon_1}$ for $i<j$, are given. (This is equivalent to giving a map $\partial\square^n\to\square(*,q)\in\cset$.) We use induction. Let $1<l\leq n$. Suppose we have $y\in\square(n,q)$ such that for $l\leq \forall i\leq n$, $\forall \epsilon$, $\partial_i^{\epsilon}y=x_i^{\epsilon}$. (For $l=n$, put $y=(1-t_n)s_n(x_n^0)+t_ns_n(x_n^1)$.) \\
\indent Consider an element $z^{\epsilon}:=x_{l-1}^{\epsilon}-\partial^{\epsilon}_{l-1}y$. Cubical identities imply $\forall i\geq l-1$, $\forall \epsilon'$, $\partial_i^{\epsilon'}z^{\epsilon}=0$. Set $y':=(1-t_{l-1})s_{l-1}z^0+t_{l-1}s_{l-1}z^1+y$. Again cubical identities imply $\partial_i^{\epsilon}y'=x_i^{\epsilon}$ for $l-1\leq \forall i\leq n$, $\forall \epsilon$. Thus we can construct a lifting inductively.
\end{proof}
Let $K\in\cset $. The notion of a local system on $K$ is defined similarly to the case of simplicial sets. For a local system $\aloc$ on $K$, we define the de Rham complex $\cdr(K;\aloc )\in\nngc$ using $\square(*,*)$. Let $\ccube(K;\aloc)$ be the normalized cochain complex of $K$ with $\aloc$-coefficients. Explicitly, an element of $\ccube^q(K;\aloc)$ is a function $u$ which assigns each $q$-cube $\sigma\in K_q$ an element $u(\sigma)\in \aloc(\sigma)$ such that $u(\sigma)=0$ for a degenerate cube $\sigma$. The differential $d:\ccube^q(K;\aloc)\to\ccube^{q+1}(K;\aloc)$ is as usual, given by
\[
du=\sum_{i=1}^{q}(-1)^i[(\partial^0_i)^*u-(\partial^1_i)^*u].
\]
Here, $(\partial^{\epsilon}_i)^*u=\aloc (\partial^{\epsilon}_i)^{-1}\circ u\circ \partial^{\epsilon}_i$.
We have a Stokes map
\[
\rho_{K,\aloc}:\cdr(K;\aloc)\longrightarrow \ccube(K;\aloc)\in\nngc.
\]
\begin{prop}\label{propcube}
For a cubical set $K$ and a local system $\aloc$ on $K$, $\rho_{K,\aloc}$ is a quasi-isomorphism.
\end{prop}
\begin{proof}
The proof is similar to that of Prop.\ref{propderham}. We use Lem.\ref{lemcube} to prove $\cdr(-;\aloc)$ is corepresentable.
\end{proof}
We can define an adjoint pair
\[
\xymatrix{
\tdr:\cset \ar@<3pt>[r]& (\cldgc)^{op}:\langle-\rangle.\ar@<3pt>[l]}
\]
We can see this is a Quillen pair, using Lem.\ref{lemcube} and Prop.\ref{propcube}. This fact is implicitly used in the proof of Lem.\ref{hirschlem}.

\subsection{Tannakian theory}\label{proobj}
In this subsection, we summarize Tannakian theory of \cite{dmos}. It states a duality between affine group schemes and certain closed $k$-categories. The category corresponding to a group scheme is the category of finite dimensional representations of it. We give a characterization of such categories and describe how the group scheme is recovered.\\
\indent  Let $\vect'$ be the category of all finite dimensional $k$-vector spaces and linear maps.
\begin{defi}
A closed $k$-category $T\in \clcat$ (see Def.\ref{defofcldgc}) is said to be  a neutral Tannakian category if it satisfies the following conditions.
\begin{enumerate}
\item $T$ is an Abelian category.

\item $\homo _T(\mathbf{1},\mathbf{1})\cong k$.
\item There exists a morphism $\omega:T\to\vect'\in \clcat$ which is faithful and exact.
\end{enumerate}
\end{defi}
In the above definition, we use $\vect'$ instead of $\vect$ in order to make neutral Tannakian categories stable under equivalences of $\clcat$. (Recall that $\vect$ is a small full subcategory of $\vect'$ such that the natural inclusion $\vect\to \vect'$ is an equivalence, see the paragraph above Def.\ref{defacldgc}.)
\begin{exa}
Let $\gam$ (resp. $G$) be a discrete group (resp. an affine group scheme over $k$). The closed $k$-category of finite dimensional $k$-representations of $\gam$ (resp. $G$) $\rep(\gam)$ (resp. $\rep(G)$) is a neutral Tannakian category.
\end{exa}
Let $k-\Alg$ denote the category of commutative and unital $k$-algebras and $\grp$ denote the category of groups.
\begin{exa}
Let $\underline{G}:\kalg\to\grp$ be a functor. A finite dimensional representation of $\underline{G}$ is a pair of a vector space $V\in \vect$ and a natural transformation $r_{-}:\underline{G}\Longrightarrow \underline{\mathrm{GL}}(V):\kalg \to \grp$, where $\underline{\mathrm{GL}}(V)$ is the functor given by $\kalg\ni R\mapsto \mathrm{GL}_R(V\otimes_kR)\in\grp$. There is an obvious notion of morphisms between finite dimensional representations of $\underline{G}$.  We denote the category of finite dimensional representations of $\underline{G}$ by $\mathrm{Rep}(\underline{G})$ and the forgetful functor $\rep(\underline{G})\to \vect$ by $\omega_G$.  $\mathrm{Rep}(\underline{G})$ has an closed tensor structure such that $\omega_G$ is a morphism of closed $k$-categories. Then $\mathit{Rep}(\underline{G})$ is a neutral Tannakian category. If $\underline{G}$ is represented by an affine group scheme $G$, $\rep(\underline{G})$ is equivalent to the category $\rep(G)$.
\end{exa}
For a neutral Tannakian category $T$ and a morphism $\omega:T\longrightarrow \vect\in \clcat$ which is exact and faithful, we define a functor $\autbar^{\otimes} (\omega):k-\Alg \to\grp$  by 
\[
\kalg\ni R\longmapsto \aut^{\otimes} (\omega\otimes _kR)\in\grp.
\]
Here $\aut ^\otimes(\omega\otimes_kR)$ is the group of tensor preserving natural isomorphisms  $\alpha: \omega\otimes _kR\Rightarrow \omega\otimes _kR: T\longrightarrow R-\mathrm{Mod}$ ($\omega \otimes_kR:T\to R-\mathrm{Mod}$ is the morphism from $T$ to the closed $k$-category of $R$-modules defined by $T\ni t\mapsto \omega (t)\otimes_kR\in R-\mathrm{Mod}$). We define a morphism of closed $k$-categories $\widetilde\omega :T\to \rep(\autbar^{\otimes}(\omega))$ by 
\[
T\ni t\longmapsto (\omega (t),ev(t):\autbar^{\otimes}(\omega )\Longrightarrow \underline{\mathrm{GL}}(\omega(t)))\in \mathrm{Rep}(\aut^{\otimes}(\omega)),
\]
where $ev(t)$ is the evaluation of elements of $\aut^{\otimes}(\omega )$ at $t$.
\begin{thm}[Thm.2.11 of \cite{dmos}] We use the above notations. \\
(1) The functor $\autbar^{\otimes}(\omega )$ is represented by an affine group scheme, which is called the \textup{Tannakian dual of} $(T,\omega )$.\\
(2) The above morphism $\widetilde\omega :T\longrightarrow \rep(\autbar^{\otimes}(\omega))$ is an equivalence of categories.\\
(3) Suppose $(T,\omega )=(\repg ,\omega_G)$ for some affine group scheme $G$. Let $\underline{G}$ be the functor corresponding to $G$. There is a natural isomorphism $\underline{G}\cong \autbar^{\otimes}(\omega)$ defined by $\underline{G}(R)\ni g\mapsto g\cdot \in \aut^{\otimes}(\omega _G\otimes _kR)$.
\end{thm}
\begin{cor}
 The pro-algebraic completion $\gam^\alg$ of a discrete group $\gam$ is isomorphic to the Tannakian dual of $\rep(\gam)$ which is equipped with the natural "evaluation" $\gam\to\aut^\otimes(\omega)=\autbar^\otimes(\omega)(k)$ given by 
 \[
 \gam\ni\gamma\longmapsto [\rep(\gam)\ni(V,r)\mapsto (r(\gamma):V\to V)].
 \]
  The maximal reductive quotient $G^\red$ of an affine group scheme $G$ is isomorphic to the Tannakian dual of the full subcategory of $\rep(G)$ consisting of semi-simple representations. (This is closed under tensors and internal homs.)
\end{cor}

\section*{Acknowledgements}
The author is grateful to Masana Harada for many valuable discussions and comments to improve presentations of the paper. He also thank Daisuke Kishimoto for letting him know the book \cite{grimor}. He thank Akira Kono for constant encouragement. He thank Takuro Mochizuki and the members of Homotopical Algebraic Geometry seminar, especially,  Isamu Iwanari and Hiroyuki Minamoto,  for attention to this work.


\begin{thebibliography}{99}
\bibitem{quiha}D. Quillen, {\em Homotopical algebra}. Lecture Notes in Mathematics, Springer-Verlag, 1967.
\bibitem{hm}G. Hochschild, G. D. Mostow, {\em Pro-affine algebraic groups} Aner. J. Math., 91, (1969) 1127-1140.
\bibitem{mac}S. Mac Lane, {\em Categories for the working mathematician}, Graduate Texts in Mathematics,
 {5}, 
Springer-Verlag, 1971.
\bibitem{sul1}D. Sullivan, {\em Geometric topology. Part I; Localization, periodicity, and Galois symmetry, Rivised virsion}, M.I.T. Press, Cambridge, 1970.
\bibitem{kan}A. K. Bousfield, D. M. Kan, {\em Homotopy limits, completions and localizations}, Lecture Notes in Mathematics, 304, Springer-Verlag, 1972.
\bibitem{bg}A. K. Bousfield, V. K. A. M. Gugenheim, {\em On PL de Rham theory and rational homotopy type}, 
Memoirs of the American Mathematical Society, {179}, American Mathematical Society, 1976.
\bibitem{vs} M.Vigu\'e-Poirrier, D. Sullivan, {\em The homology theory of the closed geodesic problem}, J. Differential Geometry, 11, no.4, (1976) 633-644.
\bibitem{sul}D. Sullivan, {\em Infinitesimal computations in topology}, Inst. Hautes 
\'{E}tudes Sci. Publ. Math., {47}, (1977) 269-331. 
\bibitem{white} G. W. Whitehead, {\em Elements of homotopy theory}, Graduate Texts in Mathematics, 61. Springer-Verlag, New York-Berlin, (1978) xxi+744.
\bibitem{grimor} P. Griffiths, J. Morgan, {\em Rational homotopy theory and differential forms}, Progress in Mathematics, 16. Birkh\"auser, Boston, Mass., (1981) xi+242

\bibitem{dmos}P. Deligne, J. S. Milne, {\em Tannakian categories}, in: A. Dold, B. Eckmann (Eds.),  Hodge cycles, motives, and Shimura
 varieties, Lecture Notes in Mathematics, {900}, Springer-Verlag, 1982, pp. 101-228.
\bibitem{sim}C. Simpson, {\em Higgs bundles and local systems}, Inst. Hautes 
\'{E}tudes Sci. Publ. Math., {75}, (1992)  5-95.
\bibitem{brszc}Brown, Szczarba {\em On the rational homotopy type of function spaces}
\bibitem{hov}M. Hovey, {\em  Model categories}, Mathematical Surveys and Monographs, {63},  
American Mathematical Society, 1999.
\bibitem{gj} P. Goerss, J. Jardine, {\em Simplicial homotopy theory}, Progress in Mathematics, 174. Birkh\"auser Verlag, Basel, (1999) xvi+510.
\bibitem{nsrat}A.G\'{o}mez-Tato, S.Halperin and D.Tanr\'{e}, 
{\em Rational homotopy theory for non-simply connected spaces}, Trans. Amer. Math. Soc., 
{352}, no.4, (2000) 1493-1525.
\bibitem{fht} Y.F\'elix, S.Halperin and J.Thomas, {\em Rational homotopy theory},  Graduate Texts in Mathematics, 205, Springer-Verlag, New York, (2001) xxxiv+535.

\bibitem{blander} B. Blander, {\em Local projective model structures on simplicial presheaves}, $K$-Theory, 24, no. 3, (2001), 283-301.
\bibitem{dhi} D. Dugger, S. Hollander, D. Isaksen, {\em Hypercovers and simplicial presheaves}, Math. Proc. Cambridge Philos. Soc., 136, no.1, (2004), 9-51.  
\bibitem{cisinski} Denis-Charles Cisinski, {\em Les pr\'efaisceaux comme mod\`eles des types d'homotopie}, Ast\'erisque, No.308, (2006), xxiv+390.
\bibitem{cube} J. Jardine, {\em Categorical homotopy theory}, Homology, Homotopy Appl. 8, no.1 (2006) 71-144.


\bibitem{champs} B. To\"en, {\em Champs affines}, Selecta Math. (N.S.), 12, no.1, (2006) 39-135.
\bibitem{hag2} B. To\"en, G. Vezzosi, {\em Homotopical algebraic geometry}, Mem. Amer. Math. Soc., 193, no.902, (2008) x+224. \bibitem{kpt1} L. Katzarkov, T. Pantev, B. T\"{o}en, {\em Schematic homotopy types and non-abelian 
Hodge theory}, Compos. Math., {144}, no.3, (2008) 582-632.
\bibitem{kpt2} L. Katzarkov, T. Pantev, B. To\"en, {\em Algebraic and topological aspects of the schmematization functor}, Compos. Math. 145, no.3, (2009) 633-686.

\bibitem{prid3} J.P. Pridham, {\em The pro-unipotent radical of the pro-algebraic fundamental algebraic fundamental group of a compact K\"ahler manifold}, Ann. Fac. Sci. Toulouse Math. (6) 16, no.1, (2007) 147-178.
\bibitem{prid0} J. P. Pridham, {\em Pro-algebraic homotopy types}, Proc. Lond. Math. Soc. (3), 97, no.2, (2008) 273-338. 

\bibitem{prid1} J. P. Pridham, {\em Non-abelian real Hodge theory for proper varieties}, preprint,  
arXiv:math/0611686. 
\bibitem{prid2} J. P. Pridham, {\em Formality and splitting of real non-abelian mixed Hodge structures}, preprint,  
arXiv:math/0902.0770. 

\bibitem{moriya} S. Moriya, {\em Rational Homotopy Theory and Differential Graded Category}, J. Pure Appl. Algebra, 214, no.4, (2010) 422-439.
\end{thebibliography}
\end{document}